\documentclass[11pt]{article} %{{{
\usepackage{latexsym, amsfonts, amssymb, amsmath, amsthm, mathtools}
\usepackage{graphicx}
\usepackage{xurl}
\usepackage{multirow}
\usepackage{float}
\usepackage{subcaption}
\usepackage[colorlinks=true,
linkcolor=black,
filecolor=black,
urlcolor=black,
citecolor=black,
pdftitle={Chemotaxis models with signal-dependent sensitivity and a logistic-type
source, III: Bifurcation and spike-like solutions},
pdfauthor={Le Chen, Ian Ruau, Wenxian Shen},
pdfkeywords={pitchfork bifurcation, center-manifold reduction, global bifurcation, Keller--Segel systems, signal-dependent sensitivity, fixed-mass chemotaxis},
breaklinks=true,
]{hyperref}
\usepackage[left=1in, right=1in, top=1in, bottom=1in]{geometry}
\usepackage[dvipsnames,svgnames]{xcolor}
\usepackage{array}
\usepackage{colortbl}
\usepackage{pgfplots}
\pgfplotsset{compat=1.18}
\makeatletter
\providecommand\UseTaggingSocket[1]{}
\providecommand\insert@pcolumn{%
\the@toks\the\@tempcnta\relax
\ignorespaces\@sharp\unskip
\the@toks\the\count@\relax}
\makeatother
\usepackage[normalem]{ulem}
\usepackage{caption}
\usepackage[shortlabels]{enumitem}
\usepackage{cancel}

\usepackage[norefs,nocites]{refcheck}

\definecolor{TableRowGray}{gray}{0.9}

\newcommand{\norm}[1]{||#1||}

\usepackage{amsrefs}

\renewcommand{\eprint}[1]{\href{https://arxiv.org/abs/#1}{\nolinkurl{#1}}}

\BibSpec{article}{%
+{}  {\PrintAuthors}                {author}
+{,} { \textit}                     {title}
+{.} { }                            {part}
+{:} { \textit}                     {subtitle}
+{,} { \PrintContributions}         {contribution}
+{.} { \PrintPartials}              {partial}
+{,} { }                            {journal}
+{}  { \textbf}                     {volume}
+{}  { \PrintDatePV}                {date}
+{,} { \issuetext}                  {number}
+{,} { \eprintpages}                {pages}
+{,} { }                            {status}
+{,} { \url}                        {url}
+{,} { \PrintDOI}                   {doi}
+{,} { available at \eprint}        {eprint}
+{}  { \parenthesize}               {language}
+{}  { \PrintTranslation}           {translation}
+{;} { \PrintReprint}               {reprint}
+{.} { }                            {note}
+{.} {}                             {transition}
+{}  {\SentenceSpace \PrintReviews} {review}
}

\newtheorem{theorem}{Theorem}[section] % Theorem is numbered by section

\newtheorem{lemma}{Lemma}[section]

\newtheorem{proposition}{Proposition}[section]

\theoremstyle{remark}
\newtheorem{remark}{Remark}[section]
\theoremstyle{remark}
\newtheorem{example}{Example}[section]

\numberwithin{equation}{section}
\numberwithin{figure}{section}
\numberwithin{table}{section}

\newcommand\RR{\ensuremath{\mathbb{R}}}

\begin{document}

\title{Chemotaxis models with signal-dependent sensitivity and a logistic-type
source, III: Bifurcation and spike-like solutions}
\author{
\begin{tabular}{ccc}
  Le Chen & Ian Ruau & Wenxian Shen \\
  \href{mailto:lzc0090@auburn.edu}{\texttt{lzc0090@auburn.edu}}
  & \href{mailto:iir0001@auburn.edu}{\texttt{iir0001@auburn.edu}}
  & \href{mailto:wenxish@auburn.edu}{\texttt{wenxish@auburn.edu}}
\end{tabular} \bigskip \\
Department of Mathematics and Statistics\\
Auburn University, Auburn, AL 36849, USA}

\date{}
\maketitle

\begin{abstract}
  This is the third part of a series of papers on boundedness, global existence,
  and asymptotic behavior of positive classical solutions to the following
  parabolic--elliptic chemotaxis model,
  \begin{equation}\label{E:main-abstract-eq}\tag{CM}
    \begin{dcases}
      u_t=\Delta u- \chi_0\nabla \cdot\Big(\frac{u^m}{(1+v)^\beta}\nabla v\Big)+au-b u^{1+\alpha}, & x \in \Omega, \\
      0=\Delta v-\mu v+\nu u^\gamma,                                                               & x \in \Omega, \\
      \frac{\partial u}{\partial n}=\frac{\partial v}{\partial n}=0,                               & x \in \partial\Omega,
    \end{dcases}
  \end{equation}
  where $\Omega$ is a bounded smooth domain in $\mathbb{R}^N$, and the
  parameters $\alpha, \gamma, m, \mu, \nu$ are positive, $\chi_0 \in
  \mathbb{R}$, and $a, b, \beta$ are nonnegative. In the first part of this
  series~\cite{chen.ruau.ea:26:chemotaxis*1}, we identified several biologically
  relevant parameter regimes that guarantee the boundedness and global existence
  of positive classical solutions. In the second
  part~\cite{chen.ruau.ea:26:chemotaxis}, we investigated the persistence and
  stabilization of globally defined bounded positive solutions, with particular
  emphasis on quantifying the effects of the signal-dependence exponent $\beta$
  and the sign of the sensitivity coefficient $\chi_0$ on the long-time
  dynamics. In this third part, treating $\chi_0$ as the bifurcation parameter,
  we study bifurcation of positive constant equilibria of (CM) when the space
  dimension $N = 1$. We pay special attention to the effects of $\beta$ and the
  domain size on the bifurcation.

  By applying center manifold reduction, we derive a quantity $\beta_{n_0}$ at a
  simple bifurcation mode $n_0$ of a positive constant solution $(u^*,v^*)$, and
  show that super-critical (resp. sub-critical) pitchfork bifurcation occurs if
  $\beta_{n_0} > 0$ (resp. $\beta_{n_0} < 0$) when $\chi_0$ passes through the
  bifurcation threshold $\chi^*$. We find two interesting ways in which $\beta$
  enters: it shifts the bifurcation \textit{threshold}, namely, larger
  signal-dependence exponent $\beta$ needs proportionately larger $\chi_0$ for
  the constant equilibrium to lose linear stability,  and it also selects
  bifurcation \textit{type}, namely, when $\beta$ is small, the bifurcation may
  be super-critical, but when $\beta$ is sufficiently large, the bifurcation is
  always sub-critical. Biologically, $\beta$ thus acts twice in the same
  direction: it makes patterns harder to trigger, and it makes them more likely
  to appear abruptly rather than to grow continuously from the constant state.
  We also find some interesting effects of the domain size. For example, it is
  found that, in the non-minimal model, the bifurcation mode $n_0$ becomes
  larger and larger when the domain size increases and hence a large domain
  creates more complicated bifurcation patterns. In addition, under suitable
  conditions we prove unbounded continuation of the local bifurcation branches
  and prove that  sequences of stationary solutions with large sensitivity
  develop spike-like solutions. We give some numerical analysis of
  the continuation of local bifurcation and spike-like solutions. Our numerical
  time-evolution experiments suggest unbounded continuation of local bifurcation
  branches and increasingly sharp stationary profiles.
\end{abstract}

\medskip


\medskip


{\hypersetup{linkcolor=black} \tableofcontents}

\section{Introduction}

This paper is the third of a series of papers on the boundedness, global
existence and asymptotic behavior of classical solutions to chemotaxis models of
the following form on a bounded, connected, and smooth domain $\Omega \subset
\mathbb{R}^N$:
\begin{align}\label{E:main-PE0}
  \begin{dcases}
    u_t = \Delta u -  \nabla \cdot (u^m\chi(v)  \nabla v) + au - b u^{1+\alpha}, & x \in \Omega, \\
    0= \Delta v - \mu v + \nu u^\gamma,                                          & x \in \Omega, \\
    \frac{\partial u}{\partial n} = \frac{\partial v}{\partial n} = 0,           & x \in \partial\Omega.
  \end{dcases}
\end{align}
Here, the function $\chi(v)$, known in the literature as the \textit{chemotactic
sensitivity function}, is given by
\begin{equation}\label{E:chi-function-eq}
  \chi(v) = \frac{\chi_0}{(1+ v)^\beta}.
\end{equation}
In equations~\eqref{E:main-PE0} and~\eqref{E:chi-function-eq}, $m, \alpha, \mu,
\nu$, and $\gamma$ are positive constants, $\chi_0$ is a real number, and $a, b$,
and $\beta$ are nonnegative constants.

System~\eqref{E:main-PE0} describes the evolution of a biological species
influenced by a chemical substance or signal produced by the species itself. In
this context, $u(t, x)$ represents the population density of the biological
species, and $v(t, x)$ represents the density of the chemical substance or
signal. When $\beta = 0$, $\chi(v)$ simplifies to the constant chemotactic
sensitivity function $\chi(v) \equiv \chi_0$. When $\beta > 0$, the magnitude of
$\chi(v)$ decreases with the signal concentration. Such signal dependence can
model partial loss of the signal through binding; see~\cite{winkler:10:absence}.
Related models with signal-dependent sensitivity or motility are studied
in~\cite{viglialoro.woolley:18:boundedness, baghaei:24:global,
black.lankeit.ea:20:stabilization}. System~\eqref{E:main-PE0} with $\beta > 0$
is known as a chemotaxis system with signal-dependent sensitivity. The
coefficient $\chi_0 > 0$ indicates that the signal is a \textit{chemoattractant}
signal, while $\chi_0 < 0$ indicates that the signal is \textit{chemorepellent}.
The reaction term $a u - b u^{1+\alpha}$ governs the local dynamics of the
biological species and is known as a
\textit{logistic source}. When $a = b = 0$,~\eqref{E:main-PE0} is a minimal
chemotaxis system. In the literature, a chemotaxis system is also referred to as
a \textit{Keller--Segel system}, named after the pioneering works of E. F.
Keller and L. A. Segel~\cite{keller.segel:70:initiation, keller.segel:71:model,
keller.segel:71:traveling}.

Central problems in studying~\eqref{E:main-PE0} include determining whether
solutions with given initial conditions exist globally and understanding the
asymptotic behavior of globally defined bounded solutions of~\eqref{E:main-PE0}.
For constant sensitivity, global existence, boundedness, and blow-up have been
studied in parabolic--elliptic models; see, e.g.,
\cite{galakhov.salieva.ea:16:on, nagai:01:blowup, nagai.senba:98:global,
hu.tao:17:boundedness, xiang:19:dynamics, tello.winkler:07:chemotaxis}. Herrero,
Medina, and Vel\'azquez~\cite{herrero.medina.ea:97:finite-time} considered a
related parabolic--elliptic reduction. Herrero and
Vel\'azquez~\cite{herrero.velazquez:96:singularity} studied finite-time
singularity patterns in a related two-dimensional parabolic--elliptic reduction.

By contrast, fewer works have focused on the global existence and asymptotic
behavior of positive solutions of~\eqref{E:main-PE0}. The objective of this
series of papers is to investigate the global existence, boundedness, and
asymptotic behavior of positive classical solutions of~\eqref{E:main-PE0}. In
Part~I~\cite{chen.ruau.ea:26:chemotaxis*1}, we studied the boundedness and
global existence of positive classical solutions of~\eqref{E:main-PE0} by
examining the roles of negative chemotactic sensitivity $\chi_0$, the nonlinear
chemotactic mobility $\frac{u^m}{(1+v)^\beta}$, and the logistic-type source
$u(a - b u^\alpha)$. In Part~II~\cite{chen.ruau.ea:26:chemotaxis}, we studied
orbitwise uniform persistence of globally defined, bounded, and positive
solutions  and stability of positive constant solutions under proper conditions
on the parameters. Among other things, it is proved that the relevant positive
constant solution loses stability when $\chi_0$ passes through the corresponding
critical sensitivity (see Proposition~\ref{P:review-prop-2} for a little more
detail).

In this paper, we explore the bifurcation of constant solutions
of~\eqref{E:main-PE0} when the domain is one-dimensional, that is, $\Omega =
(0,L)$ for some $L > 0$, and $\chi(v)$ is given in~\eqref{E:chi-function-eq}. In
this case, the system~\eqref{E:main-PE0} can be written as
\begin{align}\label{E:main-PE}
  \begin{dcases}
    u_t = u_{xx} -  \chi_0\big(\frac{u^m}{(1+v)^\beta} v_x\big)_x + au - b u^{1+\alpha}, & x \in (0,L), \\
    0=  v_{xx} - \mu v + \nu u^\gamma,                                                   & x \in (0,L), \\
    u_x(t,x) = v_x(t,x) = 0,                                                             & x =0,L.
  \end{dcases}
\end{align}
We also study spike-like stationary solutions of~\eqref{E:main-PE} when $\chi_0
\to \infty$.

We point out that there are several works on the bifurcation of positive
constant solutions and the existence of spiky stationary solutions
of~\eqref{E:main-PE0} with various choices of $m, \alpha, \mu, \nu, \gamma$, and
$\chi(v)$. For instance, for quadratic logistic growth, Wang, Yan, and
Gai~\cite{wang.yan.ea:16:qualitative} studied a general one-dimensional
stationary Keller--Segel system, deriving local bifurcation values and positive
branches, global continuation properties, and possible boundary-spike limits
along positive monotone solution sequences. Jin, Wang, and
Zhang~\cite{jin.wang.ea:16:pattern} treated stationary bifurcation on
multidimensional bounded domains under simple-mode and nonresonance hypotheses,
and established a local wave-mode selection principle for stable branches. The
closest predecessor for the present signal-dependent setting is the work of
Kurt, Shen, and Xue~\cite{kurt.shen.ea:24:stability}, which established local
and global bifurcation results for stationary solutions in a chemotaxis system
with singular sensitivity and logistic source and analyzed spike formation along
the bifurcating branches. The reader is referred to the works
\cite{wang.xu:13:spiky, chen.hao.ea:14:stability, kong.ward.ea:24:existence,
kong.wei.ea:23:existence, wang.xu:22:radial, kong.wang:22:stability,
xu:21:existence, li:18:spiky}, for existence and stability results on spiky
solutions in related Keller--Segel models. Earlier studies in the Keller--Segel
literature also treated nonconstant stationary states and pattern formation from
complementary viewpoints; see, for example, Lin, Ni, and
Takagi~\cite{lin.ni.ea:88:large} for large-amplitude stationary solutions and
Kuto, Osaki, Sakurai, and Tsujikawa~\cite{kuto.osaki.ea:12:spatial} for spatial
pattern formation in a chemotaxis-diffusion-growth model. Recently, on compact
metric graphs, including intervals as a special case, Shemtaga, Shen, and
Sukhtaiev~\cite{shemtaga.shen.ea:23:stability} identified stability thresholds
and bifurcation points for logistic Keller--Segel models.

We study bifurcation of positive constant solutions of~\eqref{E:main-PE} with
$m, \alpha, \mu, \nu, \gamma > 0$, $\beta \ge 0$, and $a, b > 0$ or $a = b = 0$.
Hence, our study includes the non-minimal model with logistic type source as
well as the minimal model. We treat $\chi_0$ as the bifurcation parameter, and
study which type of bifurcation occurs when a positive constant solution
of~\eqref{E:main-PE} loses its stability both theoretically and numerically. We
pay special attention to the effects of $\beta$ on the bifurcation of positive
constant solutions of~\eqref{E:main-PE}. Theoretically, we apply the
center-manifold reduction presented by
Henry~\cite[Chapter~6]{henry:81:geometric} to compute the local amplitude
equation and determine the bifurcation type. Crandall--Rabinowitz
theory~\cite{crandall.rabinowitz:71:bifurcation} provides an alternative local
framework, while the global extension of local bifurcation uses Theorem~4.4 of
the quasilinear elliptic bifurcation framework developed by Shi and
Wang~\cite{shi.wang:09:on}. Numerically, we discretize the system in space by a
finite-difference method with Neumann boundary conditions, which produces
discrete analogues of the instability thresholds that converge to their
continuum counterparts as the mesh is refined. The time-evolution experiments
illustrate the predicted local bifurcation directions as well as global
extension of local bifurcation pattern; see
Section~\ref{SS:bifurcation-numerics}.

In the following, we give a brief review on some relevant results obtained in
Parts I and II of this series of papers, and then state the main results of the
current paper.

\subsection{A brief review of relevant results from Parts I and II}

In this subsection, we present some relevant results obtained in Parts I and II
of this series of papers.

\smallskip

First, let us define $C^+([0, L])$ by
\[
  C^+([0,L])\coloneq\{u\in C([0,L])\,|\, \inf_{x\in [0,L]} u(x)>0\}.
\]
Then, by standard arguments, one can show the local existence of classical
solutions of~\eqref{E:main-PE} on a maximal interval $[0, T_{\max})$ with initial
condition $u(0, x;u_0) = u_0(x) \in C^+([0, L])$. For the precise definition of
classical solutions of~\eqref{E:main-PE} or~\eqref{E:main-PE0}, we refer the
reader to~\cite[Definition 1.1]{chen.ruau.ea:26:chemotaxis*1} in Part~I. The
reader is also referred to Part~I \cite{chen.ruau.ea:26:chemotaxis*1} for
biologically relevant parameter regimes that guarantee the boundedness and
global existence of positive classical solutions.

Next, note that when $a = b = 0$, for any positive constant $u^* > 0$,
$(u^*,v^*)$ is a positive constant solution of~\eqref{E:main-PE}, where $v^* =
\frac{\nu}{\mu} \big( u^*\big)^{\gamma}$. Hence,~\eqref{E:main-PE} with $a = b =
0$ has a one-parameter family of positive constant solutions,
$\{(u^*,v^*)\}_{u^*>0}$. When $a, b > 0$,~\eqref{E:main-PE} has a unique
positive constant solution $(u^*, v^*)$ given by
\begin{equation}\label{E:equilibrium}
  (u^*, v^*) = \Big(\big(\frac{a}{b}\big)^{1/\alpha},
  \:\frac{\nu}{\mu}\big(\frac{a}{b}\big)^{\gamma/\alpha}\Big).
\end{equation}

Let $(u^*,v^*)$ be a positive constant solution of~\eqref{E:main-PE}. It is
straightforward to verify that the linearization of~\eqref{E:main-PE} at
$(u^*,v^*)$ is given by
\begin{equation}\label{E:main-linear-PE1}
  \begin{dcases}
    u_t = u_{xx} - \chi_0\frac{(u^*)^m}{(1+v^*)^\beta}\Big(\mu v-\nu \gamma(u^*)^{\gamma-1}u\Big)-a\alpha u, & x\in(0,L), \cr
    0 = v_{xx} - \mu v + \nu \gamma(u^*)^{\gamma-1}u,                                                        & x\in(0,L), \cr
    u_x(t,x)=v_x(t,x)=0,                                                                                     & x=0,L.
  \end{dcases}
\end{equation}
Let $\big\{\lambda_n\big\}_{n=0}^\infty$ denote the eigenvalues of the negative
Neumann Laplacian:
\begin{align*}
  \begin{dcases}
    - u_{xx}                    = \lambda u, & x\in (0,L), \cr
    \frac{\partial u}{\partial n} = 0,       & x=0,L
  \end{dcases}
\end{align*}
with $0 = \lambda_0 < \lambda_1 < \lambda_2 < \cdots$. Note that $\lambda_n =
\frac{n^2\pi^2}{L^2}$. Let
\begin{equation}\label{E:sigma_n}
  \sigma_n =
  \sigma_n(\chi_0) \coloneqq
  -\,\lambda_n
  + \chi_0\,\nu\gamma\,
  \frac{(u^*)^{m+\gamma-1}}{(1+v^*)^\beta}\,\frac{\lambda_n}{\mu+\lambda_n}
  - a\alpha.
\end{equation}
It can be verified directly that $\big\{\sigma_n\big\}_{n=0}^\infty$ are the
eigenvalues of the following corresponding eigenvalue problem
of~\eqref{E:main-linear-PE1},
\begin{equation}\label{E:main-linear-PE2}
  \begin{dcases}
    u_{xx}-\chi_0\frac{(u^*)^m}{(1+v^*)^\beta}(\mu v-\nu \gamma(u^*)^{\gamma-1}u)-a\alpha u=\sigma u, & x\in(0,L),\cr
    0 = v_{xx}-\mu v+\nu \gamma(u^*)^{\gamma-1}u,                                                     & x\in(0,L), \cr
    u_x(t,x)=v_x(t,x)=0,                                                                              & x=0,L
  \end{dcases}
\end{equation}
and $\phi_n = \cos(\frac{n\pi x}{L})$ is an eigenfunction corresponding to the
eigenvalue $\sigma_n$. Note that, when $a = b = 0$, $\sigma_0 = 0$, and when $a,
b > 0$, $\sigma_0 < 0$.

For a given positive constant equilibrium $(u^*,v^*)$ of~\eqref{E:main-PE}, let
\begin{equation}\label{E:chi-star}
  \chi^*(u^*)= \chi_{a,b,\beta}^*(u^*) \coloneqq \inf_{n\ge
  1}\Big[\frac{(1+v^*)^\beta}{\nu\gamma\,(u^*)^{m+\gamma-1}} \times
  \frac{(\lambda_n + a\alpha)(\mu+\lambda_n)}{\lambda_n}\Big],
\end{equation}
where $v^* = \frac{\nu}{\mu}(u^*)^\gamma$. When $a, b > 0$, $u^* =
(a/b)^{1/\alpha}$; when $a = b = 0$, $u^* > 0$ is free. Throughout the paper,
$\chi^*(u^*)$ denotes $\chi_{a,b,\beta}^*(u^*)$ in the non-minimal case and
$\chi_{0,0,\beta}^*(u^*)$ in the minimal case.

The following results, expressed in the present notation, were proved
in~\cite{chen.ruau.ea:26:chemotaxis}.

\begin{proposition}[Linear stability/instability]\label{P:review-prop-2}
  Let $(u^*,v^*)$ be a positive constant solution of~\eqref{E:main-PE}. If
  $\chi_0 < \chi^*(u^*)$, then $(u^*, v^*)$ is linearly exponentially stable in
  the case $a, b > 0$, and is linearly exponentially stable with respect to
  perturbations with spatial average $u^*$ in the case $a = b = 0$. If $\chi_0 >
  \chi^*(u^*)$, then $(u^*,v^*)$ is linearly unstable.
\end{proposition}

Proposition~\ref{P:review-prop-2} identifies $\chi^*(u^*)$ as the boundary
between the two strict linear regimes: the constant equilibrium is linearly
stable below this threshold and linearly unstable above it, while the equality
case is neutral at the linear level. It is therefore important to determine
which nonconstant steady states bifurcate from $(u^*,v^*)$ when $\chi_0$ passes
through $\chi^*(u^*)$ and how the parameters in~\eqref{E:main-PE} affect the
bifurcation. In this paper, we study the bifurcation of constant solutions
of~\eqref{E:main-PE} and provide both theoretical analysis and numerical
analysis.

\subsection{Statements of the main results and remarks}

Let $(u^*,v^*)$ be a positive constant solution of~\eqref{E:main-PE}. By
Proposition~\ref{P:review-prop-2}, it is expected that there are nonconstant
steady states of~\eqref{E:main-PE} bifurcating from $(u^*,v^*)$ when $\chi_0$
passes through $\chi^*(u^*)$. To be more precise, let
\begin{equation}\label{E:chi-n-eq-new1}
  \chi_n(u^*)=\frac{(1+v^*)^\beta}{\nu\gamma(u^*)^{m+\gamma-1}}\frac{(\lambda_n+a\alpha)(\mu+\lambda_n)}{\lambda_n},
  \quad n \ge 1.
\end{equation}
By~\eqref{E:chi-star} and~\eqref{E:chi-n-eq-new1}, there is some $n_0 > 0$
independent of $\beta$ and $(u^*,v^*)$ such that
\begin{align}\label{E:bifurcation-mode-eq1}
  \chi^*(u^*)=\chi_{n_0}(u^*)\quad \forall\,\beta\ge 0.
\end{align}
Note that, if $\chi_0 - \chi^*(u^*) < 0$, then $\sigma_n(\chi_0) < 0$ for all $n
\ge 1$, and if $0 < \chi_0 - \chi^*(u^*)\ll 1$, then $\sigma_{n_0}(\chi_0) > 0$ and
$\sigma_n(\chi_0) < 0$ for $n \ge 1$  satisfying that $\chi_{n}(u^*)\not =
\chi^*(u^*)$,  where $\sigma_n(\cdot)$ is as in~\eqref{E:sigma_n}. Therefore,
we call $n_0$ a {\it bifurcation mode} if~\eqref{E:bifurcation-mode-eq1} holds.

Let
\begin{equation}\label{E:L0-Lk-eq}
  L_0:=0 \quad {\rm and}\quad L_k=L_k(a):=\pi\sqrt{k(k+1)}/(a\alpha\mu)^{1/4}\,\,
  {\rm for}\,\,  k\ge1.
\end{equation}
When $a = 0$, we agree on
\[
  L_1=L_1(a):=\infty.
\]
If $a > 0$, then $n_0 = k$ is the unique bifurcation mode when $L \in
(L_{k-1}(a), L_k(a))$, and $n_0 = k, k + 1$ are two bifurcation modes when $L =
L_k(a)$ ($k \ge 1$). If $a = 0$, then $n_0 = 1$ is the unique bifurcation mode
for any $L \in (0, L_1(0)) = (0, \infty)$. Therefore, we say $(L_{k-1}, L_k)$ is
an \textit{admissible interval for the domain size $L$} and we say {\it $L$ is
admissible} if $L \in (L_{k-1}, L_k)$ for some $k \ge 1$. Note that, in the
minimal model, we have only one admissible interval $(0, L_1(0)) = (0, \infty)$
and any $L > 0$ is admissible.

Observe that if $n_0 \ge 1$ is a unique bifurcation mode, then
\begin{align}\label{E:bifurcation-mode-eq2}
  \sigma_{n_0}(\chi^*(u^*))=0\quad {\rm and}\quad \sigma_{n}(\chi^*(u^*))<0
  \quad\forall\, \beta\ge 0,\, \,n\ge 1, n\not =n_0.
\end{align}
In addition, $\sigma_0 = - a\alpha < 0$ in the non-minimal model; in the minimal
model, the neutral constant mode is excluded by the fixed-mass constraint.
Hence, by Proposition~\ref{P:review-prop-2}, when $\chi_0$ passes through
$\chi^*(u^*)$ for any $\beta \ge 0$, the constant solution $(u^*,v^*)$
of~\eqref{E:main-PE} loses its linear stability along the eigen-direction
$\phi_{n_0}$, and it is expected that there are non-constant solutions
bifurcating from $(u^*,v^*)$ along this eigen-direction.

\medskip

Our first main theorem is on the local bifurcation of a positive constant
solution $(u^*, v^*)$ of~\eqref{E:main-PE}, which will be proved in
Section~\ref{S:local-Bifurcation}.

\begin{theorem}[Local pitchfork bifurcation]\label{T:local-bifurcation}
  Let $(u^*,v^*)$ be a positive constant solution of~\eqref{E:main-PE}. Assume
  that $n_0 \ge 1$ is a unique bifurcation mode. Let
  \begin{align}\label{E:alpha-n0}
    \alpha_{n_0}=\alpha_{n_0}(u^*,m,\beta,\gamma,\mu,\nu)
    \coloneqq \partial_{\chi_0}\sigma_{n_0}(\chi^*(u^*))
    = \nu\gamma\,\frac{(u^*)^{m+\gamma-1}}{(1+v^*)^\beta}\,
    \frac{\lambda_{n_0}}{\mu+\lambda_{n_0}}\,(>0).
  \end{align}
  There is a number $\beta_{n_0}$ depending on $u^*$ and all the parameters in
  the model (see~\eqref{E:beta-n0}) such that the following hold.

  If $\beta_{n_0}\not = 0$, pitchfork bifurcation occurs near $(u^*,v^*)$
  in~\eqref{E:main-PE} when $\chi_0$ passes through $\chi^*$. More precisely,
  there are two continuous curves of stationary solutions
  $\Gamma^\pm_{n_0}(\chi_0)$,
  \[
    \Gamma_{n_0}^{\pm}
    \colon I_\varepsilon
    \longrightarrow  C^2(\overline{\Omega}) \times C^2(\overline{\Omega}),
  \]
  bifurcating from $(u^*,v^*)$, where $I_\varepsilon\coloneqq [\chi^*(u^*),
  \chi^*(u^*) + \varepsilon)$ if $\beta_{n_0} > 0$ and $I_\varepsilon: =
  (\chi^*(u^*) - \varepsilon,\chi^*(u^*)]$ if $\beta_{n_0} < 0$ for some
  $\varepsilon > 0$, $\Gamma_{n_0}^\pm(\chi^*) = (u^*,v^*)$,
  $\Gamma_{n_0}^+(\chi_0)\not = \Gamma_{n_0}^-(\chi_0)$ for $\chi_0\not =
  \chi^*$, and the solutions on $\Gamma_{n_0}^{\pm}$ are nonconstant for $\chi_0
  \ne \chi^*(u^*)$.

  Moreover, for $(u^\pm,v^\pm) = \Gamma_{n_0}^\pm(\chi_0)$ ($\chi_0 \in
  I_\varepsilon$), there hold
  \[
    u^\pm=u^*\pm\sqrt{\frac{\alpha_{n_0}}{\beta_{n_0}}(\chi_0-\chi^*) }
    \phi_{n_0}+\mathcal{O}(|\chi_0-\chi^*|),
  \]
  and
  \[
    v^\pm=v^* \pm \frac{\nu \gamma (u^*)^{\gamma-1}}{\lambda_{n_0}+\mu}
    \sqrt{\frac{\alpha_{n_0}}{\beta_{n_0}}(\chi_0-\chi^*)}
    \phi_{n_0}+\mathcal{O}(|\chi_0-\chi^*|).
  \]

  If $\beta_{n_0} > 0$, the bifurcation is super-critical, and for each fixed
  $\chi_0 > \chi^*(u^*)$ sufficiently close to $\chi^*(u^*)$, $(u^\pm,v^\pm) \in
  \Gamma_{n_0}^\pm (\chi_0)\setminus \{(u^*,v^*)\}$ are linearly exponentially
  stable in the case $a, b > 0$, and they are linearly exponentially stable with
  respect to perturbations with spatial average $u^*$ in the case $a = b = 0$.

  If $\beta_{n_0} < 0$, the bifurcation is sub-critical, and for each fixed
  $\chi_0 < \chi^*(u^*)$ sufficiently close to $\chi^*(u^*)$, $(u^\pm,v^\pm) \in
  \Gamma_{n_0}^\pm (\chi_0)\setminus \{(u^*,v^*)\}$ are linearly unstable.
\end{theorem}

Before we state our next result, let us make a remark on
Theorem~\ref{T:local-bifurcation}.

\begin{remark}\label{R:local-bifurcation-minimal}  Let $(u^*,v^*)$
  be a positive constant solution of~\eqref{E:main-PE} and $n_0
  \ge 1$ be a unique bifurcation mode.
  \begin{itemize}
    \item[(1)] $\beta_{n_0}$ stated in Theorem~\ref{T:local-bifurcation} is
      derived explicitly via center manifold theory. Roughly speaking, let
      \[
        u(t,x)=\tilde{u}(t,x)+u^*,\quad v(t,x)=\tilde{v}(t,x)+v^*,
      \]
      and
      \[
        \tilde{u}(t,x) = \sum_{n=0}^\infty u_n(t)\cos\Big(\tfrac{n\pi x}{L}\Big) \quad \text{and} \quad
        \tilde{v}(t,x) = \sum_{n=0}^\infty v_n(t)\cos\Big(\tfrac{n\pi
        x}{L}\Big).
      \]
      Then by center manifold theory, the study of bifurcation solutions
      of~\eqref{E:main-PE} near $(u^*,v^*)$ when $\chi_0$ passes through
      $\chi^*$ is reduced to the study of bifurcation solutions of the following
      scalar ODE near $0$ when $\tilde{\chi} = \chi_0 - \chi^*(u^*)$ passes
      through $0$,
      \[
        \frac{d u_{n_0}}{dt}
        = \alpha_{n_0}\,\tilde{\chi}\,u_{n_0}
        - \beta_{n_0}\,u_{n_0}^3
        + \mathcal{O}\big( |u_{n_0}|^2 |\tilde{\chi}|^2+ |u_{n_0}|^3 |\tilde{\chi}|+
        |u_{n_0}|^4\big)
      \]
      (see~\eqref{E:center-ODE-final}), where $\beta_{n_0}$ is explicitly given
      in~\eqref{E:beta-n0}. It is then seen that the bifurcation is pitchfork as
      long as $\beta_{n_0}\not = 0$, and the sign of $\beta_{n_0}$ determines
      the bifurcation direction or bifurcation type, i.e., the bifurcation is
      super-critical or sub-critical. Although using center-manifold reduction
      to derive this coefficient is standard, it is nontrivial and worth
      deriving the explicit formula for the coefficient $\beta_{n_0}$. Such a
      formula allows us to determine the effects of the parameters on the
      bifurcation (see Theorem~\ref{T:direction-change} about the effects of
      $\beta$ and $L$), which is important biologically.

    \item[(2)] It is seen from Theorem~\ref{T:local-bifurcation} that there are
      non-constant positive stationary solutions $(u,v)$ of~\eqref{E:main-PE}
      near $(u^*,v^*)$ (with the average of $u$ being $u^*$ in the minimal
      model) for $\chi_0$ near $\chi^*(u^*)$. By
      Proposition~\ref{P:review-prop-2} and stable manifold theory
      (see~\cite[Theorem~5.1.1]{henry:81:geometric}), there are no non-constant
      positive stationary solutions $(u,v)$ of~\eqref{E:main-PE} near
      $(u^*,v^*)$ (with the average of $u$ being $u^*$ in the minimal model) for
      $\chi_0$ near $0$. When $\chi_0 = 0$, there are no any non-constant
      positive stationary solution $(u,v)$ of~\eqref{E:main-PE} (with the
      average of $u$ being $u^*$ in the minimal model). This follows from the
      maximum principle when $a, b > 0$, and from $u_{xx} = 0$ with Neumann
      boundary conditions when $a = b = 0$.

  \end{itemize}
\end{remark}

Our next main theorem concerns the effects of $\beta$ and $L$ on the bifurcation
of a positive constant solution $(u^*, v^*)$ of~\eqref{E:main-PE}, which will be
proved in Section~\ref{S:direction-change}.

\begin{theorem}[Effects of $\beta$ and $L$ on the bifurcation]\label{T:direction-change}
  Let $(u^*,v^*)$ be a positive constant solution of~\eqref{E:main-PE}. Assume
  that $n_0 \ge 1$ is a unique bifurcation mode. Let $\beta_{n_0}$ be as in
  Theorem~\ref{T:local-bifurcation}. Then the following hold.

  \begin{itemize}

    \item[(1)] Fix all the parameters except $\beta$. Then
      \begin{equation}\label{E:beta-n0-eq}
        \beta_{n_0}=\beta_{n_0}(\beta)=A\beta^2+B\beta+C
      \end{equation}
      for some constants $A, B, C$ independent of $\beta$ with $A < 0$
      (see~\eqref{E:A-closed-form}). Hence, one of the following occurs when
      $\beta$ changes:
      \begin{itemize}
        \item[(i)] There is $\beta_+ > 0$ such that $\beta_{n_0}(\beta) > 0$ for
          $0 \le \beta < \beta_+$, and $\beta_{n_0}(\beta) < 0$ for $\beta >
          \beta_+$.

        \item[(ii)] There are $0 \le \beta_- < \beta_+$ such that
          $\beta_{n_0}(\beta) < 0$ for $\beta \in [0, \beta_-)$,
          $\beta_{n_0}(\beta) > 0$ for $\beta \in (\beta_-, \beta_+)$, and
          $\beta_{n_0}(\beta) < 0$ for $\beta > \beta_+$.

        \item[(iii)] There is $\beta_0 \ge 0$ such that $\beta_{n_0}(\beta_0) =
          0$ and $\beta_{n_0}(\beta) < 0$ for all $\beta \ge 0$ with $\beta \ne \beta_0$.

        \item[(iv)]  $\beta_{n_0}(\beta) < 0$ for all $\beta \ge 0$.
      \end{itemize}

    \item[(2)] Fix all parameters except $L$. Consider $L \in (L_{k-1}(a),
      L_k(a))$ for some $k \ge 1$. Then $n_0 = k$ and
      \begin{equation}\label{E:beta-n0-eq1}
        \beta_{n_0}
        =\frac{P(\lambda_{n_0})}{Q(\lambda_{n_0})},\quad
        \lambda_{n_0}=\frac{\pi^2 n_0^2}{L^2}.
      \end{equation}
      Here, in the non-minimal model,
      \[
        \lambda_{n_0}\in    \Big(
        \frac{k}{k+1}\sqrt{a\alpha\mu},
        \frac{k}{k-1}\sqrt{a\alpha\mu}
        \Big)
        \quad {\rm if}\,\, k\ge2,
        \qquad
        \lambda_{n_0}\in \Big(\frac12\sqrt{a\alpha\mu},\infty\Big)\quad {\rm
        if}\,\,
        k=1,
      \]
      $P$ is a polynomial of degree at most five, and
      \[
        Q(\lambda)
        =12K^2(\mu+\lambda)^2(4\lambda^2-a\alpha\mu)> 0\quad{\rm for}\,\,
        \lambda >\frac{\sqrt{a\alpha\mu}}{2}\quad {\rm with}\quad
        K=\tfrac{\nu\gamma(u^*)^{m+\gamma-1}}
        {(1+v^*)^\beta}.
      \]
      In the minimal model, $\lambda_{n_0} \in (0, \infty)$, $P$ is a polynomial
      of degree at most three, and
      \[
        Q(\lambda)=24(u^*)^2(\mu+\lambda)^2>0
        \qquad{\rm for}\,\,\, \lambda>0.
      \]

      Moreover, there are constants $c$ and $\beta_\infty$, independent of $L$,
      such that
      \[
        \beta_{n_0}
        =c\frac{\pi^2}{L^2}+\mathcal{O}(1)
        \quad{\rm as}\,\, L\to0+,\quad {\rm and}\,\,
        \quad
        \beta_{n_0}\longrightarrow\beta_\infty
        \quad{\rm as}\,\, L\to\infty\,\,\, {\rm admissible\,\, lengths.}
      \]
      Hence, the following hold.
      \begin{itemize}
        \item[(i)] Suppose $P\not \equiv 0$. When $L$ changes in the admissible
          interval $ (L_{k-1}(a),L_k(a))$, in the non-minimal model, the
          bifurcation type changes at most five times, and in the minimal model,
          it changes at most three times.

        \item[(ii)] If $\beta_\infty > 0$ (respectively, $\beta_\infty < 0$),
          the bifurcation is super-critical (respectively, sub-critical) for
          every sufficiently large admissible length.

        \item[(iii)] If $c > 0$ (respectively, $c < 0$), the bifurcation is
          super-critical (respectively, sub-critical) for all sufficiently small
          $L > 0$.
      \end{itemize}
  \end{itemize}
\end{theorem}

Before stating our next result, we make the following remark on the effects of
$\beta$ and $L$ on the bifurcation of constant solutions.

\begin{remark}
  Let $(u^*,v^*)$ be a positive constant solution of \eqref{E:main-PE}.
  \begin{itemize}
    \item[(1)] Assume that $n_0 \ge 1$ is  a unique bifurcation mode.
      With all parameters except $\beta$ fixed, $\beta_{n_0}(\beta)$ has at most
      two nonnegative zeros. At these values, higher-order terms are needed for
      bifurcation analysis; we do not analyze these cases. Put
      $\chi^*(u^*,\beta) = \chi^*(u^*)$. Theorems \ref{T:local-bifurcation}
      and \ref{T:direction-change}(1)
      reveal the following two effects of $\beta$:
    \item[-] {Threshold:} $\chi^*(u^*,\beta)$ grows proportionally to $(1 +
      v^*)^\beta$ by~\eqref{E:chi-star}, so  large signal-dependence exponent $\beta$
      requires a
      larger $\chi_0$ for non-constant stationary solutions to bifurcate from $(u^*,v^*)$.
    \item[-] {Type:} $\beta_{n_0}(\beta) = A\beta^2 + B\beta + C$ determines
      whether the bifurcation is super-critical (positive) or sub-critical
      (negative). Since $A < 0$, sufficiently large $\beta$ gives sub-critical
      bifurcation, which indicates that large $\beta$ makes stable non-constant
      states more likely to appear abruptly rather than to grow continuously
      from the constant state.

    \item[(2)] Fix all parameters except for $L$. In the non-minimal model, when
      $L = L_k$ ($k \ge 1$), there are two bifurcation modes {$n_0=k$ and $k +
      1$.}  In this case, the positive constant solution $(u^*,v^*)$ loses its
      stability when $\chi_0$ passes through $\chi^*$ along two eigen-directions
      $\phi_{k}$ and $\phi_{k+1}$. We do not investigate bifurcation solutions
      in this case in this paper. With the other parameters fixed except $L$,
      Theorem~\ref{T:direction-change}(2) reveals the following effects of $L$:

    \item[-] In the non-minimal model, when $L$ increases along admissible lengths, the bifurcation
      threshold $\chi^*$ stays bounded, but bifurcation mode $n_0$ becomes
      larger and larger. Hence, a larger domain creates more complicated
      patterns.

    \item[-] The sign of the coefficient $c$ in
      Theorem~\ref{T:direction-change}(2) determines the bifurcation type for $0
      < L\ll 1$. In the non-minimal model, by some direct calculation, we have
      \begin{equation}\label{E:c-constant}
        c=\frac{\mathsf{P}-6\beta\gamma Dv^*/(1+v^*)}{24(u^*)^2}
        \qquad \text{and} \qquad \mathsf{P}=h(6\alpha-4\gamma-m+13),
      \end{equation}
      where $h = \gamma + m - 1$ and $D = \alpha + 2 - \gamma$. Suppose $D > 0$
      and $\mathsf{P} \ge 0$. Let $\beta_{\mathrm{sd}} =
      \frac{\mathsf{P}(1+v^*)}{6\gamma Dv^*}$.
      Then
      \begin{equation}\label{E:c-sd-factored}
        c=\frac{\gamma Dv^*}{4(u^*)^2(1+v^*)}(\beta_{\mathrm{sd}}-\beta).
      \end{equation}
      Thus, as $L \to 0+$, $\beta_{n_0} \to + \infty$ for $0 \le \beta <
      \beta_{\mathrm{sd}}$, so the small-domain bifurcation is super-critical,
      and $\beta_{n_0} \to - \infty$ for $\beta > \beta_{\mathrm{sd}}$, so the
      small-domain bifurcation is sub-critical. If instead $D > 0$ and
      $\mathsf{P} < 0$, then $c < 0$ for every $\beta \ge 0$, so the
      small-domain bifurcation is sub-critical.

    \item[-] The sign of $\beta_\infty$ determines the bifurcation type for
      large admissible $L$. For the non-minimal model, as $L \to \infty$, we
      have
      \[
        \beta_\infty
        =A_\infty\beta^2+B_\infty\beta+C_\infty,
        \qquad A_\infty<0,
      \]
      where the coefficients are independent of $\beta$. Indeed, this follows by
      evaluating the quadratic formula for $\beta_{n_0}(\beta)$ at $\lambda =
      \sqrt{a\alpha\mu}$. It follows that for relatively large $\beta$, the
      bifurcation type for large admissible $L$ is always sub-critical. But for
      relatively small $\beta$, the bifurcation type for large admissible $L$
      may be super-critical if $C_\infty > 0$.
  \end{itemize}
\end{remark}

We remark that the local bifurcation branches $\Gamma_{n_0}^\pm$ established in
Theorem~\ref{T:local-bifurcation} can be extended for $\chi_0$ not near $\chi^*$
by applying global bifurcation theory in the literature, for example,
\cite[Theorem~4.4]{shi.wang:09:on}. To be more precise, let
\begin{equation}\label{E:X-eq}
  X \coloneqq
  \begin{cases}
    \Big\{
    u\in H^2((0,L)) \,\Big|\,
    u_x(0)=u_x(L)=0
    \Big\},
    \qquad        & {\rm if}\quad a,b>0\cr\cr
    \Big\{
    u\in H^2((0,L)) \,\Big|\,
    u_x(0)=u_x(L)=0,\,\,\, \int_0^L u(x)dx=0
    \Big\},\qquad & {\rm if}\quad a=b=0
  \end{cases}
\end{equation}
with norm $\|u\| = \|u\|_{H^2((0,L))}$. For a given positive constant solution
$(u^*,v^*)$ of~\eqref{E:main-PE}, let
\begin{equation}\label{E:A-eq}
  \mathcal{A}=\mathcal{A}(u^*):=
  \{w+u^*\,|\, w\in X,\,\, w(x)+u^*>0\quad {\rm for}\quad x\in [0,L]\},
\end{equation}
$n_0$ be a bifurcation mode, and
\[
  W\coloneqq\big\{w\in X:\int_0^L w(x)\phi_{n_0}(x)\,dx=0\big\}.
\]
Since $\phi_{n_0} \in X$ and $\int_0^L\phi_{n_0}^2\,dx = L/2$, this is a closed
complement of $\operatorname{span}\{\phi_{n_0}\}$ in $X$.

Let $\Gamma_{n_0}^{\pm}(\chi_0) = (u^{\pm}(\chi_0),v^{\pm}(\chi_0))$, and
define the punctured local half-curves
\[
  \mathcal{G}_{n_0}^{\pm}
  \coloneqq
  \left\{
  (\chi_0,u^{\pm}(\chi_0)):
  \chi_0\in I_\varepsilon\setminus\{\chi^*\}
  \right\}.
\]
Here $I_\varepsilon$ and $\Gamma_{n_0}^{\pm}$ are as in
Theorem~\ref{T:local-bifurcation}, with $\varepsilon > 0$ chosen sufficiently
small. Equip $(0, \infty)\times\mathcal{A}$ with the topology induced by
$\mathbb{R}\times H^2(0,L)$, and let
\[
  \mathcal{S}
  =
  \left\{
  (\chi_0,u)\in (0,\infty)\times\mathcal{A}:
  u\ne u^*,\
  (u,v)\text{ solves the stationary system for some }v
  \right\}.
\]
Let $\mathcal{C}$ be the connected component of $\overline{\mathcal{S}}$
containing $(\chi^*,u^*)$, where the closure is taken in $\mathbb{R}\times
\mathcal{A}$. By Remark~\ref{R:local-bifurcation-minimal}(2),
\begin{equation}\label{E:closure-S}
  \overline{\mathcal{S}}\subset (0,\infty)\times {\mathcal{A}}.
\end{equation}

Then we have the following theorem, which will be proved in
Section~\ref{S:global-Bifurcation}.
\begin{theorem}[Extension of the local bifurcation branches]\label{T:global-extension}
  Let $(u^*,v^*)$ be a positive constant solution of~\eqref{E:main-PE}. Assume
  that $n_0$ is a unique bifurcation mode, and that the coefficient
  $\beta_{n_0}\not = 0$.

  Then the following hold.
  \begin{itemize}
    \item[(1)]
      Let $\mathcal{C}_{n_0}^{\pm}$ be the connected component of
      $\mathcal{C}\setminus\mathcal{G}_{n_0}^{\mp}$ containing
      $\mathcal{G}_{n_0}^{\pm}$. Each of these sets contains $(\chi^*,u^*)$ and
      is relatively closed in $(0, \infty)\times\mathcal{A}$. Moreover, each of
      the sets $\mathcal{C}_{n_0}^{+}$ and $\mathcal{C}_{n_0}^{-}$ satisfies one
      of the following alternatives:
      \begin{itemize}
        \item[(i)] $\mathcal{C}_{n_0}^{\pm}$ is not compact in $(0,
          \infty)\times\mathcal{A}$;
        \item[(ii)] $\mathcal{C}_{n_0}^{\pm}$ contains another point
          $\bigl(\widehat{\chi}_0, u^*\bigr)$ with $\widehat{\chi}_0 \ne
          \chi^*(u^*)$;
        \item[(iii)] $\mathcal{C}_{n_0}^{\pm}$ contains a point
          $\bigl(\widehat{\chi}_0, u^* + w\bigr)$ with $w \in W\setminus\{0\}$.
      \end{itemize}

    \item[(2)]  Assume $a, b > 0$, $0 < \gamma \le 1 + \alpha$, and ${1\le m\le
      \alpha+2}$. If $\mathcal{C}_{n_0}^{+}$ (respectively,
      $\mathcal{C}_{n_0}^{-}$) is not compact, then
      \begin{align}\label{E:bifurcation-branches}
        \sup\left\{ \chi_0:(\chi_0,u)\in\mathcal{C}_{n_0}^{+}
        \right\}=+\infty\quad {\rm (respectively,} \quad  \sup\left\{
        \chi_0:(\chi_0,u)\in\mathcal{C}_{n_0}^{-} \right\}=+\infty).
      \end{align}

    \item[(3)]  Assume $a = b = 0$, ${1\le m\le 2}$, and $0 < \gamma \le 1$.
      Then neither $\mathcal{C}_1^{+}$ nor $\mathcal{C}_1^{-}$ satisfies
      alternative~(ii) or~(iii) in (1), and hence both sets are not compact
      and~\eqref{E:bifurcation-branches} holds. Moreover, every nonconstant
      solution on $\mathcal{C}_1^+$ (respectively, $\mathcal{C}_1^-$) is
      strictly decreasing (respectively, strictly increasing) on $[0, L]$. For
      any sequence $(\chi_j,u_j) \in \mathcal{C}_1^+$ (respectively,
      $\mathcal{C}_1^-$) with $\chi_j \to \infty$,
      \[
        \int_0^L f(x)u_j(x)\,dx \longrightarrow Lu^*f(0)
        \qquad
        \left(\text{respectively, }Lu^*f(L)\right)
      \]
      for every $f \in C([0, L])$. In particular, $\max_{x\in[0, L]}u_j(x) \to
      \infty$.
  \end{itemize}
\end{theorem}

We remark that Theorem~\ref{T:global-extension}(3) indicates the monotone steady
states in $\mathcal{C}_1^+$ (resp. $\mathcal{C}_1^-$) form a spike as $\chi_0
\to \infty$. In general, we have the following theorem concerning positive
nonconstant stationary solutions of~\eqref{E:main-PE} for large $\chi_0$, which
is proved in Section~\ref{S:spiky-solu}.

\begin{theorem}[Spike-like stationary solutions]\label{T:spiky-solu-thm}
  Let $(u^*,v^*)$ be a positive constant solution of~\eqref{E:main-PE} and
  $(u_n,v_n)$ be positive nonconstant classical stationary solutions
  of~\eqref{E:main-PE} with $\chi_0 = \chi_n \to \infty$ and $\int_0^L u_n\,dx =
  Lu^*$ for every $n$ in the case that $a = b = 0$. If
  \[
    \liminf_{n\to\infty}\max_{[0,L]}u_n>u^*,
  \]
  then there is a subsequence $\{u_{n_k}\}$ and a point $x^* \in [0, L]$ such
  that, for every $\delta > 0$,
  \begin{equation}\label{E:spiky-thm-eq2}
    \max_{[x^*-\delta,x^*+\delta]\cap[0,L]}|(u_{n_k})_x|
    \longrightarrow\infty\quad {\rm as}\quad n_k\to\infty.
  \end{equation}
  Moreover, in the non-minimal model, if $0 < \gamma < 1 + \alpha$, then
  \begin{equation}\label{E:spiky-thm-eq3}
    u_{n_k}^\gamma\rightharpoonup c\,\,\text{ in }\,\, L^q,
    \quad v_{n_k}\to\frac\nu\mu c\,\, \text{ in }\,\, C^1,
    \quad v_{n_k}\rightharpoonup\frac\nu\mu c\,\, \text{ in }\,\, W^{2,q}
  \end{equation}
  for $q = \frac{1+\alpha}{\gamma}$ and some constant $0 \le c \le
  (u^*)^\gamma$.
\end{theorem}

We make the following remark on Theorem~\ref{T:spiky-solu-thm} before stating
our last result.

\begin{remark}\label{R:spike-solution-rk}

  \begin{itemize}
    \item[(1)] Theorem~\ref{T:spiky-solu-thm} applies to a branch from
      Theorem~\ref{T:global-extension}(1) if it extends to arbitrarily large
      $\chi_0$ and $\liminf_{\chi_0\to\infty}\max_{[0,L]}u(\cdot;\chi_0) > u^*$.
      In the non-minimal model, whether there are positive non-constant
      stationary solutions satisfying the conditions in
      Theorem~\ref{T:spiky-solu-thm} remains open. Our numerical experiments
      presented in Section~\ref{SS:bifurcation-numerics} provide numerical
      evidence consistent with such behavior.

    \item[(2)] For the non-minimal model, the spikes in
      Theorem~\ref{T:spiky-solu-thm} cannot carry a Dirac mass for any $\alpha,
      \gamma > 0$. Indeed, the stationary logistic identity $ a\int_0^L u_n\,dx
      = b\int_0^L u_n^{1+\alpha}\,dx$ gives $\int_0^L u_n^{1+\alpha}\,dx <
      L\left(u^*\right)^{1+\alpha}$ by Lemma~\ref{L:spiky-solu-lm1}(1).
      H\"older's inequality therefore yields
      \begin{equation*}%\label{E:no-delta-concentration}
        \int_{(x^*-\delta,x^*+\delta)\cap(0,L)}u_n\,dx
        \le(2\delta)^{\frac{\alpha}{1+\alpha}}\|u_n\|_{L^{1+\alpha}}
        <u^* L^{\frac{1}{1+\alpha}}(2\delta)^{\frac{\alpha}{1+\alpha}}.
      \end{equation*}
      Thus, the mass in shrinking neighborhoods tends to zero uniformly in $n$,
      even though the gradients diverge.

    \item[(3)]   For nonconstant solutions $(\chi_n,u_n) \in
      \mathcal{C}_1^{\pm}$ under the assumptions of
      Theorem~\ref{T:global-extension}(3), $\chi_n \to \infty$ implies $
      \max_{[0,L]}u_n \longrightarrow \infty$. Thus, the liminf condition in
      Theorem~\ref{T:spiky-solu-thm} is automatic for these sequences.
  \end{itemize}
\end{remark}

The theoretical results stated in the above provide several insights into the
bifurcation of positive constant solutions of~\eqref{E:main-PE}, including the
finding of the number $\beta_{n_0}$ whose sign determines the local bifurcation
type, the effects of $\beta$ and $L$ on the bifurcation, the extension of local
bifurcation branches, and existence of spike-like solutions.
Theorem~\ref{T:global-extension}(3) proves endpoint Dirac-mass concentration for
the specified minimal-model first-mode branches. The existence and limiting
behavior of spike-like solutions outside these hypotheses remain questions for
further study. In this paper, we investigate large-sensitivity profiles
numerically.

To carry out numerical analysis of the bifurcation and large-sensitivity
profiles, we discretize the space variable in~\eqref{E:main-PE} by a
finite-difference method on $N$ uniform subintervals, with $N + 1$ grid points,
to obtain a system of $N + 1$ coupled ordinary differential equations
(see~\eqref{E:main-DIS} for the non-minimal model
and~\eqref{E:main-DIS-conservative} for the minimal model). Let $(u^*,v^*)$ be a
positive constant equilibrium of~\eqref{E:main-PE}. For each $N \ge 1$, set $h =
L/N$. Note that the eigenvalues of the discrete Laplacian $-\Delta_h$ are given
by
\begin{equation}\label{E:disc-eigenvalues}
  \lambda_{0,N}^{\rm disc}=0,
  \qquad
  \lambda_{n,N}^{\rm disc}
  =\frac{4}{h^2}\sin^2\left(\frac{n\pi}{2N}\right),
  \qquad 1\le n\le N.
\end{equation}
It is not difficult to see that the eigenvalues of the linearization of the
semidiscrete system~\eqref{E:main-DIS} at $(u^*,v^*)$ are
\begin{equation}\label{E:disc-growth-rates}
  \sigma_{n,N}^{\rm disc}(\chi_0)
  \coloneqq
  -\lambda_{n,N}^{\rm disc}-a\alpha
  +\frac{\chi_0\nu\gamma (u^*)^{m+\gamma-1}}
  {(1+v^*)^\beta}
  \frac{\lambda_{n,N}^{\rm disc}}
  {\mu+\lambda_{n,N}^{\rm disc}},
  \qquad 0\le n\le N.
\end{equation}
In particular, $\sigma_{0,N}^{\rm disc} = - a\alpha$. For each $1 \le n \le N$,
the equation $\sigma_{n,N}^{\rm disc}(\chi_0) = 0$ determines the instability
threshold for that mode. The discrete instability threshold is therefore
\begin{equation}\label{E:disc-threshold}
  \chi_N^{*,{\rm disc}}(u^*,\beta)
  \coloneqq
  \min_{1\le n\le N}
  \frac{(1+v^*)^\beta}{\nu\gamma (u^*)^{m+\gamma-1}}
  \frac{\left(\lambda_{n,N}^{\rm disc}+a\alpha\right)
  \left(\mu+\lambda_{n,N}^{\rm disc}\right)}
  {\lambda_{n,N}^{\rm disc}}.
\end{equation}
In the minimal case, this threshold is understood on the fixed-mass subspace, so
that the neutral constant mode is excluded. The following proposition, proved in
Subsection~\ref{SS:numerical-scheme}, establishes convergence to the continuous
instability threshold.
\begin{proposition}[Consistency of thresholds]\label{P:semidiscrete-threshold} Fix $m, \alpha, \mu, \nu, \gamma > 0$ and
  $\beta \ge 0$, and assume either $a, b > 0$ or $a = b = 0$. Then
  \begin{equation}\label{E:disctere-limit-eq}
    \lim_{N\to\infty}
    \chi_N^{*,{\rm disc}}(u^*,\beta)
    =
    \chi_{a,b,\beta}^*(u^*).
  \end{equation}
\end{proposition}

We carry out several numerical experiments using the above mentioned numerical
scheme. Our experiments suggest that nonconstant profiles persist as $\chi_0$
increases and that their maxima remain separated from the equilibrium value
$u^*$ in the tested regimes (see
Figures~\ref{F:beta2-super-non-minimal}--\ref{F:beta5-sub-minimal}), where $u^*
= (a/b)^{1/\alpha}$ for the non-minimal model and is the prescribed average for
the minimal model.

\smallskip

The rest of the paper is organized as follows. In
Section~\ref{S:local-Bifurcation}, we prove Theorem~\ref{T:local-bifurcation}.
In Section~\ref{S:direction-change}, we prove Theorem~\ref{T:direction-change}.
In Section~\ref{S:global-Bifurcation}, we prove
Theorem~\ref{T:global-extension}. We prove Theorem~\ref{T:spiky-solu-thm} in
Section~\ref{S:spiky-solu}. In Section~\ref{SS:bifurcation-numerics}, we
establish the threshold consistency in
Proposition~\ref{P:semidiscrete-threshold} and present several numerical
experiments.

\section{Local bifurcation of constant solutions  and proof of  Theorem~\ref{T:local-bifurcation}}
\label{S:local-Bifurcation}

In this section, we study bifurcation of positive constant solutions $(u^*,v^*)$
of~\eqref{E:main-PE} with $a, b > 0$ or $a = b = 0$ and $\Omega = (0,L)$. We
prove Theorem~\ref{T:local-bifurcation} in
Subsections~\ref{SS:local-bifurcation-proof-1}--\ref{SS:local-bifurcation-proof-2}.
We begin with a lemma collecting several technical trigonometric identities that
will be used later in the proofs.

\begin{lemma}\label{L:trig-integrals}
  Let $(u_i)_{i\ge 0}$, $(v_j)_{j\ge 0}$, and $(w_\ell)_{\ell\ge 0}$ belong to
  $\ell^1(\mathbb{N}_0)$. We have
  \begin{equation}\label{E:cosine-sine-sine}
    \int_0^L
    \Big(\sum_{i=0}^{\infty}u_i\cos(\tfrac{i\pi x}{L})\Big)\sin(\tfrac{n\pi
    x}{L})\sin(\tfrac{k\pi x}{L})\,dx
    =
    \frac{L}{4}
    \left(
    u_{|n-k|}-u_{n+k}
    +\delta_{nk}u_0
    \right),\qquad n,k\ge 1,
  \end{equation}
  \begin{align}\label{E:two-cosines-two-sines}
    \int_0^L
    &
    \Big(\sum_{i=0}^{\infty}u_i\cos(\tfrac{i\pi
    x}{L})\Big)
    \Big(\sum_{j=0}^{\infty}v_j\cos(\tfrac{j\pi
    x}{L})\Big)\sin(\tfrac{n\pi
    x}{L})\sin(\tfrac{k\pi
    x}{L})\,dx\nonumber \\
    & =\frac{L}{8}
    \Big(
    \sum_{\substack{ i-j =    n-k }} u_i v_j
    + \sum_{\substack{ i-j = -(n-k) }} u_i v_j
    + \sum_{\substack{ i+j =    n-k }} u_i v_j
    + \sum_{\substack{ i+j = -(n-k) }} u_i v_j \notag \\
    &
    - \sum_{\substack{ i-j =    n+k }} u_i v_j
    - \sum_{\substack{ i-j = -(n+k) }} u_i v_j
    - \sum_{\substack{ i+j =    n+k }} u_i v_j
    \Big), \qquad n,k\ge 1,
  \end{align}
  \begin{align}\label{E:cos-cos-cos}
    \int_0^L &
    \Big( \sum_{i=0}^{\infty} u_i \cos\Big(\tfrac{i\pi x}{L}\Big) \Big)
    \Big( \sum_{j=0}^{\infty} v_j \cos\Big(\tfrac{j\pi x}{L}\Big) \Big)
    \cos\Big(\tfrac{k\pi x}{L}\Big)\,dx\nonumber \\
    & =
    \frac{L}{4}\Big(
    \sum_{\substack{ i-j =  k }} u_i v_j
    \;+\;
    \sum_{\substack{ i-j = -k }} u_i v_j
    \;+\;
    \sum_{\substack{ i+j =  k }} u_i v_j
    \Big)
    +\frac{L}{4}\mathbf{1}_{\{k=0\}}u_0v_0, \qquad k\ge0,
  \end{align}
  and
  \begin{align}\label{E:cos-cos-cos-cos}
    \int_{0}^{L}
    & \Big( \sum_{i=0}^{\infty}
    u_i \cos\Big(\tfrac{i\pi
    x}{L}\Big) \Big)
    \Big( \sum_{j=0}^{\infty} v_j \cos\Big(\tfrac{j\pi x}{L}\Big) \Big)
    \Big( \sum_{\ell=0}^{\infty} w_\ell \cos\Big(\tfrac{\ell\pi x}{L}\Big) \Big)
    \cos\Big(\tfrac{k\pi x}{L}\Big)\,dx \\
    = & \frac{L}{8}\Big(
    \sum_{ i+j+\ell =  k } u_i v_j w_\ell
    + \sum_{ i+j-\ell =  k } u_i v_j w_\ell
    + \sum_{ i+j-\ell = -k } u_i v_j w_\ell
    + \sum_{ i-j+\ell =  k } u_i v_j w_\ell  \nonumber \\
    & \quad
    + \sum_{ i-j+\ell = -k } u_i v_j w_\ell
    + \sum_{ i-j-\ell =  k } u_i v_j w_\ell
    + \sum_{ i-j-\ell = -k } u_i v_j w_\ell
    \Big)
    +\frac{L}{8}\mathbf{1}_{\{k=0\}}u_0v_0w_0, \qquad k\ge0.
  \end{align}
  All indices in the sums range over $\mathbb{N}_0$, and a sum whose constraint
  has no solutions is understood to be zero.
\end{lemma}

\begin{proof}
  The identities follow from the product-to-sum formulas and termwise
  integration.
\end{proof}

\subsection{Local pitchfork bifurcation and proof of Theorem~\ref{T:local-bifurcation} for the non-minimal model}
\label{SS:local-bifurcation-proof-1}

In this subsection, we prove Theorem~\ref{T:local-bifurcation} for the
non-minimal model.

\begin{proof}[Proof of Theorem~\ref{T:local-bifurcation}  for the non-minimal model]

  Our main idea is to reduce the bifurcation analysis of the PDE
  system~\eqref{E:main-PE} to a scalar ODE by applying center manifold theory.

  For clarity, we divide the proof into five steps.

  \smallskip

  \noindent\textbf{Step 1.~} In this step, we construct a center manifold
  of~\eqref{E:main-PE} at $(u^*,v^*)$ for $\chi_0$ near $\chi^*(u^*)$.

  \smallskip

  First, we make the following change of variables:
  \[
    u(t,x)=\tilde{u}(t,x)+u^*,\quad v(t,x)=\tilde{v}(t,x)+v^*.
  \]
  Then $(\tilde{u},\tilde{v})$ satisfies
  \begin{align}\label{E:local-bifurcation-eq2}
    \tilde{u}_t
    & = \tilde{u}_{xx}
    - \chi_0\Big[\Big( B_0
    + B_u\tilde{u}
    + B_v\tilde{v}\Big. \nonumber
    + B_{uu}\tilde{u}^2
    + B_{uv}\tilde{u}\tilde{v}
    + B_{vv}\tilde{v}^2 \nonumber \\
    & \qquad +  {\mathcal{O}\big(|\tilde{u}|^3+|\tilde{v}|^3  \big) }
    + \mathcal{O}\big(|\tilde{u}|^4+|\tilde{v}|^4\big)\Big)\tilde{v}_x\Big]_x  -
    b\Big( A'\tilde{u}
    + A''\tilde{u}^2  + A'''\tilde{u}^3
    + \mathcal{O}(|\tilde{u}|^4)\Big),
  \end{align}
  and
  \begin{align}\label{E:local-bifurcation-eq3}
    0 = \tilde{v}_{xx} - \mu\tilde{v} + C'\tilde{u} +  C'' \tilde{u}^2
    + C'''\tilde{u}^3
    + \mathcal{O}(|\tilde{u}|^4),
  \end{align}
  where
  \begin{alignat}{3}
    B_0 & \coloneqq \frac{(u^*)^m}{(1+v^*)^\beta}, & \qquad B_u & \coloneqq
    \frac{m(u^*)^{m-1}}{(1+v^*)^\beta}, & \qquad B_v & \coloneqq
    -\frac{\beta(u^*)^m}{(1+v^*)^{\beta+1}}, \label{E:B0-Bu-Bv-eq} \\[6pt]
    B_{uu} & \coloneqq \frac{m(m-1)(u^*)^{m-2}}{2(1+v^*)^\beta}, & \qquad
    B_{uv} & \coloneqq -\frac{m\beta(u^*)^{m-1}}{(1+v^*)^{\beta+1}}, & \qquad
    B_{vv} & \coloneqq \frac{\beta(\beta+1)(u^*)^m}{2(1+v^*)^{\beta+2}},
    \label{E:Buu-Buv-Bvv-eq}
  \end{alignat}
  \begin{align}\label{E:L1-L2-eq}
    A' & \coloneqq \alpha(u^*)^\alpha,
    & A''  & \coloneqq \tfrac{1}{2}(1+\alpha)\alpha(u^*)^{\alpha-1},
    & A''' & \coloneqq
    \tfrac{1}{6}(1+\alpha)\alpha(\alpha-1)(u^*)^{\alpha-2},
  \end{align}
  and
  \begin{equation}\label{E:C-const}
    C' \coloneqq \nu\gamma(u^*)^{\gamma-1}, \qquad
    C'' \coloneqq \frac{\nu\gamma(\gamma-1)(u^*)^{\gamma-2}}{2}, \qquad
    C''' \coloneqq \frac{\nu\gamma(\gamma-1)(\gamma-2)(u^*)^{\gamma-3}}{6}.
  \end{equation}
  % \textcolor{magenta}{
  %   In both proofs below, undifferentiated Taylor remainders $\mathcal{O}(|\tilde{u}|^4)$, $\mathcal{O}(|\tilde{v}|^4)$ and $\mathcal{O}(|\tilde{u}|^j+|\tilde{v}|^j)$, $j=3,4$, satisfy the indicated pointwise bounds before multiplication by $\tilde{v}_x$ and flux differentiation.
  %   In~\eqref{E:u_n-eq}, however, $\mathcal{O}(|\tilde{u}|^4)$ denotes $R_4=R_4^{(0)}+\partial_xR_4^{(1)}$, where the source and flux remainders satisfy
  %   \[
  %     |R_4^{(0)}|\le C|\tilde{u}|^4,\qquad
  %     |R_4^{(1)}|\le C\left(|\tilde{u}|+|\tilde{v}|\right)^3|\tilde{v}_x|.
  %   \]
  %   Thus each is bounded by $C\left(|\tilde{u}|+|\tilde{v}|+|\tilde{v}_x|\right)^4$, and $R_4^{(1)}=0$ at $x=0,L$ since $\tilde{v}_x=0$ there.
  %   Smallness is measured in $X_\theta$ defined below: $X_\theta\hookrightarrow H^{2\theta}(0,L)\hookrightarrow C^1([0,L])$ for $3/4<\theta<1$, and~\eqref{E:K-operator-eq} gives $\|\tilde{v}\|_{C^1}\le C\|\tilde{u}\|_{L^\infty}\le C\|\tilde{u}\|_{X_\theta}$, uniformly for $\chi_0$ near $\chi^*$.
  % }
  Here and throughout this proof, $\mathcal{O}(|\tilde{u}|^4)$ denotes a
  function-valued remainder $R_4(\tilde{u})(t,x)$ satisfying
  $|R_4(\tilde{u})(t,x)| \le C |\tilde{u}(t,x)|^4$ for all $x \in (0, L)$
  whenever $\|{\tilde{u}(t,\cdot)}\|_{L^\infty(0,L)}$ is sufficiently small; the
  same convention applies to $\mathcal{O}(|\tilde{v}|^4)$, to mixed remainders
  such as $\mathcal{O}(|\tilde{u}|^4 + |\tilde{v}|^4)$, and
  $\mathcal{O}(|\tilde{u}|^3 + |\tilde{v}|^3)$.

  Next, let
  \[
    X=\Big\{w\in H^2(0,L):w_x(0)=w_x(L)=0\Big\},
  \]
  and choose $\theta \in (3/4, 1)$.  Let  $A = I - \partial_{xx}$ on $L^2(0,L)$,
  with domain $X$, and define the fractional power space $X_\theta =
  D(A^\theta)$ with its graph norm. Then $X_\theta$ is embedded into $C^1([0,
  L])$ (see~\cite[Theorem 1.6.1]{henry:81:geometric}). Note that
  \[
    \tilde{v}=\mathcal{K}(\tilde{u}),
  \]
  where $\mathcal{K}$ is defined by
  \begin{equation}\label{E:K-operator-eq}
    \Big(\mu-\partial_{xx}\Big)\mathcal{K}(w)
    =\nu\Big[\Big(u^*+w\Big)^\gamma-(u^*)^\gamma\Big],
    \qquad
    \partial_x\mathcal{K}(w)=0\quad\text{at }x=0,L.
  \end{equation}
  After substituting $\widetilde{v} = \mathcal{K}(\tilde{u})$
  into~\eqref{E:local-bifurcation-eq2}, we get
  \begin{align}\label{E:local-bifurcation-eq2-new}
    \tilde{u}_t
    & = \tilde{u}_{xx}
    - \chi_0\Big[\Big(  B_0+B_u\tilde{u}
    + B_v\mathcal{K}(\tilde{u})\Big. \nonumber
    + B_{uu}\tilde{u}^2
    + B_{uv}\tilde{u}\mathcal{K}(\tilde{u})
    + B_{vv}\big(\mathcal{K}(\tilde{u})\big)^2 \nonumber \\
    & \qquad {+ \mathcal{O}\big(|\tilde{u}|^3+|\mathcal{K}(\tilde{u})|^3\big) }
    +
    \mathcal{O}\big(|\tilde{u}|^4+|\mathcal{K}(\tilde{u})|^4\big)\Big)\Big(\mathcal{K}(\tilde{u})\Big)_x\Big]_x
    \nonumber \\
    & \qquad - b\Big( A'\tilde{u}
    + A'' \tilde{u}^2 + A'''\tilde{u}^3
    + \mathcal{O}(|\tilde{u}|^4)\Big).
  \end{align}

  Now, we decompose $X_\theta$ as follows:
  \begin{align*}
    X_\theta=\tilde{X}_\theta\oplus {\rm span}\big(\phi_{n_0}\big),
  \end{align*}
  that is, for any $u \in X_\theta$, there are unique $\tilde{u} \in \tilde{X}_\theta$ and $u_{n_0} \in \mathbb{R}$ such that
  \begin{align*}
    u=\tilde{u}+u_{n_0}\phi_{n_0},
  \end{align*}
  where $\tilde{X}_\theta \subset X_\theta$, and for any $\tilde{u} \in \tilde{X}_\theta$, $\tilde{u}$ can be expanded as
  \begin{align*}
    \tilde{u}=\sum_{k\not= n_0} u_k\cos(\tfrac{k\pi x}{L}),\,\,
    u_k\in\mathbb{R}.
  \end{align*}
  Let
  \begin{equation}\label{E:tilde-chi}
    \tilde{\chi} \coloneqq \chi_0 - \chi^*(u^*),
  \end{equation}
  and consider~\eqref{E:local-bifurcation-eq2-new} adjoined with the trivial
  equation
  \begin{equation}\label{E:chi-ode}
    \frac{d \tilde{\chi}}{dt}=0.
  \end{equation}
  By applying center manifold theory (e.g. \cite[Theorems~6.1.7
  and~6.2.1]{henry:81:geometric}), there is a $C^4$ map $\tilde{h}:
  \operatorname{span}\{\phi_{n_0}\}\times\mathbb{R} \to \tilde{X}_\theta$ such
  that
  \begin{align*}
    \tilde{h}(0,0)=0, \quad D\tilde{h}(0,0)=0, \quad { \tilde{h}(0,\tilde{\chi})=0\quad \forall\, |\tilde{\chi}|\ll 1,}
  \end{align*}
  and
  \begin{align}\label{E:center-manifold-eq}
    \mathcal{M}=\{\big(\tilde{h}(u_{n_0},\tilde{\chi})+u_{n_0}\phi_{n_0},\tilde{\chi}\big)\,|\, |u_{n_0}|\ll
    1,\,\, |\tilde{\chi}|\ll 1\}
  \end{align}
  is locally invariant
  under~\eqref{E:local-bifurcation-eq2-new}+\eqref{E:chi-ode}, and local
  bifurcating solutions
  of~\eqref{E:local-bifurcation-eq2-new}+\eqref{E:chi-ode} lie in this
  invariant manifold.

  \smallskip

  \noindent\textbf{Step 2.~} For $(\tilde{u}_0,\tilde{\chi}) \in \mathcal{M}$, let
  $(\tilde{u}(t, x;\tilde{u}_0, \tilde{\chi}),\tilde{\chi})$ be the solution
  of~\eqref{E:local-bifurcation-eq2-new}+\eqref{E:chi-ode} with $\tilde{u}(0,
  x;\tilde{u}_0, \tilde{\chi}) = \tilde{u}_0(x)$. We may put $\tilde{u}(t, x) =
  \tilde{u}(t, x;\tilde{u}_0, \tilde{\chi})$ if no confusion occurs. Let $\tilde{v}(t, \cdot) = \mathcal{K}(\tilde{u}(t, \cdot))$. Let
  \begin{equation}\label{E:local-bifurcation-eq4}
    \tilde{u}(t,x) = \sum_{n=0}^\infty u_n(t)\cos\Big(\tfrac{n\pi x}{L}\Big) \quad \text{and} \quad
    \tilde{v}(t,x) = \sum_{n=0}^\infty v_n(t)\cos\Big(\tfrac{n\pi x}{L}\Big).
  \end{equation}
  In this step, we derive the expressions for $v_k(t)$ in terms of $u_n$.

  \smallskip

  Note that
  \[
    \mathcal{M}\subset X\times \RR.
  \]
  Hence, $\{u_n(t)\},\{v_n(t)\},\{\tfrac{n^2\pi^2}{L^2}u_n(t)\},
  \{\tfrac{n^2\pi^2}{L^2}v_n(t)\} \in \ell^2$. Consequently,
  $\{u_n(t)\},\{v_n(t)\} \in \ell^1$, and
  \begin{align*}
    \tilde{u}_x(t,x) =-\sum_{n=0}^\infty \tfrac{n \pi }{L}
    u_n(t)\sin(\tfrac{n\pi x}{L})\qquad\text{and}\qquad
    \tilde{u}_{xx}(t,x) =-\sum_{n=0}^\infty \tfrac{n^2 \pi^2 }{L^2}
    u_n(t)\cos(\tfrac{n\pi x}{L}),
  \end{align*}
  and similarly for $\tilde{v}_x$ and $\tilde{v}_{xx}$.
  From~\eqref{E:local-bifurcation-eq3}, we have
  \begin{align*}
    0 & = -\sum_{n=0}^\infty \frac{n^2\pi^2}{L^2}v_n(t)\cos(\tfrac{n\pi x}{L})
    - \mu\sum_{n=0}^\infty v_n(t)\cos(\tfrac{n\pi x}{L}) + C'\sum_{n=0}^\infty
    u_n(t)\cos(\tfrac{n\pi x}{L}) \\
    &
    \qquad
    +
    C''
    \Big(\sum_{n=0}^\infty
    u_n(t)\cos(\tfrac{n\pi
    x}{L})\Big)^2
    + C'''\Big(\sum_{n=0}^\infty u_n(t)\cos(\tfrac{n\pi x}{L})\Big)^3 +
    \mathcal{O}\big(|\tilde{u}|^4\big).
  \end{align*}
  For $k = 0$, integrating the above equation over $(0, L)$, we have
  \begin{align}\label{E:v0-eq}
    v_0
    &{ =\frac{C'}{\mu}u_0\! +\!\frac{C''}{\mu} \Big(u_0^2+\frac12\sum_{n=1}^\infty
    u_n^2\Big)
    +\frac{C'''}{\mu} \Big( u_0^3+\frac32u_0\sum_{n=1}^\infty\! u_n^2 +\frac34
    \sum_{i,j\ge1}u_i u_j u_{i+j} \Big)
    +\mathcal{O}_{v,0}\left(\norm{u}_{\ell^1}^4\right),}
  \end{align}
  { where $\mathcal{O}_{v.0}(\|u\|_{l^1}^4)=\int_0^L \mathcal{O}(|u|^4)dx$.}
  For $k \ge 1$, we multiply by $\cos\left(k\pi x/L\right)$ and integrate over
  $(0, L)$. By Lemma~\ref{L:trig-integrals},
  \begin{align}\label{E:vk-eq}
    v_k(t)
    = & \;C_k^{(1)}u_k(t)
    +C_k^{(2)}
    \Big(
    \sum_{ i-j =  k } u_i u_j
    + \sum_{ i-j = -k } u_i u_j
    + \sum_{ i+j =  k } u_i u_j
    \Big) \notag \\
    &
    +C_k^{(3)}
    \Big(
    \sum_{ i+j+\ell =  k } u_i u_j u_\ell
    + \sum_{ i+j-\ell =  k } u_i u_j u_\ell
    + \sum_{ i+j-\ell = -k } u_i u_j u_\ell
    + \sum_{ i-j+\ell =  k } u_i u_j u_\ell\notag \\
    &
    + \sum_{ i-j+\ell = -k } u_i u_j u_\ell
    + \sum_{ i-j-\ell =  k } u_i u_j u_\ell
    + \sum_{ i-j-\ell = -k } u_i u_j u_\ell
    \Big)
    +\mathcal{O}_{v,k}\Big(\|u\|_{\ell^1}^4\Big),
  \end{align}
  where
  \begin{equation}\label{E:Ck-restated}
    C_k^{(1)} \coloneqq \frac{C'}{\mu+\lambda_k}, \qquad
    C_k^{(2)} \coloneqq \frac{C''}{2\bigl(\mu+\lambda_k\bigr)}, \qquad
    C_k^{(3)} \coloneqq \frac{C'''}{4\bigl(\mu+\lambda_k\bigr)},\quad
    k\ge0,
  \end{equation}
  { (and $\mathcal{O}_{v,k}(\|u\|_{l^1}^4)=\int_0^L \mathcal{O}(|u|^4)\cos(\frac{k\pi x}{L})dx$.}
  In the above, we write $u = (u_n)_{n\ge 0}$ and set $\|u\|_{\ell^2} \coloneqq
  \Big(\sum_{n=0}^\infty |u_n|^2\Big)^{1/2}$ and $\|u\|_{\ell^1} \coloneqq
  \sum_{n=0}^\infty |u_n|$. Whenever a scalar remainder in the mode equations is
  written as ${ \mathcal{O}_{k,v}}(\|u\|_{\ell^1}^j)$, it is bounded by $C_k(1 +
  |\chi_0|)\|u\|_{\ell^1(\mathbb{N}_0)}^j$ with a constant $C_k$ that may depend
  on the order $j$ and on the mode index. Note that the remainder sequences are
  summable in the mode index, $\sum_k|\mathcal{O}_k(\cdot)| < \infty$. We
  use a subscript to distinguish termwise remiders, that is, for $k\ge0$ the
  reminders $\mathcal{O}_{u,k}(\cdot)$ (resp. $\mathcal{O}_{v,k}$) corresponds to
  the $k$-th mode from the $u$-equation (resp. from the $v$-equation).

  \smallskip

  \noindent\textbf{Step 3.~} In this step, we derive the system of first order
  ODEs satisfied by $\{u_n(t)\}_{n=0}^\infty$.

  \smallskip

  First, by~\eqref{E:local-bifurcation-eq2}, we have
  \begin{align}\label{E:u_n-eq}
    \Big(\sum_{n=0}^\infty u_n(t)\cos(\tfrac{n\pi x}{L})\Big)_t =I_0 + I_1 + I_2
    + I_3 + I_4 + I_5 + I_6
    { + \mathcal{O}(|\tilde{u}|^4)},
  \end{align}
  where
  \begin{align*}
    I_0
    &\coloneqq -\sum_{n=0}^\infty \lambda_n\,u_n(t)\cos(\tfrac{n\pi x}{L}) - {
    a\alpha  \sum_{n=0}^\infty u_n(t)\cos(\tfrac{n\pi x}{L}) },
  \end{align*}
  \begin{align*}
    I_1
    &\coloneqq \chi_0 \Big[ \Big\{B_0
    + B_u\sum_{n=0}^\infty u_n(t)\cos\big(\tfrac{n\pi x}{L}\big)
    + B_v\sum_{n=0}^\infty v_n(t)\cos\big(\tfrac{n\pi x}{L}\big)
    \Big\} \times
    \sum_{n=0}^\infty \tfrac{n\pi}{L} v_n(t)\sin\big(\tfrac{n\pi x}{L}\big)
    \Big]_x,
  \end{align*}
  \begin{align*}
    I_2
    &\coloneqq
    \chi_0
    \Big[
    B_{uu}\Big(\sum_{n=0}^\infty u_n(t)\cos\big(\tfrac{n\pi x}{L}\big)\Big)^2
    \times
    \sum_{n=0}^\infty \tfrac{n\pi}{L} v_n(t)\sin\big(\tfrac{n\pi x}{L}\big)
    \Big]_x,
  \end{align*}
  \begin{align*}
    I_3
    &\coloneqq
    \chi_0
    \Big[
    B_{uv}\big(\sum_{n=0}^\infty u_n(t)\cos(\tfrac{n\pi x}{L})\big)
    \big(\sum_{n=0}^\infty v_n(t)\cos(\tfrac{n\pi x}{L})\big)
    \times \sum_{n=0}^\infty \tfrac{n\pi}{L} v_n(t)\sin(\tfrac{n\pi x}{L})
    \Big]_x,
  \end{align*}
  \begin{align*}
    I_4
    &\coloneqq
    \chi_0
    \Big[
    B_{vv}\big(\sum_{n=0}^\infty v_n(t)\cos(\tfrac{n\pi x}{L})\big)^2
    \times
    \sum_{n=0}^\infty \tfrac{n\pi}{L} v_n(t)\sin(\tfrac{n\pi x}{L})
    \Big]_x,
  \end{align*}
  \begin{align*}
    I_5
    \coloneqq
    - bA''
    \Big(\sum_{n=0}^\infty u_n(t)\cos(\tfrac{n\pi x}{L})\Big)^2,
    \quad
    I_6\coloneqq
    - b A'''
    \Big(\sum_{n=0}^\infty u_n(t)\cos(\tfrac{n\pi x}{L})\Big)^3.
  \end{align*}

  Next, we project~\eqref{E:u_n-eq} onto the $k$th mode to get the ODE
  satisfied by $u_k(t)$. The regularity supplied by the parabolic equation gives
  for $k = 0, 1, 2, \cdots$,
  \begin{gather*}
    \int_0^{{L}}   \Big(\sum_{n=0}^\infty u_n(t)\cos(\tfrac{n\pi x}{L})\Big)_t   \cos\big(\tfrac{k\pi x}{L}\big)dx= \int_0^L u_k'(t)
    \cos^2\big(\tfrac{k\pi x}{L}\big)dx =
    \begin{cases}
      L u_k'(t)            & {\rm if}\quad k=0, \\
      \tfrac{L}{2} u_k'(t) & {\rm if}\quad k\ge 1.
    \end{cases}
  \end{gather*}
  % \textcolor{magenta}{
  %   Projecting $R_4=\mathcal{O}(t,x)$ against $\cos(k\pi x/L)$ and integrating its flux part by parts, with no boundary term by the Neumann conditions, gives $\left|\int_0^L R_4\cos(k\pi x/L)\,dx\right|\le C_k\|u\|_{\ell^1}^4\le C_k'\|\tilde{u}\|_{X_\theta}^4$, using $\|\tilde{v}\|_{C^1}\le C\|\tilde{u}\|_{L^\infty}\le C\|u\|_{\ell^1}$ and $X_\theta\hookrightarrow\ell^1$ for cosine coefficients, which justifies the fourth-order remainder in~\eqref{E:ode-eq0} and below.
  % }
  Hence,
  \begin{equation}\label{E:ode-eq0}
    u'_k(t)\int_0^L \cos^2\big(\tfrac{k\pi x}{L}\big)dx=\sum_{\ell=0}^6
    \mathcal{F}_k^{(\ell)} {+\mathcal{O}_{k,u}(\|u\|_{l^1}^4)}\quad {\rm for}\quad k=0,1,2,\cdots,
  \end{equation}
  where
  \begin{equation}\label{E:Fk-projection}
    \mathcal{F}_k^{(\ell)} = \mathcal{F}_k^{(\ell)}(u,v) \coloneqq \int_0^L
    I_{\ell}(t,x) \cos\big(\tfrac{k\pi x}{L}\big) dx, \quad \ell =
    0,1,\ldots,6,\; k = 0,1,2,\dots,
  \end{equation}
  { and  $\int_0^L\mathcal{O}(|\tilde{u}|^4)\cos\left(\tfrac{k\pi
  x}{L}\right)\,dx$.}

  For $k = 0$, Neumann boundary conditions imply $\int_0^L I_\ell(t, x)\,dx =
  0$, $\ell = 1, 2, 3, 4$, and using the cosine orthogonality and
  Lemma~\ref{L:trig-integrals} on the remaining integrals we get
  \begin{align*}
    \sum_{\ell=0}^{6} \mathcal{F}_0^{(\ell)}(u,v)
    =      & {-aL\alpha u_0(t)}
    - b L A''
    \Big(
    u_0(t)^2 + \frac{1}{2}\sum_{n=1}^\infty u_n(t)^2
    \Big) \\
    & - b L A'''
    \Big(
    \frac{1}{4}u_0(t)^3
    + \frac{3}{4}\sum_{i,j\ge0} u_i(t)u_j(t)u_{i+j}(t)
    \Big).
  \end{align*}
  For $k \ge 1$, we have
  \begin{align*}
    \mathcal{F}_k^{(0)}
    =
    \frac{L}{2}\Big(-\lambda_k + a - b(1+\alpha)(u^*)^\alpha\Big)
    u_k(t){=-{\frac{L}{2}}(\lambda_k+a\alpha)u_k(t),}
  \end{align*}

  \begin{align*}
    \mathcal{F}_k^{(1)}
    & ={\frac{\chi_0\pi^2}{4L}}k^2v_k \left(2B_0+B_u u_0+B_v v_0\right) \\
    & \quad + \frac{\chi_0\pi^2}{4L}k \sum_{n=1}^{\infty}n v_n
    \left[ B_u\left(u_{|n-k|}-u_{n+k}\right) + B_v\left(v_{|n-k|}-v_{n+k}\right)
    \right].
  \end{align*}

  \begin{align*}
    \mathcal{F}_k^{(2)}
    & =
    \chi_0\,\frac{\pi^2}{8L}\,k\sum_{n=1}^\infty
    n\,v_n(t)\,B_{uu}
    \Big(
    \sum_{\substack{ i-j =    n-k }} u_i u_j
    + \sum_{\substack{ i-j = -(n-k) }} u_i u_j \\
    & \quad
    + \sum_{\substack{ i+j =    n-k }} u_i u_j
    + \sum_{\substack{ i+j = -(n-k) }} u_i u_j
    - \sum_{\substack{ i-j =    n+k }} u_i u_j
    - \sum_{\substack{ i-j = -(n+k) }} u_i u_j
    - \sum_{\substack{ i+j =    n+k }} u_i u_j
    \Big),
  \end{align*}
  \begin{align*}
    \mathcal{F}_k^{(3)}
    & =
    \chi_0\,\frac{\pi^2}{8L}\,k\sum_{n=1}^\infty
    n\,v_n(t)\,B_{uv}
    \Big(
    \sum_{ i-j =    n-k } u_i v_j
    + \sum_{ i-j = -(n-k) } u_i v_j
    + \sum_{ i+j =    n-k } u_i v_j \\
    &
    \quad
    + \sum_{ i+j = -(n-k) } u_i v_j
    - \sum_{ i-j =    n+k } u_i v_j
    - \sum_{ i-j = -(n+k) } u_i v_j
    - \sum_{ i+j =    n+k } u_i v_j \Big),
  \end{align*}
  \begin{align*}
    \mathcal{F}_k^{(4)}
    = & \chi_0\,\frac{\pi^2}{8L}\,k\sum_{n=1}^\infty n\,v_n(t)\,B_{vv}
    \Big(
    \sum_{ i-j =    n-k } v_i v_j
    + \sum_{ i-j = -(n-k) } v_i v_j \\
    & \quad
    + \sum_{ i+j =    n-k } v_i v_j
    + \sum_{ i+j = -(n-k) } v_i v_j
    - \sum_{ i-j =    n+k } v_i v_j
    - \sum_{ i-j = -(n+k) } v_i v_j
    - \sum_{ i+j =    n+k } v_i v_j
    \Big),
  \end{align*}
  \begin{align*}
    \mathcal{F}_k^{(5)}
    = -bA''\frac{L}{4}
    \Big(
    \sum_{ i-j =  k } u_i u_j
    + \sum_{ i-j = -k } u_i u_j
    + \sum_{ i+j =  k } u_i u_j
    \Big),
  \end{align*}
  and
  \begin{align*}
    \mathcal{F}_k^{(6)}
    = & -bA'''\frac{L}{8}
    \Big(
    \sum_ { i+j+\ell =  k }u_i u_j u_\ell
    + \sum_ { i+j-\ell =  k }u_i u_j u_\ell
    + \sum_ { i+j-\ell = -k }u_i u_j u_\ell \\
    & + \sum_ { i-j+\ell =  k }u_i u_j u_\ell
    + \sum_ { i-j+\ell = -k }u_i u_j u_\ell
    + \sum_ { i-j-\ell =  k }u_i u_j u_\ell
    + \sum_ { i-j-\ell = -k }u_i u_j u_\ell
    \Big).
  \end{align*}

  By~\eqref{E:ode-eq0} and~\eqref{E:Fk-projection}, we then have
  \begin{subequations}\label{E:mode-ode}
    \begin{equation}\label{E:mode-ode-u0}
      \frac{d u_0}{dt}
      = -a\alpha\,u_0
      - bA''
      \Big( u_0^2 + \frac{1}{2}\sum_{i=1}^\infty u_i^2 \Big)
      + \mathcal{O}_{u,0}\big(\|u\|_{\ell^1}^3\big).
    \end{equation}
    For each $k \ge 1$, we have
    \begin{equation}\label{E:mode-ode-uk}
      \begin{aligned}
        \frac{d u_k}{dt}
        = & \sigma_k(\chi_0)\,u_k
        - \frac{bA''}{2}\,\mathcal{S}^{(2)}_k(u)
        + \chi_0\, \mathcal{G}_k(u,v)
        - \frac{bA'''}{4}\,\mathcal{S}^{(3)}_k(u)
        +\mathcal{O}_{u,k}\big(\|u\|_{\ell^1}^4\big),
      \end{aligned}
    \end{equation}
  \end{subequations}
  where
  \begin{equation}\label{E:Sk-2-eq}
    \mathcal{S}^{(2)}_k(u)
    \coloneqq \sum_{\substack{ i-j =  k }} u_i u_j
    + \sum_{\substack{ i-j = -k }} u_i u_j
    + \sum_{\substack{ i+j =  k }} u_i u_j,
  \end{equation}
  \begin{equation}\label{E:Sk-3-eq}
    \begin{aligned}
      \mathcal{S}^{(3)}_k(u)
      \coloneqq &
      \sum_{ i+j+\ell =  k } u_i u_j u_\ell
      + \sum_{ i+j-\ell =  k } u_i u_j u_\ell
      + \sum_{ i+j-\ell = -k } u_i u_j u_\ell \\
      & + \sum_{ i-j+\ell =  k } u_i u_j u_\ell
      + \sum_{ i-j+\ell = -k } u_i u_j u_\ell
      + \sum_{ i-j-\ell =  k } u_i u_j u_\ell
      + \sum_{ i-j-\ell = -k } u_i u_j u_\ell,
    \end{aligned}
  \end{equation}
  and
  \begin{equation}\label{E:Gk-eq}
    \mathcal{G}_k(u,v)
    \coloneqq \frac{2}{L\, \chi_0}\,\sum_{\ell=1}^{4}
    \mathcal{F}_k^{(\ell)}\big(u,v\big)
    -\lambda_kB_0C_k^{(1)}u_k{= {\mathcal{O}_{k,u}}(\|u\|_{l^1}^2)}.
  \end{equation}

  \noindent\textbf{Step 4:~} In this step, we derive the ODE satisfied by
  $u_{n_0}$.
  \smallskip

  First, note that, on the center manifold $\mathcal{M}$
  (see~\eqref{E:center-manifold-eq}), we have
  \[
    \tilde{h}(u_{n_0},\tilde{\chi})=\big(h_0(u_{n_0},\tilde{\chi}), \cdots,
    h_{n_0-1}(u_{n_0},\tilde{\chi}),h_{n_0+1}(u_{n_0},\tilde{\chi}),\cdots),
  \]
  where
  \[
    h_n(u_{n_0},\tilde{\chi}) =
    a_{n,1}\,u_{n_0}^2
    + a_{n,2}\,\tilde{\chi}\,u_{n_0}
    + \mathcal{O}\big(|\tilde{\chi}|^2 |u_{n_0}| +|\tilde{\chi}| u_{n_0}^2 +
    |u_{n_0}|^3\big),
    \quad n \ne n_0,
  \]
  and
  \[
    \sum_{n\not = n_0}|h_n(u_{n_0},\tilde{\chi})|<\infty\quad\forall |u_{n_0}|\ll
    1 \quad {\rm and}\quad |\tilde{\chi}|\ll 1.
  \]
  Hence, the manifold
  \[
    \widetilde{\mathcal{M}}
    = \Big\{ \big(U,\tilde{\chi}\big) =\big( (u_n)_{n\ge0},\tilde{\chi}\big) \big|
    u_{n_0} = o(1),\; \tilde{\chi} = o(1),\;u_n =
    h_n\big(u_{n_0},\tilde{\chi}\big) \text{ for } n \ne n_0 \Big\}
  \]
  is locally invariant
  under~\eqref{E:mode-ode-u0}+\eqref{E:mode-ode-uk}+\eqref{E:chi-ode} and
  contains all sufficiently small equilibria. For each fixed $\widetilde{\chi}$,
  the restriction of~\eqref{E:mode-ode-u0}+\eqref{E:mode-ode-uk} to
  $\widetilde{\mathcal{M}}$ is the following scalar ODE:
  \begin{align}\label{E:center-ODE-initial}
    \frac{d u_{n_0}}{dt}
    & = \big(\partial_{\chi_0}\sigma_{n_0}(\chi^*)\big)\tilde{\chi}\,u_{n_0}
    - \frac{bA''}{2}\,\mathcal{S}^{(2)}_{n_0}(u)
    + (\chi^*+\tilde{\chi})\,\mathcal{G}_{n_0}(u,v)
    - \frac{bA'''}{4}\,\mathcal{S}^{(3)}_{n_0}(u)
    + \mathcal{O}\big(\|u\|_{\ell^1}^4\big),
  \end{align}
  where $u = \{u_n\}$ with $u_n = h_n(u_{n_0}, \tilde{\chi})$ for $n\not = n_0$.

  \smallskip

  In the following, we find the coefficients of the leading terms
  in~\eqref{E:center-ODE-initial}. Since $h_k(0, \tilde{\chi}) = 0$, the
  coefficient of $\tilde{\chi}^2$ is zero. Hence, $u_k = h_k(u_{n_0}, \tilde{\chi})
  = a_{k,1}u_{n_0}^2 + a_{k,2}\tilde{\chi}u_{n_0} +
  \mathcal{O}\Big(|\tilde{\chi}|^3 + |u_{n_0}|^3\Big)$ for $k \ge 1$, $k \ne
  n_0$. We then obtain from~\eqref{E:Sk-2-eq}, \eqref{E:Sk-3-eq},
  and~\eqref{E:Gk-eq}
  \begin{equation}\label{E:S2-k-eq-new}
    \mathcal{S}^{(2)}_k(u)
    = \begin{cases}
      u_{n_0}^2 + \mathcal{O}\Big(|\tilde{\chi}|\,|u_{n_0}|^2 + |u_{n_0}|^3\Big),
      & k = 2 n_0, \\
      \mathcal{O}\Big(|\tilde{\chi}|\,|u_{n_0}|^2 + |u_{n_0}|^3\Big),
      & k \ne 2 n_0,
    \end{cases}
  \end{equation}
  \begin{equation}\label{E:S3-k-eq-new}
    \mathcal{S}^{(3)}_k(u)=\mathcal{O} \Big(|\tilde{\chi}|^2
    |u_{n_0}|+|\tilde{\chi}| u_{n_0}^2 + |u_{n_0}|^3\Big),
  \end{equation}
  and
  \begin{equation}\label{E:Gk-eq-new}
    \mathcal{G}_k(u,v)
    = \Gamma_k\, u_{n_0}^2 + \mathcal{O}\Big(|\tilde{\chi}|  u_{n_0}^2 +
    |u_{n_0}|^3\Big),
  \end{equation}
  where $\Gamma_k = 0$ for $k \ne 2 n_0$, and
  \begin{equation}\label{E:Gamma2n0-explicit}
    \Gamma_{2n_0}
    = \Big(\frac{n_0\pi}{L}\Big)^2
    \Big[
    4B_0\,C_{2n_0}^{(2)}
    + B_u\,C_{n_0}^{(1)}
    + B_v\,\bigl(C_{n_0}^{(1)}\bigr)^2
    \Big].
  \end{equation}
  Here $B_0$, $B_u$, and $B_v$ are defined in~\eqref{E:B0-Bu-Bv-eq}, and
  $C_{n_0}^{(1)}$ and $C_{2n_0}^{(2)}$ are defined in~\eqref{E:Ck-restated}.

  We compute the coefficients of the leading terms in $h_k$. For $k = 0$,
  differentiating $u_0 = h_0(u_{n_0},\tilde{\chi})$ with respect to time and
  using~\eqref{E:mode-ode-uk} with $k = n_0$,
  by~\eqref{E:S2-k-eq-new}, \eqref{E:S3-k-eq-new}, and~\eqref{E:Gk-eq-new}, we
  have that, for $|\tilde{\chi}| + |u_{n_0}|$ sufficiently small,
  \begin{align}\label{E:du0-center}
    \frac{d u_0}{dt}
    = & \Big(2 a_{0,1} u_{n_0} + a_{0,2} \tilde{\chi}\Big)\frac{d u_{n_0}}{dt}
    + \mathcal{O}\Big(|\tilde{\chi}|^3 + |u_{n_0}|^3\Big) \notag =
    \mathcal{O}\Big(|\tilde{\chi}|^3 + |u_{n_0}|^3\Big).
  \end{align}
  On the other hand, substituting $u_n = h_n(u_{n_0}, \tilde{\chi})$ on the
  right-hand side of~\eqref{E:mode-ode-u0} gives
  \begin{align*}
    \frac{d u_0}{dt}
    = & -a\alpha\,\Big(a_{0,1} u_{n_0}^2 + a_{0,2} \tilde{\chi} u_{n_0}\Big)
    - \frac{bA''}{2} u_{n_0}^2
    + \mathcal{O}\Big(|\tilde{\chi}|^3 + |u_{n_0}|^3\Big).
  \end{align*}
  By comparing the above two ODEs, we see that
  \begin{equation}\label{E:a_0,i}
    {a_{0,1} = -\frac{bA''}{2a\alpha}=-\frac{1+\alpha}{4u^*}},
    \qquad
    a_{0,2}  = 0.
  \end{equation}

  We can carry out the same procedure for $u_k$ with $k \ge 1$ and $k \ne n_0$.
  Differentiating $u_k = h_k(u_{n_0}, \tilde{\chi})$ and
  using~\eqref{E:mode-ode-uk} yields
  \begin{align}\label{E:duk-center}
    \frac{d u_k}{dt}
    = \Big(2 a_{k,1} u_{n_0} + a_{k,2} \tilde{\chi}\Big)\frac{d u_{n_0}}{dt}
    + \mathcal{O}\Big(|\tilde{\chi}|^3 +
    |u_{n_0}|^3\Big)=\mathcal{O}\Big(|\tilde{\chi}|^3 + |u_{n_0}|^3\Big).
  \end{align}
  Substitute $u_n = h_n(u_{n_0}, \tilde{\chi})$ into~\eqref{E:mode-ode-uk}. Using
  $u_k = a_{k,1}u_{n_0}^2 + a_{k,2}\tilde{\chi}u_{n_0} +
  \mathcal{O}(|\tilde{\chi}|^3 + |u_{n_0}|^3)$, we obtain
  \begin{align}\label{E:duk-expanded}
    \frac{d u_k}{dt}
    = & \sigma_k(\chi^*)\Big(a_{k,1} u_{n_0}^2 + a_{k,2} \tilde{\chi} u_{n_0}
    \Big)
    - \frac{bA''}{2}\,\mathcal{S}^{(2)}_k(u) + \chi^*\,\mathcal{G}_k(u,v)
    + \mathcal{O}\Big(|\tilde{\chi}|^3 + |u_{n_0}|^3\Big).
  \end{align}
  By~\eqref{E:S2-k-eq-new}, \eqref{E:S3-k-eq-new}, and~\eqref{E:Gk-eq-new}
  again, we have
  \begin{subequations}\label{E:a_k,i}
    \begin{equation}\label{E:a_k,i-zero}
      a_{k,2} = 0 \quad \text{for all } k \ne n_0 \text{ with } k \ge 1,
    \end{equation}
    and
    \begin{equation}\label{E:a_k,i>0}
      a_{k,1}
      = \begin{cases}
        \displaystyle
        \frac{\frac{bA''}{2} - \chi^*\Gamma_{2n_0}}
        {\sigma_{2n_0}(\chi^*)},
        & k = 2 n_0, \\
        0, & k \ne n_0,\, 2 n_0 \text{ with } k \ge 1.
      \end{cases}
    \end{equation}
  \end{subequations}

  By~\eqref{E:a_0,i} and~\eqref{E:a_k,i-zero}, $a_{k,2} = 0$ for every $k \ne
  n_0$; hence smoothness of the center-manifold map and $h_k(0, \widetilde{\chi})
  = 0$ give
  \begin{align}\label{E:hk-expansion}
    h_k(u_{n_0},\widetilde{\chi})
    &=a_{k,1}u_{n_0}^2
    +\mathcal{O}\left(
    |\widetilde{\chi}|^2|u_{n_0}|
    +|\widetilde{\chi}||u_{n_0}|^2
    +|u_{n_0}|^3
    \right).
  \end{align}
  We then have
  \begin{equation}\label{E:S2-center}
    \mathcal{S}^{(2)}_{n_0}(u)
    = \bigl(4a_{0,1} + 2a_{2n_0,1}\bigr)\,u_{n_0}^3
    + \mathcal{O}\Big({|u_{n_0}|^{{2}} |\tilde{\chi}|^2+ {|u_{n_0}|^3}
    |\tilde{\chi}|+ |u_{n_0}|^4}\Big),
  \end{equation}
  \begin{equation}\label{E:S3-center}
    \mathcal{S}^{(3)}_{n_0}(u)
    = 3\,u_{n_0}^3 + \mathcal{O}\Big({ |u_{n_0}|^{{2}} |\tilde{\chi}|^2+
    {|u_{n_0}|^3} |\tilde{\chi}|+ |u_{n_0}|^4}\Big),
  \end{equation}
  and
  \begin{equation}\label{E:Gn0-cubic}
    \mathcal{G}_{n_0}(u,v)
    = \Gamma_{n_0}^{(3)}\,u_{n_0}^3
    + \mathcal{O}\Big({|u_{n_0}|^{{2}} |\tilde{\chi}|^2+ {|u_{n_0}|^3}
    |\tilde{\chi}|+ |u_{n_0}|^4}\Big),
  \end{equation}
  where the $\Gamma_{n_0}^{(3)}$ coefficient is
  \begin{equation}\label{E:Gamma-n0-explicit}
    \Gamma_{n_0}^{(3)} =\Big(\frac{n_0\pi}{L}\Big)^2\Big[
    B_0H+\Big(B_u+B_vF\Big)G +F\Big(R-\frac{S}{2}\Big) \Big],
  \end{equation}
  $B_0$, $B_u$, $B_v$ are defined in~\eqref{E:B0-Bu-Bv-eq},
  \begin{alignat}{2}\label{E:E-F-G-H}
    {  F\coloneqq C_{n_0}^{(1)}, } \quad
    G  \coloneqq C_{2n_0}^{(1)}a_{2n_0,1}+C_{2n_0}^{(2)},   \quad
    & H \coloneqq 2C_{n_0}^{(2)}\Big(2a_{0,1}+a_{2n_0,1}\Big) +3C_{n_0}^{(3)},
  \end{alignat}
  and
  \begin{align*}
    R \coloneqq B_ua_{0,1}+B_v \big(
    C_0^{(1)}a_{0,1}+C_0^{(2)}\big)+\frac{Q}{2}, \,\,
    S \coloneqq B_u a_{2n_0,1}+B_vG+\frac{Q}{2},
    \,\,
    Q \coloneqq B_{uu}+B_{uv}F+B_{vv}F^2.
  \end{align*}

  Substituting the expansion~\eqref{E:hk-expansion} into the quadratic and cubic
  terms of~\eqref{E:center-ODE-initial}, using~\eqref{E:S2-center},
  \eqref{E:S3-center}, \eqref{E:Gn0-cubic}, and the fact that
  $h_{n}(0,\tilde{\chi}) = 0$ for $|\tilde{\chi}|\ll 1$, we must have
  \begin{equation}\label{E:center-ODE-final}
    \frac{d u_{n_0}}{dt}
    = \alpha_{n_0}\,\tilde{\chi}\,u_{n_0}
    - \beta_{n_0}\,u_{n_0}^3
    + { \mathcal{O}\Big( |u_{n_0}|^{{2}} |\tilde{\chi}|^2+ {|u_{n_0}|^3}
    |\tilde{\chi}|+ |u_{n_0}|^4\Big),}
  \end{equation}
  where
  \begin{align}\label{E:alpha_n0}
    \alpha_{n_0}  & = \partial_{\chi_0}\sigma_{n_0}(\chi^*)
    = \nu\gamma\,\frac{(u^*)^{m+\gamma-1}}{(1+v^*)^\beta}\,
    \frac{\lambda_{n_0}}{\mu+\lambda_{n_0}}
    > 0, \\
    \beta_{n_0}
    & = \frac{(1+\alpha)\alpha\,b\,(u^*)^{\alpha-1}}{4}
    \bigl(4a_{0,1} + 2a_{2n_0,1}\bigr)
    - \chi^*\,\Gamma_{n_0}^{(3)}
    + \frac{(1+\alpha)\alpha(\alpha-1)b\,(u^*)^{\alpha-2}}{8}.
    \label{E:beta-n0}
  \end{align}

  \noindent\textbf{Step 5:~} In this step, we conclude the bifurcation analysis.

  \smallskip

  By~\eqref{E:center-ODE-final}, when $\beta_{n_0} \ne
  0$ and $|\tilde\chi|, |u_{n_0}|\ll 1$, \eqref{E:center-ODE-final} has two branches $u_{n_0}^\pm(\tilde{\chi})$ of
  equilibrium solutions on the side $\beta_{n_0}\widetilde{\chi} > 0$, with
  \begin{align}\label{E:two-branches}
    u_{n_0}^\pm(\widetilde{\chi})
    =\pm\sqrt{\frac{\alpha_{n_0}}{\beta_{n_0}}
    \widetilde{\chi}}+\mathcal{O}(|\widetilde{\chi}|),
  \end{align}
  which gives two branches $u^\pm$ of equilibrium solutions of~\eqref{E:main-PE}
  on the side $\beta_{n_0} (\chi_0 - \chi^*) > 0$ given in
  Theorem~\ref{T:local-bifurcation}. Moreover,
  when $\beta_{n_0} > 0$, the bifurcation occurs when $\chi_0 > \chi^*$ and the
  bifurcation solutions $u^\pm$ are linearly asymptotically stable, which is
  referred to as super-critical pitchfork bifurcation, and when $\beta_{n_0} <
  0$, the bifurcation occurs when $\chi_0 < \chi^*$ and solutions $u^\pm$ are
  linearly unstable, which is referred to as sub-critical pitchfork bifurcation.
\end{proof}

\subsection{Local pitchfork bifurcation and proof of Theorem~\ref{T:local-bifurcation} for the minimal model}
\label{SS:local-bifurcation-proof-2}

In this subsection, we prove Theorem~\ref{T:local-bifurcation} for the minimal
model.

\begin{proof}[Proof of Theorem~\ref{T:local-bifurcation} for the minimal model]
  Throughout this proof, assume $a = b = 0$ and fix the positive constant
  equilibrium $(u^*,v^*)$, where $v^* = \frac{\nu}{\mu}(u^*)^\gamma$. As in the
  non-minimal case, set
  \[
    u(t,x)=u^*+\widetilde{u}(t,x),\qquad
    v(t,x)=v^*+\widetilde{v}(t,x).
  \]
  Then $(\widetilde{u},\widetilde{v})$
  satisfies~\eqref{E:local-bifurcation-eq2}--\eqref{E:local-bifurcation-eq3}
  with $a = b = 0$, and $\widetilde{v} = \mathcal{K}(\widetilde{u})$, where
  $\mathcal{K}$ is defined by~\eqref{E:K-operator-eq}. Consequently, $\widetilde{u}$ satisfies~\eqref{E:local-bifurcation-eq2-new} with $b = 0$.

  The proof follows the same arguments as in the non-minimal model. We give the
  main differences. We divide the proof into four steps.

  \smallskip

  \noindent\textbf{Step $1'$:~} We identify the invariant fixed-mass phase space
  and construct the center manifold.

  \smallskip

  Recall the fractional power space $X_\theta$ introduced earlier in
  Section~\ref{SS:local-bifurcation-proof-1}, where $3/4 < \theta < 1$. Define
  \[
    \hat{X}_\theta
    \coloneqq
    \big\{w\in X_\theta:\int_0^L w(x)\,dx=0\big\}.
  \]
  Then
  \[
    X_\theta=\hat{X}_\theta\oplus\operatorname{span}\{1\}.
  \]
  It is easy to see that $\hat{X}_\theta$ is invariant
  under~\eqref{E:local-bifurcation-eq2}--\eqref{E:local-bifurcation-eq3} with $a
  = b = 0$.

  Set
  \[
    \hat{X}_\theta
    =\operatorname{span}\{\phi_{n_0}\}\oplus\hat{X}_\theta^s.
  \]
  Since $n_0$ is the unique bifurcation mode, the linearization at $(\widetilde{u},\chi_0) = (0,\chi^*)$, restricted to $\hat{X}_\theta$, has the simple
  eigenvalue $\sigma_{n_0}(\chi^*) = 0$, while $\sigma_k(\chi^*) < 0$ for every
  $k
  \ge 1$, $k \ne n_0$.

  As in Step 1 of the proof for the non-minimal model, there is a locally
  invariant manifold
  \[
    \mathcal{M}_0
    =\left\{
    \left(\hat{h}(u_{n_0},\widetilde{\chi})+u_{n_0}\phi_{n_0},\widetilde{\chi}\right):
    |u_{n_0}|\ll1,\ |\widetilde{\chi}|\ll1
    \right\}
  \]
  of~\eqref{E:local-bifurcation-eq2-new} with $b = 0$, where $\hat{h}:\operatorname{span}\{\phi_{n_0}\}\times\mathbb{R} \longrightarrow \hat{X}_\theta^s$ is a $C^4$ map satisfying $\hat{h}(0, 0) = 0$, $D\hat{h}(0, 0) = 0$,
  $\hat{h}(0, \widetilde{\chi}) = 0$ for all $|\widetilde{\chi}|\ll1$.

  \smallskip

  \noindent\textbf{Step $2'$:~} We derive the Fourier-mode equations on the
  fixed-mass space.

  \smallskip

  For a solution on $\mathcal{M}_0$, write
  \[
    \widetilde{u}(t,x)
    =\sum_{k=1}^\infty u_k(t)\cos\Big(\tfrac{k\pi x}{L}\Big),\qquad
    \widetilde{v}(t,x)
    =\sum_{k=0}^\infty v_k(t)\cos\Big(\tfrac{k\pi x}{L}\Big).
  \]
  The sum for $\widetilde{u}$ begins at $k = 1$ because mass conservation gives
  $u_0(t) = 0$. In contrast, $v_0$ need not vanish when $\gamma \ne 1$. Setting
  $u_0 = 0$ in~\eqref{E:v0-eq} gives
  \begin{align*}
    v_0
    ={}&C_0^{(2)}\sum_{j=1}^\infty u_j^2
    +\frac{\nu\gamma(\gamma-1)(\gamma-2)(u^*)^{\gamma-3}}{8\mu}
    \sum_{i,j\ge1}u_i u_j u_{i+j}
    +\mathcal{O}_0\left(\|u\|_{\ell^1}^4\right).
  \end{align*}
  For $k \ge 1$, the coefficients $v_k$ are given by~\eqref{E:vk-eq} with $u_0 =
  0$.

  Projecting the $u$-equation onto the nonconstant cosine modes and setting $b =
  0$ in~\eqref{E:mode-ode-uk} yields
  \[
    \frac{du_k}{dt}
    =\sigma_k(\chi_0)u_k
    +\chi_0\mathcal{G}_k(u,v)
    +\mathcal{O}_k\left(\|u\|_{\ell^1}^4\right),\qquad k\ge1.
  \]
  Therefore, near $\chi_0 = \chi^*$,
  \begin{equation}\label{E:ode-chi*-minimal}
    \begin{cases}
      \displaystyle
      \frac{du_{n_0}}{dt}
      =\left(\partial_{\chi_0}\sigma_{n_0}(\chi^*)\right)
      \widetilde{\chi} u_{n_0}
      +(\chi^*+\widetilde{\chi})\mathcal{G}_{n_0}(u,v)
      +\mathcal{O}_{n_0}\left(\|u\|_{\ell^1}^4\right),
      \\[3mm]
      \displaystyle
      \frac{du_k}{dt}
      =\left(\sigma_k(\chi^*)
      +(\partial_{\chi_0}\sigma_k)(\chi^*)\widetilde{\chi}\right)u_k
      +(\chi^*+\widetilde{\chi})\mathcal{G}_k(u,v)
      +\mathcal{O}_k\left(\|u\|_{\ell^1}^4\right),
      \quad k\ge1,\ k\ne n_0,
      \\[3mm]
      \displaystyle
      \frac{d\widetilde{\chi}}{dt}=0.
    \end{cases}
  \end{equation}

  \smallskip

  \noindent\textbf{Step $3'$:~} We compute the reduced equation on
  $\mathcal{M}_0$ through cubic order.

  \smallskip

  By the same arguments as in the non-minimal case, we get the scalar reduced
  equation
  \begin{equation}\label{E:center-ODE-initial-minimal}
    \frac{du_{n_0}}{dt}
    =\alpha_{n_0}\widetilde{\chi} u_{n_0}
    -\beta_{n_0}u_{n_0}^3
    +\mathcal{O}\left(
    |u_{n_0}|^2|\widetilde{\chi}|^2+|u_{n_0}|^3|\widetilde{\chi}|+|u_{n_0}|^4
    \right),
  \end{equation}
  with
  \begin{equation}\label{E:beta-n0-minimal}
    \beta_{n_0}=-\chi^*\Gamma_{n_0}^{(3)},
  \end{equation}
  and $\alpha_{n_0}$ as in~\eqref{E:alpha-n0}, where we set $a = b = 0$, and
  $a_{0,1} = 0$ due to the fixed-mass constraint.

  \smallskip

  \noindent\textbf{Step $4'$:~} We use the same argument as in Step 5 for the
  non-minimal model, now in the fixed-mass space.
  Theorem~\ref{T:local-bifurcation} in the minimal case.
\end{proof}

\section{Effects of $\beta$ and $L$ on the bifurcation and proof of Theorem~\ref{T:direction-change}}
\label{S:direction-change}

In this section, we study the effects of the nonlinear sensitivity exponent
$\beta$ and the domain size $L$ on the bifurcation of positive constant
solutions of~\eqref{E:main-PE} and prove Theorem~\ref{T:direction-change}.

\subsection{Effects of $\beta$  and proof of Theorem~\ref{T:direction-change}(1)}
\label{SS:beta-analysis}

In this subsection, we explore the changes of bifurcation direction as $\beta$
changes, but all the other parameters are fixed, and prove
Theorem~\ref{T:direction-change}(1).

\begin{proof} [Proof of Theorem~\ref{T:direction-change}(1)]
  We first prove \eqref{E:beta-n0-eq}. Recall
  \[
    \beta_{n_0}
    = \frac{(1+\alpha)\alpha\,b\,(u^*)^{\alpha-1}}{4}
    \bigl(4a_{0,1} + 2a_{2n_0,1}\bigr)
    - \chi^*\,\Gamma_{n_0}^{(3)}
    + \frac{(1+\alpha)\alpha(\alpha-1)b\,(u^*)^{\alpha-2}}{8}
  \]
  (see~\eqref{E:beta-n0}). Throughout this proof, we assume that all the
  parameters are fixed, except $\beta$. We first find the dependence of
  $a_{2n_0,1}$ and $\chi^*\Gamma_{n_0}^{(3)}$ on $\beta$.

  First, for simplicity, define
  \[
    p(\beta)\coloneqq a_{2n_0,1}(\beta) =
    \frac{\frac{bA''}{2}-\chi^*\Gamma_{2n_0}}
    {\sigma_{2n_0}(\chi^*)},
    \qquad
    q(\beta)\coloneqq
    \chi^*(u^*,\beta)\Gamma_{n_0}^{(3)}(\beta).
  \]
  Using this notation, formula~\eqref{E:beta-n0} becomes
  \begin{equation}\label{E:beta-n0-pq}
    \beta_{n_0}(\beta)
    =
    \frac{bA''}{2}
    \left(
    4a_{0,1}+2p(\beta)
    \right)
    -q(\beta)
    +\frac{3bA'''}{4}.
  \end{equation}
  Set
  \begin{equation}\label{E:Pi-eq}
    \Pi
    \coloneqq
    \frac{\chi^*(u^*,\beta)}{(1+v^*)^\beta}\lambda_{n_0}
    \overset{\text{by~\eqref{E:chi-star}}}{=}
    \frac{(\lambda_{n_0}+a\alpha)(\mu+\lambda_{n_0})}{\nu\gamma\,(u^*)^{m+\gamma-1}}
    .
  \end{equation}
  Hence, $\Pi$ is positive and independent of $\beta$. It follows
  from~\eqref{E:Gamma2n0-explicit} that
  \[
    \chi^*\Gamma_{2n_0}
    =
    \Pi
    \Big(
    4(u^*)^mC_{2n_0}^{(2)}
    +m(u^*)^{m-1}C_{n_0}^{(1)}
    -\beta
    \frac{(u^*)^m\bigl(C_{n_0}^{(1)}\bigr)^2}{1+v^*}
    \Big).
  \]
  Hence, $\chi^*\Gamma_{2n_0}$ is affine in $\beta$. Since
  \[
    \sigma_{2n_0}(\chi^*)=-4\lambda_{n_0}+ \Pi \frac{4\nu \gamma
    (u^*)^{m+\gamma-1}}{\mu+4\lambda_{n_0}}-a\alpha,
  \]
  the function $p$ is also affine in $\beta$. This together
  with~\eqref{E:Gamma-n0-explicit} and~\eqref{E:Pi-eq} shows that $q$ is a
  polynomial of degree at most two.

  Because $p$ is affine and $q$ is quadratic,  the right-hand side
  of~\eqref{E:beta-n0-pq} is a quadratic polynomial in $\beta$. Its Taylor
  expansion at $\beta = 0$ is exact and gives
  \[
    \beta_{n_0}(\beta)=A\beta^2+B\beta+C,
  \]
  where
  \begin{align}\label{E:A-B-C-simple}
    A
    =-\frac12q''(0),
    \quad
    B
    =bA''p'(0)-q'(0),
    \quad
    C
    =
    \frac{bA''}{2}
    \left(
    4a_{0,1}+2p(0)
    \right)
    -q(0)
    +\frac{3bA'''}{4}.
  \end{align}

  A direct differentiation of~\eqref{E:Gamma-n0-explicit} twice with respect to
  $\beta$ gives
  \begin{equation}\label{E:A-closed-form}
    A
    =
    \frac{
    \bigl(C_{n_0}^{(1)}\bigr)^3
    (u^*)^m\Pi
    \left(
    4C_{2n_0}^{(1)}(u^*)^m\Pi
    -\sigma_{2n_0}(\chi^*)
    \right)
    }{
    8\sigma_{2n_0}(\chi^*)(1+v^*)^2
    }.
  \end{equation}
  We now determine its sign. By the definitions of the coefficients,
  \[
    C_{n_0}^{(1)}>0,\qquad
    C_{2n_0}^{(1)}>0,\qquad
    u^*>0,\qquad
    \Pi>0.
  \]

  In the case that $a, b > 0$, since $n_0$ is the unique bifurcation mode and
  $2n_0 \ne n_0$, we have
  \begin{equation}\label{E:A-num-denom}
    \sigma_{2n_0}(\chi^*)<0, \quad {\rm and\,\, then}\,\,
    4C_{2n_0}^{(1)}(u^*)^m\Pi
    -\sigma_{2n_0}(\chi^*)>0.
  \end{equation}
  Therefore, the numerator in~\eqref{E:A-closed-form} is positive. The
  denominator is negative because $ 8\sigma_{2n_0}(\chi^*)(1 + v^*)^2 < 0$.
  Consequently, $A < 0$. In particular, $A$ cannot vanish under the unique
  bifurcation mode assumption.

  In the minimal model, we impose the fixed-mean constraint, so $a_{0,1} = 0$.
  We retain the notation $p(\beta)$, $q(\beta)$, and $\Pi$, with the
  coefficients evaluated for $a = b = 0$ and $a_{0,1} = 0$. In this case, we
  have $\beta_{n_0}(\beta) = -q(\beta) = A\beta^2 + B\beta + C$, where
  \[
    A=-\frac12q''(0),
    \qquad
    B=-q'(0),
    \qquad
    C=-q(0).
  \]
  The calculation of the quadratic coefficient is unchanged and gives the same
  expression as~\eqref{E:A-closed-form}, where $\sigma_{2n_0}(\chi^*)$ is now
  evaluated at $a = 0$. Since $n_0 = 1$ is the unique bifurcation mode on the
  fixed-mean space, $\sigma_{2n_0}(\chi^*) < 0$, because $2n_0 \notin \{0,
  n_0\}$. All the other factors in the numerator are positive,
  and~\eqref{E:A-num-denom} holds. Hence, the numerator is positive and the
  denominator is negative. It follows that $A < 0$ in the minimal model as
  well. Thus, in both the non-minimal and minimal cases, $\beta_{n_0}(\beta)$
  is a quadratic polynomial with a strictly negative leading coefficient.

  In the following, we show that one of the cases (i)-(iv) in
  Theorem~\ref{T:direction-change}(1) occurs. Let $D\coloneq B^2 - 4AC$ and
  $\beta_-$ and $\beta_+$ be the roots of $\beta_{n_0}(\beta)$. We then have
  \begin{enumerate}
    \item[(i)] If $C > 0$, there is exactly one positive root $\beta_+$, with
      $\beta_{n_0} > 0$ on $[0, \beta_+)$ and $\beta_{n_0} < 0$ on
      $(\beta_+,\infty)$.
    \item[(ii)] If $C \le 0$, $B > 0$, and $D > 0$, there are two nonnegative
      roots $0 \le \beta_- < \beta_+$, with $\beta_{n_0} > 0$ precisely between
      them. The equality $\beta_- = 0$ occurs when $C = 0$.
    \item[(iii)] If either $C < 0$, $B > 0$, and $D = 0$, or $C = 0$ and $B \le
      0$, there is one nonnegative root $\beta_0$ at which $\beta_{n_0}(\beta_0)
      = 0$, and $\beta_{n_0}(\beta) < 0$ for every other $\beta \ge 0$.
    \item[(iv)] In all remaining cases with $C < 0$, one has $\beta_{n_0}(\beta)
      < 0$ for every $\beta \ge 0$.
  \end{enumerate}
  This concludes the proof of Theorem~\ref{T:direction-change}(1).
\end{proof}

Table~\ref{Tb:bifurcation-directions} illustrates several of the sign patterns
in Theorem~\ref{T:direction-change}(1); it does not include a zero of the cubic
coefficient. For the non-minimal model (left panel), $\Omega = (0,1)$ and $n_0 =
1$. The parameter vector $(a,b,\alpha,m,\gamma,\nu,\mu,L)$ equals $(1, 1, 1, 1,
1, 1, 1, 1)$ in the first block, $(1, 1, 3, 1, 2, 10, 100, 1)$ in the second
block, and $(1, 1, 5, 1, 1, 1, 1, 1)$ in the third block. For the minimal model
(right panel), $a = b = 0$. The parameter vector $(u^*,m,\gamma,\nu,\mu,L)$
equals $(1, 1, 1, 1, 1, 4)$ in the first block, $(3, 1, 2, 1, 1, 4)$ in the
second block, and $(1, \tfrac13, \tfrac12, 1, 1, 4)$ in the third block.

\begin{table}[htpb]
  \centering
  \caption{Bifurcation directions for the non-minimal model and the minimal
  model.}
  \label{Tb:bifurcation-directions}
  \renewcommand{\arraystretch}{1.15}
  \setlength{\tabcolsep}{3pt}
  \scriptsize
  \begin{tabular}{|c|c|c|c|c|}
    \hline
    \rowcolor{TableRowGray}
    \multicolumn{5}{|c|}{\textbf{Non-minimal model}} \\ \hline
    \rowcolor{TableRowGray}
    $(m,\gamma,n_0)$ & $\beta$ & $\chi^*(u^*)$ & $\beta_1$  & Type \\ \hline
    \multirow{3}{*}{$(1, 1, 1)$}
    & $0$     & $11.9709$     & $+5.7006$  & super-critical \\
    & $2$     & $47.8837$     & $+0.4202$  & super-critical \\
    & $3$     & $95.7674$     & $-2.2322$  & sub-critical   \\ \hline
    \multirow{3}{*}{$(1, 2, 1)$}
    & $1$     & $7.8796$      & $-17.4850$ & sub-critical   \\
    & $5$     & $11.5366$     & $+8.0931$  & super-critical \\
    & $10$    & $18.5798$     & $-3.6136$  & sub-critical   \\ \hline
    \multirow{3}{*}{$(1, 1, 1)$}
    & $0$     & $16.3762$     & $-2.9280$  & sub-critical   \\
    & $1$     & $32.7524$     & $-14.3380$ & sub-critical   \\
    & $5$     & $524.0387$    & $-60.1009$ & sub-critical   \\ \hline
  \end{tabular}\hspace{1em}
  \begin{tabular}{|c|c|c|c|c|}
    \hline
    \rowcolor{TableRowGray}
    \multicolumn{5}{|c|}{\textbf{Minimal model}} \\ \hline
    \rowcolor{TableRowGray}
    $(u^*,m,\gamma,n_0)$ & $\beta$ & $\chi^*(u^*)$ & $\beta_{n_0}$ & Type \\ \hline
    \multirow{3}{*}{$(1, 1, 1, 1)$}
    & $0$     & $1.6169$      & $+0.0097$     & super-critical \\
    & $4$     & $25.8696$     & $+0.0212$     & super-critical \\
    & $5$     & $51.7392$     & $-0.0450$     & sub-critical   \\ \hline
    \multirow{3}{*}{$(3, 1, 2, 1)$}
    & $0$     & $0.0898$      & $-0.0299$     & sub-critical   \\
    & $1$     & $0.8983$      & $+0.0228$     & super-critical \\
    & $4$     & $898.2502$    & $-0.0577$     & sub-critical   \\ \hline
    \multirow{3}{*}{$(1, \tfrac13, \tfrac12, 1)$}
    & $0$     & $3.2337$      & $-0.0211$     & sub-critical   \\
    & $1$     & $6.4674$      & $-0.0512$     & sub-critical   \\
    & $2$     & $12.9348$     & $-0.0881$     & sub-critical   \\ \hline
  \end{tabular}
\end{table}

\subsection{Effects of domain size $L$ and proof of Theorem~\ref{T:direction-change}(2)}
\label{SS:beta-analysis-1}

In this subsection, we explore the changes of bifurcation direction and
bifurcation modes as the domain size $L$ changes, but all the other parameters
are fixed, and prove Theorem~\ref{T:direction-change}(2).

\begin{proof}[Proof of
  Theorem~\ref{T:direction-change}(2)]

  Recall that
  \[
    \beta_{n_0}
    = \frac{(1+\alpha)\alpha\,b\,(u^*)^{\alpha-1}}{4}
    \bigl(4a_{0,1} + 2a_{2n_0,1}\bigr)
    - \chi^*\,\Gamma_{n_0}^{(3)}
    + \frac{(1+\alpha)\alpha(\alpha-1)b\,(u^*)^{\alpha-2}}{8}
  \]
  (see~\eqref{E:beta-n0}). Note that $\beta_{n_0}$ depends on $L$ since
  $a_{2n_0,1}$ and $\chi^*\Gamma_{n_0}^{(3)}$ depend on $\lambda_{n_0} =
  \frac{n_0^2\pi^2}{L^2}$. Put
  \[
    K=\frac{\nu\gamma(u^*)^{m+\gamma-1}}{(1+v^*)^\beta}.
  \]

  We first establish the rational representation. In the non-minimal model, it is not difficult to see that
  \[
    \chi^*=\frac{(\lambda+a\alpha)(\mu+\lambda)}
    {K\lambda},
    \qquad
    a_{2n_0,1}
    =\frac{N_3(\lambda)}
    {3K(\mu+\lambda)(a\alpha\mu-4\lambda^2)},
  \]
  where $N_3$ is a polynomial of degree at most three, and
  \[
    \chi^*\Gamma_{n_0}^{(3)}
    =\frac{(\lambda+a\alpha)N_4(\lambda)}
    {12K^2(\mu+\lambda)^2(a\alpha\mu-4\lambda^2)},
  \]
  where $N_4$ is a polynomial of degree at most four. Substitution into the
  formula for $\beta_{n_0}$ gives
  \[
    \beta_{n_0}
    =\frac{P(\lambda)}
    {12K^2(\mu+\lambda)^2(4\lambda^2-a\alpha\mu)},
  \]
  where $P(\lambda)$ is a polynomial of degree at most $5$ given by
  \[
    P(\lambda )=-12K^2\big(2bA''a_{0,1}+3bA'''/4\big)
    (\mu+\lambda)^2(a\alpha\mu-4\lambda^2)-4KbA''(\mu+\lambda)N_3+(\lambda+a\alpha)
    N_4.
  \]
  In the minimal model, $n_0 = 1$, $a_{0,1} = 0$, and
  \[
    \chi^*=\frac{\mu+\lambda}{K},
    \qquad
    \sigma_2(\chi^*)=-\frac{12\lambda^2}{\mu+4\lambda}.
  \]
  The same substitution and simplification yield
  \[
    \beta_1
    =\frac{P(\lambda)}{24(u^*)^2(\mu+\lambda)^2},
    \qquad \deg P\le3.
  \]

  Next, we consider the sign changes of $P(\lambda_{n_0})$. Recall that, for
  every $k \ge 1$,
  \begin{equation}\label{E:n0-branch}
    n_0= k\ \text{is the unique minimizing mode exactly for}\quad
    L\in\left(L_{k-1},L_k\right),
  \end{equation}
  while at $L = L_k$ the infimum is attained at the two modes $k$ and $k + 1$
  and at no others.

  In the non-minimal model, $n_0 = k$ is the unique minimizing mode when $L \in
  (L_{k-1}, L_k)$. On these intervals,
  \[
    \frac12\sqrt{a\alpha\mu}<\lambda_{n_0}=\lambda_1<\infty \qquad(k=1),
  \]
  and
  \[
    \frac{k}{k+1}\sqrt{a\alpha\mu} <\lambda_{n_0}=\lambda_k<
    \frac{k}{k-1}\sqrt{a\alpha\mu} \qquad(k\ge2).
  \]
  In the minimal model, $n_0 = 1$ for every $L > 0$, and $\lambda_1$ ranges over
  $(0, \infty)$. In either case, $\lambda_{n_0}$ decreases strictly on each
  interval with fixed mode. Since the denominator is positive, a nonzero $P$
  changes sign at its odd-multiplicity zeros. This proves part~(i), including
  the degree bounds.

  Finally, in the non-minimal model,
  \[
    \lambda_{n_0(L)}\longrightarrow \lambda_\infty:=\sqrt{a\alpha\mu}
    \quad(L\to\infty)
  \]
  along admissible lengths. The rational expression is continuous at
  $\lambda_\infty$, so its limit defines $\beta_\infty$. In the minimal model,
  $\lambda_1 \to 0$, and direct calculation gives
  \[
    P(0)=24(u^*)^2\mu^2\beta_\infty^{\min}.
  \]
  Hence, $\beta_1 \to \beta_\infty^{\min}$.

  In both models, polynomial division gives
  \[
    \beta_{n_0}(\lambda)=c\lambda+\mathcal{O}(1) \quad(\lambda\to\infty),
  \]
  with $c = c_{\min}$ in the minimal case. Since $n_0 = 1$ for all sufficiently
  small $L$,
  \[
    \beta_{n_0} =c\frac{\pi^2}{L^2}+\mathcal{O}(1) \quad(L\to0+).
  \]
  The signs of $\beta_\infty$ and $c$, when nonzero, prove parts~(ii) and~(iii),
  respectively.
\end{proof}

\begin{example}\label{Ex:example-1}
  Computation of $\beta_{n_0}$ and switch of bifurcation regimes when $a = b =
  \mu = \nu = m = \alpha = \gamma = 1$
\end{example}

In this example, we discuss the switch of bifurcation regimes when $a = b = \mu
= \nu = m = \alpha = \gamma = 1$. Note that, in this case, we have $u^* = v^* =
1$. Also, by direct calculation, we have
\begin{align*}
  \beta_{n_0}&=\frac{P(\lambda_{n_0},\beta)}{Q(\lambda_{n_0},\beta)}\\
  &=\frac{
  (224\!-\!96\beta)\lambda_{n_0}^4\!+\!(192\!-\!152\beta)\lambda_{n_0}^3 \!+\!(8\!-\!104\beta\!-\!16\beta^2)\lambda_{n_0}^2
  \!+\!(96\!-\!28\beta\!-8\!\beta^2)\lambda_{n_0}\!+\!56\!+\!11\beta\!-\!\beta^2
  }{96(1+\lambda_{n_0})(4\lambda_{n_0}^2-1)}.
\end{align*}
Since $Q(\lambda_{n_0}, \beta) > 0$ for $\lambda_{n_0} > 1/2$, the sign of
$\beta_{n_0}$ is determined by $P(\lambda_{n_0}, \beta)$.

Writing $\lambda = \lambda_{n_0}$, we obtain
\begin{equation}\label{E:new-beta-eq1}
  \beta_{n_0} =\Big(\frac{7}{12}-\frac{\beta}{4}\Big)\lambda +\mathcal{O}(1)
  \qquad(\lambda\to\infty).
\end{equation}
Thus, the small-domain bifurcation is super-critical for $\beta < 7/3$ and
sub-critical for $\beta > 7/3$. For large admissible domains, $\lambda_{n_0}
\to 1$, giving
\begin{equation}\label{E:new-beta-eq2}
  \beta_\infty
  =1-\frac{41}{64}\beta-\frac{25}{576}\beta^2.
\end{equation}
Hence, the large-domain bifurcation is super-critical for $\beta <
\beta_{\mathrm{ld}}$ and sub-critical for $\beta > \beta_{\mathrm{ld}}$, where
\[
  \beta_{\mathrm{ld}} =\frac{3\sqrt{21529}-369}{50}\approx1.423658.
\]
\begin{table}[htbp]
  \centering
  \small
  \renewcommand{\arraystretch}{1.12}
  \setlength{\tabcolsep}{8pt}
  \begin{tabular}{@{}cl@{\hspace{2em}}cl@{}}
    \hline
    $\beta$ & Roots & $\beta$ & Roots \\ \hline
    $1$   & None                   & $2.1$      & $7.370444$ \\
    $1.5$ & $0.773411,\;1.518665$ & $7/3$      & None       \\
    $1.9$ & $0.530181,\;3.826548$ & $2.4$      & None       \\
    $2$   & $5.092791$            & $3$        & None       \\ \hline
  \end{tabular}
  \caption{Roots of $P(\lambda,\beta)$ with $\lambda>1/2$ for selected values of $\beta$, rounded to six decimal places.}
  \label{Tb:polynomial-roots}
\end{table}
For $\lambda > 1/2$, $P(\lambda, \beta)$ decreases strictly with $\beta$.
Moreover, $P(\lambda, 1) > 0$ and $P(\lambda, 7/3) < 0$. Thus, the bifurcation
is super-critical for $0 \le \beta \le 1$ and sub-critical for $\beta \ge 7/3$
at every admissible length.

For intermediate values of $\beta$, the roots in
Table~\ref{Tb:polynomial-roots}, together with the eigenvalue intervals in
Theorem~\ref{T:direction-change}(2), give the following examples as $L$
increases.

For instance, for $\beta = 1.5$ and $\beta = 1.9$, the positive roots are
\[
  \begin{array}{c|cc}
    \beta & \lambda_- & \lambda_+\\ \hline
    1.5 & 0.773411 & 1.518665\\
    1.9 & 0.530181 & 3.826548
  \end{array}
\]
(values rounded), with
\[
  \operatorname{sgn}\beta_{n_0}
  =
  \begin{cases}
    +,                             & \frac12<\lambda<\lambda_-
    \text{ or }\lambda>\lambda_+, \\
    -,                             & \lambda_-<\lambda<\lambda_+.
  \end{cases}
\]
Since $\lambda_k = k^2\pi^2/L^2$ decreases as $L$ increases on $(L_{k-1},L_k)$,
the signs follow the sequences
\[
  \begin{array}{c|ccc}
    \beta & k=1 & k=2 & k\ge3\\ \hline
    1.5 & +\to-\to+ & +\to-\to+
    & \begin{cases}
      -\to+, & k=3, \\
      -,     & k\ge4,
    \end{cases}\\[6pt]
    1.9 & +\to-\to+ & - & -
  \end{array}
\]
where $+$ and $-$ denote super-critical and sub-critical bifurcation,
respectively; the cubic coefficient vanishes at each transition.

\section{Extension of the local bifurcation branches and proof of Theorem~\ref{T:global-extension}}
\label{S:global-Bifurcation}

In this section, we study the extension of the local bifurcation branches and
prove Theorem~\ref{T:global-extension}.

We first collect some estimates for stationary solutions of~\eqref{E:main-PE} to
be used in the proofs of Theorem~\ref{T:global-extension} as well as
Theorem~\ref{T:spiky-solu-thm}. Throughout, $\alpha, \gamma, m, \mu, \nu > 0$,
$\beta \ge 0$.

\begin{lemma}\label{L:spiky-solu-lm1}
  Assume $a, b > 0$. Let $u^* = \big(\frac{a}{b}\big)^{\frac{1}{\alpha}}$. Then
  the following hold.
  \begin{enumerate}[(1)]
    \item For every $\alpha, \gamma, m > 0$, and every nonconstant positive
      stationary solution $(u,v)$ of~\eqref{E:main-PE},
      \begin{equation}\label{E:bound-u-eq1-2}
        0<\min_{[0,L]}u<u^*<\max_{[0,L]}u,
        \quad {\rm and}\quad
        \int_0^L u\,dx<Lu^*,\quad
        \int_0^L u^{1+\alpha}\,dx<L\left(u^*\right)^{1+\alpha}.
      \end{equation}

    \item Assume $0 < \gamma \le 1 + \alpha$ and ${1\le m\le \alpha+2}$. For
      every $\overline{X} > 0$, there exist $\theta \in (0, 1)$ and $ C =
      C(\overline{X}, a, b, \alpha, \gamma, \mu, \nu, L, m, \beta) > 0 $ such
      that every positive nonconstant stationary solution with $0 \le \chi_0 \le
      \overline{X}$ satisfies
      \begin{equation}\label{E:bound-u-eq3}
        \|u\|_{C^{2+\theta}([0,L])},\,\, \|v\|_{C^{2+\theta}([0,L])}\le C.
      \end{equation}
  \end{enumerate}
\end{lemma}

\begin{proof}
  (1) Note that $u_t = 0$. Then by the first equation in~\eqref{E:main-PE}, we
  have
  \begin{equation}\label{E:mass-identity}
    \int_{0}^L \big(a u-b u^{{1}+{\alpha}}\big)dx=0,\quad {\rm hence}\quad  a\int_0^L u\,dx=b\int_0^L u^{1+\alpha}\,dx.
  \end{equation}
  Since $u$ is positive and nonconstant, $u - u^*$ changes sign. Hence,
  \[
    0<\inf_{x\in[0,L]}u(x)<u^*<\sup_{x\in [0,L]}u(x).
  \]
  By strict H\"older's inequality, it holds
  \[
    \Big(\int_0^L u\,dx\Big)^{1+\alpha}
    <
    L^\alpha\int_0^L u^{1+\alpha}\,dx
    =L^\alpha(u^*)^\alpha\int_0^L u\,dx.
  \]
  \eqref{E:bound-u-eq1-2} then follows.

  \smallskip

  (2) Fix $\overline{X} > 0$ and assume $0 \le \chi_0 \le \overline{X}$. Let
  $(u,v)$ be a positive nonconstant stationary solution of \eqref{E:main-PE}.
  First, note that
  \begin{equation}\label{E:elliptic-u-eq}
    0=u_{xx}+b_1(x) u_x+c_1(x),\quad 0<x<L,\,\,\, u_x(0)=u_x(L)=0,
  \end{equation}
  and
  \begin{equation}\label{E:elliptic-v-eq}
    0=v_{xx}+b_2(x)v+c_2(x),\quad 0<x<L,\,\, v_x(0)=v_x(L)=0,
  \end{equation}
  where $b_1(x) = -\chi_0 m \frac{u^{m-1}(x) v_x(x)}{(1+v)^\beta}$, $c_1(x) =
  -\chi_0 \frac{u^m}{(1+v)^\beta} (\mu v-\nu u^\gamma) +\chi_0 \beta
  u^m\frac{|v_x|^2}{(1+v)^{1+\beta}}+u(a-bu^\alpha)$, $b_2(x) = -\mu$, $c_2(x) =
  \nu u^\gamma(x)$. By a priori estimates for elliptic equations, it then
  suffices to prove that $b_1(x)$, $c_1(x)$, $b_2(x)$, and $c_2(x)$ are
  uniformly H\"older continuous on $[0,L]$. For this, it suffices to prove that
  $u, u_x, v, v_x$, and $v_{xx}$ are uniformly bounded on $[0,L]$.

  Next, we prove the uniform boundedness of $\|v\|_\infty$ and $\|v_x\|_\infty$.
  Note that for any $0 \le \eta \le 1+\alpha$, H\"older's inequality
  and~\eqref{E:bound-u-eq1-2} yield
  \begin{equation}\label{E:M-eta}
    M_\eta:=\int_0^L u^\eta\,dx\le L\left(u^*\right)^\eta.
  \end{equation}
  By the second equation in~\eqref{E:main-PE}, we have
  \[
    \int_0^L v(x)dx=\frac{\nu}{\mu}M_\gamma \le \frac{\nu}{\mu} L (u^*)^\gamma,
    \quad {\rm and}\quad
    v_x(x)=\mu\int_0^x v(x)dx-\nu\int_0^x u^\gamma(x)dx.
  \]
  It then follows that
  \begin{equation}\label{E:vx-sup}
    \|v_x\|_\infty
    \le\mu\int_0^L v\,dx+\nu\int_0^L u^\gamma\,dx
    =2\nu\int_0^L u^\gamma\,dx\le2\nu L\left(u^*\right)^\gamma,
  \end{equation}
  and there is $x_0 \in [0,L]$ such that $v(x_0) \le \frac{\nu}{\mu}
  (u^*)^\gamma$, which implies that
  \begin{equation}\label{E:v-infty-eq}
    \|v\|_\infty \le |v(x_0)|+\sup_{x\in [0,L]}\Big|\int_{x_0}^ x v_y(y)dy\Big|\le \frac{\nu}{\mu}(u^*)^\gamma +2\nu L^2\left(u^*\right)^\gamma.
  \end{equation}

  Now, we prove the uniform boundedness of $\|u\|_\infty$ and $\|u_x\|_\infty$.
  Notice that
  \begin{equation}\label{E:first-integral}
    u_x=\chi_0\frac{u^m}{(1+v)^{\beta}}v_x-F(x),
    \qquad F(x):=\int_0^x\left(au-bu^{1+\alpha}\right)\,dy,
  \end{equation}
  and
  \begin{equation}\label{E:F-bound}
    \|F\|_\infty
    \le a\int_0^L u\,dx+b\int_0^L u^{1+\alpha}\,dx<2aLu^*.
  \end{equation}
  Choose $x_0 \in [0,L]$ such that $u(x_0) < u^*$. We then have
  \[
    u(x)=u(x_0)e^{\chi_0 \int_{x_0}^ x\frac{u^{m-1}(y)}{(1+v(y))^\beta}v_y(y)dy}-\int_{x_0}^x
    e^{\chi_0\int_z^x \frac{u^{m-1}(y)}{(1+v(y))^\beta} v_y(y)dy} F(z) dz\quad \forall \, x\in [0,L].
  \]
  This together with~\eqref{E:M-eta} with $\eta = m-1( \le \alpha+1)$,
  \eqref{E:vx-sup}, \eqref{E:first-integral} and~\eqref{E:F-bound} implies that
  \begin{equation}\label{E:u-ux-sup}
    \|u\|_\infty\le C
    \quad {\rm and}\quad
    \|u_x\|_\infty\le C
  \end{equation}
  for some $C = C(\overline{X}, a, b, \alpha, \gamma, \mu, \nu, L, m, \beta) >
  0$.

  Finally, by the second equation in~\eqref{E:main-PE} again, we have
  \begin{equation}\label{E:vxx-sup}
    \|v_{xx}\|_\infty\le \mu\|v\|_\infty+\nu\|u^\gamma\|_\infty.
  \end{equation}
  By $m \ge 1$ and the uniform bounds of $\|u\|_\infty$, $\|u_x\|_\infty$,
  $\|v\|_\infty$, $\|v_x\|_\infty$, and $\|v_{xx}\|_\infty$, the coefficients
  $b_1,c_1,b_2$, and $c_2$ in~\eqref{E:elliptic-u-eq}
  and~\eqref{E:elliptic-v-eq} are in $C^\theta([0,L])$ with $0 < \theta <
  \min\{1,\gamma\}$ when $m = 1$ and $0 < \theta < \min\{1,\gamma,m-1\}$ when $m
  > 1$. Then, by a priori estimates for elliptic equations, there is $C =
  C(\overline{X}, a, b, \alpha, \gamma, \mu, \nu, L, m, \beta) > 0$ such that
  \eqref{E:bound-u-eq3} holds.
\end{proof}

Similarly, we have the following analogue of Lemma~\ref{L:spiky-solu-lm1} for
the minimal model.

\begin{lemma}\label{L:minimal-spike-bounds}
  Assume that $a = b = 0$,  and $1 \le m \le 2$, $0 < \gamma \le 1$. Let
  $(u^*,v^*)$ be a positive constant stationary solution of~\eqref{E:main-PE}.
  Then the following hold.
  \begin{enumerate}[label=(\arabic*)]
    \item    Let $(u,v)$ be a positive nonconstant  stationary  solution
      of~\eqref{E:main-PE} with $\int_0^L u(x)dx = L u^*$. It holds that $0 <
      \min_{[0,L]}u < u^* < \max_{[0,L]}u$ and $\int_0^L u^2\,dx > L(u^*)^2$.

    \item For every $\bar{X} > 0$, there are $\theta \in (0, 1)$ and $C =
      C(\overline{X}, u^*, \mu, \nu, \beta, \gamma, L, m) > 0$ such that every
      positive nonconstant stationary solution with $\int_0^L u(x)\,dx = Lu^*$
      and $0 \le \chi_0 \le
      \overline{X}$ satisfies
      \begin{equation}\label{E:bound-u-eq4}
        \|u\|_{C^{2+\theta}([0,L])},\,\, \|v\|_{C^{2+\theta}([0,L])}\le C.
      \end{equation}
  \end{enumerate}
\end{lemma}

\begin{proof}
  (1) The strict inequalities involving the minimum and maximum follow from
  continuity, nonconstancy, and the prescribed average. The strict $L^2$
  inequality follows from the Cauchy--Schwarz inequality, since equality would
  require $u$ to be constant.

  (2) Integration of the first stationary equation gives
  \begin{equation}\label{E:minimal-spike-flux}
    u_x=\chi_0\frac{u^m}{(1+v)^\beta}v_x.
  \end{equation}
  The proof then follows the same arguments as in
  Lemma~\ref{L:spiky-solu-lm1}(2), with $a = b = \alpha = 0$, which show that
  $u, u_x, v, v_x$, and $v_{xx}$ are uniformly bounded. Then by a priori
  estimates for elliptic equations, there are $\theta \in (0, 1)$ and
  $C = C(\overline{X}, u^*, \mu, \nu, \beta, \gamma, L, m) > 0$ such
  that~\eqref{E:bound-u-eq4} holds.
\end{proof}

\subsection{Proof of Theorem~\ref{T:global-extension}(1)}

In this subsection, we give the outline of the proof of
Theorem~\ref{T:global-extension}(1).

\begin{proof}[Outline of the proof of Theorem~\ref{T:global-extension}(1)]
  Let $(u^*,v^*)$ be a positive constant stationary solution
  of~\eqref{E:main-PE}. Let $X$ and $\mathcal{A} = \mathcal{A}(u^*)$ be as
  in~\eqref{E:X-eq} and~\eqref{E:A-eq}, respectively. Let
  \begin{equation}\label{E:Y-eq}
    Y:=\begin{cases} L^2((0,L))\qquad &{\rm if}\quad a,b>0\cr\cr
      \{u\in L^2((0,L))\,|\, \int_0^L u(x)dx=0\}\quad &{\rm if}\quad a=b=0
    \end{cases}
  \end{equation}
  with norm $\|u\| = \|u\|_{L^2((0,L))}$. Let
  \begin{equation}\label{E:B-eq}
    \mathcal{B}=\mathcal{B}(u^*):=\{u\in X\,|\, u+u^*\in\mathcal{A}\}.
  \end{equation}
  For $w \in \mathcal{B}$, let $v(w)$ solve
  \begin{equation}\label{E:aux-v-eq}
    \begin{cases}
      -v_{xx}+\mu v=\nu(w+u^*)^\gamma,\quad &x\in (0,L)\cr
      v_x(0)=v_x(L)=0,
    \end{cases}
  \end{equation}
  and set
  \[
    F(\chi_0,w)
    =w_{xx}
    -\chi_0\Big(
    \frac{(w+u^*)^m (v(w))_x}{(1+v(w))^\beta}
    \Big)_x
    +a(w+u^*)-b(w+u^*)^{1+\alpha}.
  \]
  Then $F$ is a mapping from $(0, \infty)\times\mathcal{B} \to Y$. If
  $(\chi_0,w) \in (0, \infty)\times \mathcal{B}$ satisfies $F(\chi_0, w) = 0$,
  then $(u,v) = (w + u^*,v(w))$ is a positive stationary solution
  of~\eqref{E:main-PE}, where $v(w)$ solves~\eqref{E:aux-v-eq}.

  We prove Theorem~\ref{T:global-extension}(1) by applying the global
  bifurcation theorem of Shi and Wang~\cite[Theorem~4.4]{shi.wang:09:on} to the
  mapping $F:(0,\infty)\times\mathcal{B} \subset (0, \infty)\times X \to Y$.

  The set $\mathcal{B}$ is open and convex in $X$. The operator $F$ is smooth,
  and its linearizations and their convex interpolations are compact
  perturbations of the Neumann Laplacian, hence Fredholm of index zero. At $p^*
  = (\chi^*,0)$, the Fourier representation gives
  \[
    \ker D_wF(\chi^*, 0)=\operatorname{span}\{\phi_{n_0}\},
    \qquad
    D_{\chi_0w}F(\chi^*,0)\phi_{n_0}
    =\alpha_{n_0}\phi_{n_0}
    \notin\operatorname{range}\big(D_wF(\chi^*,0)\big),
  \]
  since $\alpha_{n_0} > 0$ and the range is orthogonal to $\phi_{n_0}$. The
  domain spaces have Hilbert norms, so the remaining smoothness hypotheses
  of~\cite[Theorem~4.4]{shi.wang:09:on} also hold.

  Hence, \cite[Theorem~4.4]{shi.wang:09:on} yields  that the connected sets
  $\mathcal{C}_{n_0}^{\pm}$ stated in Theorem~\ref{T:global-extension}(1)
  satisfy alternatives (i)--(iii). Moreover, it is clear that
  \[
    \mathcal{S}=
    \{(\chi_0, u^*+w)\,|\, \chi_0\in (0,\infty), \,\,
    w\in\mathcal{B}\setminus\{0\},\,\, F(\chi_0,w)=0\}.
  \]
  Since $\mathcal{C}$ is a connected component of $\overline{\mathcal{S}}
  \subset (0, \infty)\times {\mathcal{A}}$ (see~\eqref{E:closure-S}) (the
  closure is taken in $\mathbb{R}\times\mathcal{A}$), it is relatively closed in
  $(0, \infty)\times
  \mathcal{A}$. By local uniqueness of the bifurcating curve, after decreasing
  $\varepsilon$ if necessary, $\mathcal{G}_{n_0}^{\mp}$ is relatively open in
  $\mathcal{C}$. It follows that $\mathcal{C}\setminus\mathcal{G}_{n_0}^{\mp}$
  is relatively closed in $(0, \infty)\times\mathcal{A}$. Therefore, its
  connected component $\mathcal{C}_{n_0}^{\pm}$ containing
  $\mathcal{G}_{n_0}^{\pm}$ is relatively closed in $(0,
  \infty)\times\mathcal{A}$. Moreover,
  $\mathcal{G}_{n_0}^{\pm}\cup\{(\chi^*,u^*)\}$ is a connected subset of
  $\mathcal{C}\setminus\mathcal{G}_{n_0}^{\mp}$. Hence, $(\chi^*,u^*) \in
  \mathcal{C}_{n_0}^{\pm}$.
\end{proof}

\subsection{Proof of Theorem~\ref{T:global-extension}(2)}

\begin{proof}[Proof of Theorem~\ref{T:global-extension}(2)]
  Suppose $\mathcal{C}_{n_0}^{+}$ is not compact but its $\chi_0$-coordinate is
  bounded. Choose a sequence of nonconstant stationary solutions $(\chi_j,u_j) =
  (\chi_j,w_j + u^*) \in \mathcal{C}_{n_0}^{+}$ with no convergent subsequence
  in $\mathcal{C}_{n_0}^+$. Without loss of generality, we may assume that
  \[
    0<\chi_j\le\overline{X},
    \qquad
    \chi_j\to\bar{\chi}\ge0
  \]
  for some $\overline{X} > 0$ and $\bar{\chi} \ge 0$. By
  Lemma~\ref{L:spiky-solu-lm1}(2) and regularity for elliptic equations, we have
  \[
    \sup_j\left(
    \|u_j\|_{C^{2+\theta}}
    +\|v_j\|_{C^{2+\theta}}
    \right)<\infty
  \]
  for some $\theta > 0$. Hence, after extraction and relabeling, we may assume
  that
  \[
    u_j\to\bar{u},\qquad v_j\to\bar{v}
    \quad\text{in }C^2([0,L]).
  \]
  The limit is a nonnegative stationary solution at $\chi_0 = \bar{\chi}$.

  By the relative closedness of $\mathcal{C}_{n_0}^+$ in $(0,
  \infty)\times\mathcal{A}$, $(\bar{\chi},\bar{u})$ lies on the boundary of $(0,
  \infty)\times\mathcal{A}$. By Lemma~\ref{L:spiky-solu-lm1}(2),
  $(u_j(x),v_j(x))$ converges to $(\bar{u}(x),\bar{v}(x))$ in $C^2([0, L])$ and
  $(\bar{u}(x),\bar{v}(x))$ is a nonnegative stationary solution
  of~\eqref{E:main-PE} with $\chi_0 = \bar{\chi}$. By
  Lemma~\ref{L:spiky-solu-lm1}(1), $\sup_{x\in [0,L]}\bar{u}(x) \ge u^*$. Then
  strong comparison principle for elliptic equations implies that $\bar{u}(x) >
  0$ for all $x \in [0, L]$.  Thus $(\bar{\chi},\bar{u})$ belongs to the closure of
  $\mathcal{S}$ in $\mathbb{R}\times\mathcal{A}$. By~\eqref{E:closure-S},
  $\bar{\chi} > 0$. Hence, $(\bar{\chi},\bar{u}) \in (0, \infty)\times\mathcal{A}$,
  contradicting the boundary conclusion above. Therefore, if
  $\mathcal{C}_{n_0}^+$ is not compact, then its $\chi_0$-coordinate is
  unbounded. The same argument applies to $\mathcal{C}_{n_0}^-$.
\end{proof}

\subsection{Proof of Theorem~\ref{T:global-extension}(3)}

\begin{proof}[Proof of Theorem~\ref{T:global-extension}(3)]
  First, we prove $u^+(\chi_0)$ is monotone for $\chi_0$ near $\chi^*$.
  From~\eqref{E:two-branches}, set $s =
  \sqrt{\frac{\alpha_1}{\beta_1}\tilde{\chi}} + \mathcal{O}(|\tilde{\chi}|)$,
  and write the positive local half-branch as
  \[
    u^+=u^*+s\cos\big(\tfrac{\pi x}{L}\big)+r_s,
    \qquad s>0,\qquad \|r_s\|_{C^2}=o(s).
  \]
  Write $u_s = u^+$, and denote its sensitivity and signal by $\chi_s$ and
  $v_s$. The signal equation gives $v_s = v^* + s\eta\phi_1 + o(s)$ in $C^2$,
  where $\eta = \nu\gamma(u^*)^{\gamma-1}/(\mu + \lambda_1)$. Since $\chi_s \to
  \chi^*$ and $\chi^*{(u^*)^m}\eta/(1 + v^*)^\beta = 1$,
  differentiating~\eqref{E:minimal-spike-flux} gives $u_s'' = s\phi_1'' + o(s)$
  uniformly. Together with $H^2\hookrightarrow C^1$, this yields $\|r_s\|_{C^2}
  = o(s)$. The Neumann conditions imply $r_s'(0) = r_s'(L) = 0$. Integrating
  $r_s''$ from either endpoint yields
  \[
    |r_s'(x)|
    \le\|r_s''\|_\infty\min\{x,L-x\}.
  \]
  Since $\sin\left(\frac{\pi x}{L}\right) \ge \frac{2}{L}\min\{x,L - x\}$, we
  obtain
  \[
    (u^+)'(x)
    \le\Big(
    -\frac{2\pi s}{L^2}+\|r_s''\|_\infty
    \Big)\min\{x,L-x\}<0
    \qquad(0<x<L)
  \]
  for sufficiently small $s > 0$. Although $(u^+)'(0) = (u^+)'(L) = 0$, for
  every $0 \le x_1 < x_2 \le L$,
  \[
    u^+(x_2)-u^+(x_1)
    =\int_{x_1}^{x_2}(u^+)'(x)\,dx<0.
  \]
  Thus $u^+$ is strictly decreasing on $[0, L]$. Similarly, $u^-$ is strictly
  increasing on $[0, L]$.

  Next, we prove all the solutions in $\mathcal{C}_1^\pm$ are monotone. By the
  first equation in~\eqref{E:main-PE} with $a = b = 0$ and $u_t = 0$, we have
  \begin{equation}\label{E:y-eq}
    u_x=\frac{\chi_0u^{m}}{(1+v)^\beta}y,
    \qquad
    y_{xx}
    +\Big(
    \frac{\chi_0\nu\gamma u^{m+\gamma-1}}{(1+v)^\beta}-\mu
    \Big)y=0,
    \qquad y=v_x,\quad y(0)=y(L)=0.
  \end{equation}
  For a nonconstant solution, uniqueness implies
  \[
    y(x_0)=y_x(x_0)=0
    \quad\Longrightarrow\quad
    y\equiv0
    \quad\Longrightarrow\quad
    u\equiv u^*,
  \]
  a contradiction. Thus, every zero of $v_x(x)$ is simple, which implies that
  every zero of $u_x(x)$ is simple. Therefore, strict monotonicity persists
  under sufficiently small perturbations among stationary solutions.

  Suppose that a sequence of decreasing nonconstant solutions satisfies
  $(\chi_j,u_j) \to (\widehat{\chi}, u^*)$ in $(0, \infty)\times\mathcal{A}$. If
  $D_wF(\widehat{\chi}, 0)$ were invertible, the implicit function theorem would
  exclude such a sequence. Hence, $\widehat{\chi} = \chi_n(u^*)$ for some $n \ge
  1$, where~\eqref{E:chi-n-eq-new1} with $a = 0$ gives strictly increasing
  critical values. At each such value the kernel is
  $\operatorname{span}\{\phi_n\}$ and the parameter crossing is transversal, as
  follows from~\eqref{E:sigma_n}. The local simple-eigenvalue bifurcation
  theorem~\cite[Theorem~4.3]{shi.wang:09:on} applied to the stationary mapping
  $F$ therefore gives
  \[
    u_j=u^*+s_j\phi_n+r_j,
    \qquad s_j\ne0,\quad \|r_j\|_{H^2}=o(|s_j|).
  \]
  If $n \ge 2$, the derivative of $\phi_n$ takes both signs in $(0, L)$. Since
  $H^2(0,L)\hookrightarrow C^1([0, L])$, the derivative of $u_j$ would also take
  both signs for large $j$, a contradiction. Thus, $n = 1$ and $\widehat{\chi} =
  \chi^*$.

  Consequently,
  \[
    \{(\chi^*,u^*)\}\cup
    \left\{
    (\chi_0,u)\in\mathcal{C}_1^+:
    u_x<0\text{ on }(0,L)
    \right\}
  \]
  is relatively open and closed in $\mathcal{C}_1^+$.

  Near a positive stationary limit, $H^2$ convergence bootstraps to $C^2$
  convergence. A nonconstant limit of decreasing solutions remains strictly
  decreasing by uniqueness in~\eqref{E:y-eq}. The interior sign and simple
  endpoint zeros persist under small stationary perturbations. Constant limits
  occur only at $(\chi^*,u^*)$; there, local uniqueness and removal of the
  opposite half-arc retain the decreasing orientation.

  Connectedness, and the analogous argument for $\mathcal{C}_1^-$, yield
  \[
    u_x<0\text{ on }\mathcal{C}_1^+\setminus\{(\chi^*,u^*)\},
    \qquad
    u_x>0\text{ on }\mathcal{C}_1^-\setminus\{(\chi^*,u^*)\},
  \]
  with the inequalities understood on $(0, L)$. Thus, alternative~(ii) is
  excluded. Moreover,
  \[
    \int_0^L(u-u^*)\phi_1\,dx
    =-\frac{L}{\pi}\int_0^L
    u_x\sin\big(\frac{\pi x}{L}\big)\,dx\ne0,
  \]
  excluding alternative~(iii). Hence, both connected sets are noncompact.

  Suppose that sensitivity is bounded on $\mathcal{C}_1^+$. Since this component
  is noncompact, choose a sequence of nonconstant solutions $(\chi_j,u_j) \in
  \mathcal{C}_1^+$ with no convergent subsequence in $\mathcal{C}_1^+$.
  Lemma~\ref{L:minimal-spike-bounds}(2) gives, after extraction,
  \[
    \chi_j\to\bar{\chi}\ge0,\qquad
    (u_j,v_j)\to(\bar{u},\bar{v})
    \quad\text{in }C^2([0,L]).
  \]
  By the relative closedness of $\mathcal{C}_{1}^+$ in $(0,
  \infty)\times\mathcal{A}$, $(\bar{\chi},\bar{u})$ lies on the boundary of $(0,
  \infty)\times\mathcal{A}$. By Lemma~\ref{L:minimal-spike-bounds}(2),
  $(u_j(x),v_j(x))$ converges to $(\bar{u}(x),\bar{v}(x))$ in $C^2([0, L])$ and
  $(\bar{u}(x),\bar{v}(x))$ is a nonnegative stationary solution
  of~\eqref{E:main-PE} with $\chi_0 = \bar{\chi}$. Note that $\int_0^L \bar{u}(x)\,dx = L u^*$. Hence, $\bar{u}\not \equiv 0$ and strong comparison
  principle for elliptic equations implies that $\bar{u}(x) > 0$ for all $x \in
  [0, L]$. Thus, $(\bar{\chi},\bar{u})$ belongs to the closure of $\mathcal{S}$ in
  $\mathbb{R}\times\mathcal{A}$. By~\eqref{E:closure-S}, $\bar{\chi} > 0$. Hence,
  $(\bar{\chi},\bar{u}) \in (0, \infty)\times\mathcal{A}$, contradicting the
  boundary conclusion above. Hence~\eqref{E:bifurcation-branches} holds.  The
  same argument applies to $\mathcal{C}_1^-$.

  Finally, adapting the arguments of~\cite[Theorem 3.1]{wang.xu:13:spiky}
  and~\cite[Theorem 3.1]{li:18:spiky}, we now prove that the mass of $u$
  concentrates at $x = 0$ along $\mathcal{C}_1^+$ and at $x = L$ along
  $\mathcal{C}_1^-$.

  Let $(\chi_j,u_j) \in \mathcal{C}_1^+$ with $\chi_j \to \infty$, and let $v_j$
  be the corresponding signal. Then $u_j$ is strictly decreasing and $\int_0^L
  u_j\,dx = Lu^*$.  The $v$-estimate in the proof of
  Lemma~\ref{L:minimal-spike-bounds} yields
  \[
    \|v_j\|_\infty\le V:=\frac{\nu}{\mu}(u^*)^\gamma+2\nu L^2(u^*)^\gamma,
  \]
  independently of $j$ and $\chi_j$. By~\eqref{E:y-eq}, the function $w_j = -
  (v_j)_x$ is positive on $(0, L)$ and satisfies
  \[
    -w_j''=
    \Big(
    \frac{\chi_j\nu\gamma u_j^{m+\gamma-1}}{(1+v_j)^\beta}-\mu
    \Big)w_j.
  \]

  We claim that $u_j(\delta) \to 0$ for every $\delta \in (0, L)$. Otherwise,
  after passing to a subsequence, $u_j(\delta) \ge \varepsilon$ for some
  $\varepsilon > 0$. Monotonicity then implies that, on $I = [\delta/2,
  \delta]$,
  \[
    \frac{\chi_j\nu\gamma u_j^{m+\gamma-1}}{(1+v_j)^\beta}-\mu
    \ge
    \frac{\chi_j\nu\gamma\varepsilon^{m+\gamma-1}}{(1+V)^\beta}-\mu
    \longrightarrow\infty
    \qquad\text{uniformly}.
  \]
  Set $\psi(x) = \sin\left(\frac{2\pi(x-\delta/2)}{\delta}\right)$, and
  $\lambda_I = \frac{4\pi^2}{\delta^2}$. Integration by parts gives
  \[
    \begin{aligned}
      \int_{\delta/2}^{\delta}
      \Big(
      \frac{\chi_j\nu\gamma u_j^{m+\gamma-1}}{(1+v_j)^\beta}-\mu
      \Big)w_j\psi\,dx
      &
      =\lambda_I\int_{\delta/2}^{\delta}w_j\psi\,dx
      +w_j(\delta)\psi'(\delta)
      -w_j(\delta/2)\psi'(\delta/2) \\
      &
      <\lambda_I\int_{\delta/2}^{\delta}w_j\psi\,dx,
    \end{aligned}
  \]
  which is impossible, since the coefficient of $w_j$ exceeds $\lambda_I$
  throughout $I$ for large $j$. Thus, $u_j(\delta) \to 0$, and
  \[
    \int_\delta^L u_j(x)\,dx
    \le (L-\delta)u_j(\delta)\longrightarrow0.
  \]
  Together with the fixed mass, this implies
  \[
    \int_0^L f(x)u_j(x)\,dx\longrightarrow Lu^*f(0)
    \qquad\text{for every }f\in C([0,L]).
  \]
  Hence, $u_j(x)\,dx$ converges weakly-* to $Lu^*\delta_0$ as finite measures on
  $[0, L]$. By reflection, the corresponding limit along $\mathcal{C}_1^-$ is
  $Lu^*\delta_L$. In particular, $\|u_j\|_\infty \to \infty$ along either
  branch.
\end{proof}

\section{Spike-like stationary solutions  and proof of Theorem~\ref{T:spiky-solu-thm}}
\label{S:spiky-solu}

In this section, we prove Theorem~\ref{T:spiky-solu-thm}.

\subsection{Constant limits at large sensitivity}

We use a common form of the convergence argument in~\cite[Theorem~4.3, arXiv
version~1]{kurt.shen.ea:24:stability}.

\begin{lemma}\label{L:minimal-spike-limit}
  Fix all model parameters except $\chi_0$. Let $(u_n,v_n)$ be positive
  stationary solutions of~\eqref{E:main-PE} with $\chi_0 = \chi_n \to \infty$.
  Assume either $a, b > 0$, with $u^* = (a/b)^{1/\alpha}$, or $a = b = 0$ and
  $\int_0^L u_n\,dx = Lu^* > 0$. Suppose $\sup_n\|u_n\|_{L^r(0,L)} < \infty$ for
  some finite $r > \max\{1, \gamma\}$, and put $q = r/\gamma > 1$. There exist a
  subsequence $\{n_k\}$ and a constant $c \ge 0$ such that
  \begin{equation}\label{E:minimal-spike-limit-1}
    u_{n_k}^\gamma\rightharpoonup c\,\,\text{ in }\,\, L^q,
    \quad v_{n_k}\to\frac\nu\mu c\,\, \text{ in }\,\, C^1,
    \quad v_{n_k}\rightharpoonup\frac\nu\mu c\,\, \text{ in }\,\, W^{2,q}.
  \end{equation}
\end{lemma}

\begin{proof}
  First, by the boundedness assumption of $\{u_n\}$ in $L^r$, $\{v_n\}$ is
  bounded in $W^{2,q}$. By weak compactness and the compact embedding
  $W^{2,q}(0,L)\hookrightarrow C^1([0, L])$, without loss of generality, we may
  assume that there are nonnegative functions $W, V$ such that
  \[
    u_{n}^\gamma\rightharpoonup W\text{ in }L^q,
    \qquad v_{n}\to V\text{ in }C^1,
    \qquad v_{n}\rightharpoonup V\text{ in }W^{2,q}.
  \]
  Note that
  \begin{equation}\label{E:lm2-v-limit}
    -V_{xx}+\mu V=\nu W,\qquad V_x(0)=V_x(L)=0.
  \end{equation}
  Also, $(1 + v_n(x))^{-\beta} \ge \kappa > 0$ for some constant $\kappa > 0$
  and all $x \in [0, L]$. In both models, Lemma~\ref{L:spiky-solu-lm1}(1) or
  fixed mass gives $\int_0^L u_n\,dx \le M: = Lu^*$. For every $m > 0$, direct
  integration of the stationary $u$--equation from $0$ to $x$, using the Neumann
  conditions at $0$, gives
  \begin{equation}\label{E:un-eq}
    (u_n)_x=\chi_n\frac{u_n^m}{(1+v_n)^\beta}(v_n)_x-F_n,
  \end{equation}
  where
  \[
    F_n(x)=\int_0^x\left(au_n(y)-bu_n(y)^{1+\alpha}\right)\,dy,
    \qquad \|F_n\|_\infty\le2aM.
  \]
  Note that $F_n = 0$ in the minimal case.

  Next, we show that $V(\cdot)$ is a constant function. Assume that $V(\cdot)$
  is not a constant function. Choose a non-empty closed interval $J \subset
  \{V_x \ne 0\}\cap(0, L)$ and $0 \le \varphi \in C_c^1(0,L)$ equal to one on
  $J$, supported where $V_x$ has a fixed nonzero sign. Then for all large $n$,
  $(v_{n})_x$ has that sign and $|(v_{n})_x| \ge \sigma > 0$ for some $\sigma >
  0$ and $x$ on the support. By~\eqref{E:un-eq}, we have
  \[
    \chi_n\kappa\sigma\int_Ju_n^m\,dx
    \le\Big|-\int_0^L u_n\varphi_x\,dx+\int_0^L F_n\varphi\,dx\Big|
    \le M\|\varphi_x\|_\infty+2aM\|\varphi\|_1.
  \]
  Thus, $\int_Ju_n^m\,dx \to 0$. With $H = \sup_n\|u_n\|_r$, H\"older's and
  Chebyshev's inequalities yield, for every $\varepsilon > 0$,
  \[
    \int_Ju_n^\gamma\,dx
    \le\varepsilon^\gamma|J|
    + H^\gamma\Big(\varepsilon^{-m}\int_Ju_n^m\,dx\Big)^{1-\gamma/r}.
  \]
  Letting $n \to \infty$ and then $\varepsilon \to 0$ shows that $W = 0$ almost
  everywhere on $J$. Exhausting the components of $\{V_x \ne 0\}$ gives $WV_x =
  0$ almost everywhere. Multiplication of~\eqref{E:lm2-v-limit} by $V_x$
  therefore gives
  \[
    \partial_x(V_x^2-\mu V^2)=0\quad\text{almost everywhere}.
  \]
  Note that $V_x^2 - \mu V^2$ is absolutely continuous and hence constant. At
  the extrema of $V$, including the Neumann endpoints, $V_x = 0$. Thus, $(\min
  V)^2 = (\max V)^2$, and nonnegativity makes $V$ constant, which contradicts
  the assumption that $V$ is not a constant function.
  Equation~\eqref{E:lm2-v-limit} then gives $W = c$ and $V = \nu c/\mu$.
\end{proof}

\subsection{Proof of Theorem~\ref{T:spiky-solu-thm}}

\begin{proof}
  We first prove~\eqref{E:spiky-thm-eq2} for both models at once.

  To this end, choose $x_n \in [0, L]$ such that $u_n(x_n) = \max_{[0,L]}u_n$.
  Extract a further subsequence with $x_n \to x^* \in [0, L]$. Assume
  that~\eqref{E:spiky-thm-eq2} does not hold with this $x^*$. Then, after
  passing to a further subsequence if necessary, we may assume that there is
  $\delta > 0$ such that
  \begin{equation}\label{E:unif-bound-u_x-n}
    |(u_n)_x|\le K\quad\text{on }I\quad {\rm for}\quad n\gg 1,
  \end{equation}
  where $I = [x^* - \delta, x^* + \delta]\cap[0, L]$. Note that, for all large
  $n$, $x_n \in I$, and for any $x \in I$,
  \[
    \max_{[0,L]}u_n=u_n(x_n)=u_n(x)+\int_x^{x_n}  (u_n)_x(y)dy.
  \]
  It then follows that
  \begin{equation}\label{E:unif-bound-u-n}
    \max_{[0,L]}u_n
    \le\frac1{|I|}\int_Iu_n\,dx+K|I|
    \le\frac{Lu^*}{|I|}+K|I|\quad \forall\, n\gg 1.
  \end{equation}
  Consequently, for every $p > 1$,
  \begin{equation}\label{E:lp-bound-u-n}
    \limsup_{n\to\infty} \int_0^L u_n^p(x)dx<\infty.
  \end{equation}
  By the assumption $ \liminf_{n\to\infty}\max_{[0,L]}u_n > u^*$ in the theorem,
  we may assume that
  \begin{equation}\label{E:u-n-limit}
    h_n:=u_n(x_n)=\max_{[0,L]}u_n\longrightarrow\, \ell{ > }\,u^*.
  \end{equation}

  Fix $r > \max\{1, \gamma\}$ and set $q = r/\gamma > 1$. The uniform $L^\infty$
  bound implies boundedness in $L^r(0, L)$. By
  Lemma~\ref{L:minimal-spike-limit}, there exist a subsequence $\{n_k\}$ and a
  constant $c \ge 0$ such that
  \begin{equation}\label{E:new-spike-limit}
    u_{n_k}^\gamma\rightharpoonup c\,\,\text{ in }\,\, L^q,
    \quad v_{n_k}\to\frac\nu\mu c\,\, \text{ in }\,\, C^1,
    \quad v_{n_k}\rightharpoonup\frac\nu\mu c\,\, \text{ in }\,\, W^{2,q}.
  \end{equation}
  By~\eqref{E:unif-bound-u_x-n} and~\eqref{E:unif-bound-u-n}, we may assume that
  $u_{n_k}(x) \to U(x)$ uniformly on $I$ as $k \to \infty$ for some $U \in
  C(I)$. Since $u_n^\gamma\rightharpoonup c$, we have $U^\gamma = c$ on $I$.
  Moreover, $x_n \to x^*$ and $u_n(x_n) \to \ell$ imply $U(x^*) = \ell$, so $c =
  \ell^\gamma$. It then follows that
  \[
    0\le\int_0^L(h_{n_k}^\gamma-u_{n_k}^\gamma)\,dx
    \longrightarrow L\ell^\gamma-Lc=0.
  \]
  Thus, $u_{n_k}^\gamma \to \ell^\gamma$ in $L^1$. This implies that $L\ell =
  \lim_{n_k}\int_0^L u_{n_k}\,dx \le Lu^*$, which
  contradicts~\eqref{E:u-n-limit} in either model. Therefore, along the
  subsequence with $x_n \to x^*$,~\eqref{E:spiky-thm-eq2} holds with this $x^*$.

  Next, we prove~\eqref{E:spiky-thm-eq3}. By Lemma~\ref{L:spiky-solu-lm1},
  \[
    \limsup_{n\to\infty} \int_0^L  u_n^{1+\alpha}(x)dx<\infty.
  \]
  Then by the assumption $\gamma < 1 + \alpha$ and
  Lemma~\ref{L:minimal-spike-limit}, \eqref{E:new-spike-limit} holds with $q =
  \frac{1+\alpha}{\gamma}$. This together with~\eqref{E:bound-u-eq1-2} and
  H\"older's inequality implies that
  \[
    cL=\lim_{n_k\to\infty}\int_0^L u_{n_k}^\gamma(x)dx\le L (u^*)^\gamma.
  \]
  \eqref{E:spiky-thm-eq3} then follows. The further extraction
  preserves~\eqref{E:spiky-thm-eq2}, since the local gradient norms tend to
  infinity; relabeling gives the same subsequence for both conclusions.
\end{proof}

\section{Numerical analysis of the bifurcation and spike-like solutions}
\label{SS:bifurcation-numerics}

In this section, we numerically investigate bifurcation and the formation of
spike-like profiles in the minimal and non-minimal models. We first describe the
spatial discretization and prove that the discrete instability threshold
converges to its continuous counterpart as the mesh is refined. We then present
numerical experiments for parameters in the super-critical and sub-critical
regimes, comparing the finite-time evolution near the instability threshold and
examining the profiles at larger values of $\chi_0$.

\subsection{Numerical scheme and discrete threshold on a mesh and proof of Proposition~\ref{P:semidiscrete-threshold}}
\label{SS:numerical-scheme}

In this subsection, we describe the finite-difference schemes for the two
models. For the non-minimal model, we discretize the space variable
in~\eqref{E:main-PE} by a finite-difference method on $N$ uniform subintervals,
with $N + 1$ grid points, and obtain the following discretized system of ODEs:
\begin{align}\label{E:main-DIS}
  \begin{dcases}
    u_i'
    = (\Delta_h u)_i
    + \beta\chi_0\,\dfrac{u_i^m}{(1+v_i)^{\beta+1}} (D_h v_i)^2
    - m\chi_0\,\dfrac{u_i^{m-1}}{(1+v_i)^\beta} (D_h u_i)(D_h v_i)\cr
    \hphantom{\dot{u}_i = {}}
    - \chi_0\,\dfrac{u_i^m}{(1+v_i)^\beta}
    \left(\mu v_i-\nu u_i^\gamma\right)
    + au_i - b u_i^{1+\alpha},
    \quad i=0,1,\ldots,N,\cr
    0 = (\Delta_h v)_i - \mu v_i + \nu u_i^\gamma,
    \quad i=0,1,\ldots,N,\cr
    u_{-1}=u_1,\ u_{N+1}=u_{N-1},\
    v_{-1}=v_1,\ v_{N+1}=v_{N-1}.
  \end{dcases}
\end{align}
Here $x_i = ih$, $h = L/N$, $u_i(t)\approx u(t, x_i)$, and $v_i(t)\approx v(t,
x_i)$. For $1 \le i \le N - 1$, set
\[
  D_h w_i \coloneqq \frac{w_{i+1}-w_{i-1}}{2h},
  \qquad
  (\Delta_h w)_i \coloneqq \frac{w_{i+1}-2w_i+w_{i-1}}{h^2},
\]
while the ghost-point convention gives
\[
  D_h w_0 = D_h w_N = 0,
  \qquad
  (\Delta_h w)_0 = \frac{2(w_1-w_0)}{h^2},
  \qquad
  (\Delta_h w)_N = \frac{2(w_{N-1}-w_N)}{h^2}.
\]
The vector $(u_i^*,v_i^*) = (u^*,v^*)$, $0 \le i \le N$, is a constant solution
of~\eqref{E:main-DIS}.

For the minimal model, we use the same grid and elliptic solver, but its
constant equilibria are parametrized by the conserved mean. We therefore compare
stability within one fixed-mass class. Changing the prescribed mean changes the
reference constant equilibrium; the neutral constant mode is excluded from the
instability calculation. The continuous $u$--equation preserves mass because it
is in divergence form and the Neumann conditions eliminate the boundary flux.
The expanded pointwise scheme~\eqref{E:main-DIS} has no exact discrete product
rule and need not preserve that identity. Therefore, for the minimal experiments
we instead discretize the complete diffusion--chemotaxis flux. Set
\[
  q_i\coloneqq\frac{u_i^m}{(1+v_i)^\beta},
  \qquad
  \mathcal{F}_{i+\frac12}
  \coloneqq
  \frac{u_{i+1}-u_i}{h}
  -\chi_0\frac{q_i+q_{i+1}}{2}\frac{v_{i+1}-v_i}{h}.
\]
We use
\begin{equation}\label{E:main-DIS-conservative}
  u_i'=
  \begin{cases}
    2\mathcal{F}_{\frac12}/h,                                       & i=0,           \\
    \left(\mathcal{F}_{i+\frac12}-\mathcal{F}_{i-\frac12}\right)/h, & 1\le i\le N-1, \\
    -2\mathcal{F}_{N-\frac12}/h,                                    & i=N.
  \end{cases}
\end{equation}
The endpoint nodes in~\eqref{E:main-DIS-conservative} represent half control
volumes, so the trapezoidal discrete mass
\[
  M_h(u)\coloneqq
  h\left(\frac{u_0}{2}+\sum_{i=1}^{N-1}u_i+\frac{u_N}{2}\right)
\]
satisfies $dM_h(u)/dt = 0$ by telescoping.
We work at fixed mass and use mean-zero perturbations
\[
  u_0(x)=u^*\pm \varepsilon\,\phi_{n_0}(x),
\]
so that $\frac{1}{L}\int_0^L u_0(x)\,dx = u^*$.

As mentioned in the introduction, $\chi_N^{*,{\rm disc}}(u^*,\beta)$ given
in~\eqref{E:disc-threshold} is the instability threshold of the constant
solution $(u^*,v^*)$ of~\eqref{E:main-DIS}. For the minimal model, the same
threshold applies to the conservative scheme~\eqref{E:main-DIS-conservative} on
the fixed-mass subspace: at a constant state, its linearization is
$\Delta_h\delta u - \chi_0q^*\Delta_h\delta v$, with $q^* = (u^*)^m/(1 +
v^*)^\beta$, which agrees with the expanded scheme when $a = b = 0$. In the
following, we prove Proposition~\ref{P:semidiscrete-threshold}, that is,
$\lim_{N\to\infty} \chi_N^{*,{\rm disc}}(u^*,\beta) = \chi_{a,b,\beta}^*(u^*)$.

\begin{proof}[Proof of Proposition~\ref{P:semidiscrete-threshold}]
  Recall that discrete Neumann eigenvalues $\lambda_{n,N}^{\rm disc}$ are given
  by~\eqref{E:disc-eigenvalues}. For every fixed $n \ge 1$, with $N \ge n$,
  \[
    \lambda_{n,N}^{\rm disc}
    \longrightarrow
    \lambda_n
    =
    \left(\frac{n\pi}{L}\right)^2 \quad { {\rm as}\quad N\to\infty}.
  \]
  Moreover, since $\sin x \ge 2x/\pi$ for $0 \le x \le \pi/2$, $
  \lambda_{n,N}^{\rm disc} \ge \frac{4n^2}{L^2}.$

  Let
  \begin{align}\label{E:chi_N-disc}
    F(\lambda)
    \coloneqq
    \frac{(1+v^*)^\beta}
    {\nu\gamma (u^*)^{m+\gamma-1}}
    \frac{(\lambda+a\alpha)(\mu+\lambda)}{\lambda}.
  \end{align}
  and $C\coloneqq \frac{(1+v^*)^\beta} {\nu\gamma (u^*)^{m+\gamma-1}}$. Then
  $F(\lambda) \ge C\lambda$. On the other hand, $\sup_{N\ge1}
  F(\lambda_{1,N}^{\rm disc}) < \infty$. Consequently, there exists $K \ge 1$,
  independent of $N$, such that no mode $n > K$ can minimize either the discrete
  or the continuum thresholds. Thus, for all sufficiently large $N$,
  \[
    \chi_N^{*,{\rm disc}}(u^*,\beta)
    =
    \min_{1\le n\le K}
    F\left(\lambda_{n,N}^{\rm disc}\right),
    \qquad
    \chi_{a,b,\beta}^*(u^*)
    =
    \min_{1\le n\le K}F(\lambda_n).
  \]
  Since convergence on this finite set of modes is uniform,
  \eqref{E:disctere-limit-eq} then follows.
\end{proof}

\subsection{Numerical experiments and observations}
\label{SS:experiments}

In this subsection, we use these schemes to illustrate the super-critical and
sub-critical cases in both models. We compare the near-threshold computations
with the local theory and examine the profiles at the reported final times for
larger sensitivities.

In Figures~\ref{F:beta2-super-non-minimal}--\ref{F:beta5-sub-minimal}, the value
denoted by $\chi^*$ in each plot title is the discrete threshold $\chi_N^{*,
{\rm disc}}$ for the mesh used in that experiment. The first rows of
Figures~\ref{F:beta2-super-non-minimal} and~\ref{F:beta3-sub-non-minimal} show
near-threshold experiments for the non-minimal model with $\beta = 2$ and $\beta
= 3$, respectively. These choices correspond to the first block in the left
panel (non-minimal model) of Table~\ref{Tb:bifurcation-directions} and lie on
opposite sides of the transition $\beta_+\approx2.1586$. Thus,
Figure~\ref{F:beta2-super-non-minimal} uses parameters in the super-critical
regime, whereas Figure~\ref{F:beta3-sub-non-minimal} uses parameters in the
sub-critical regime.

\begin{figure}[htbp]
  \centering
  \begin{subfigure}[b]{0.45\textwidth}
    \centering
    \includegraphics[width=\textwidth]{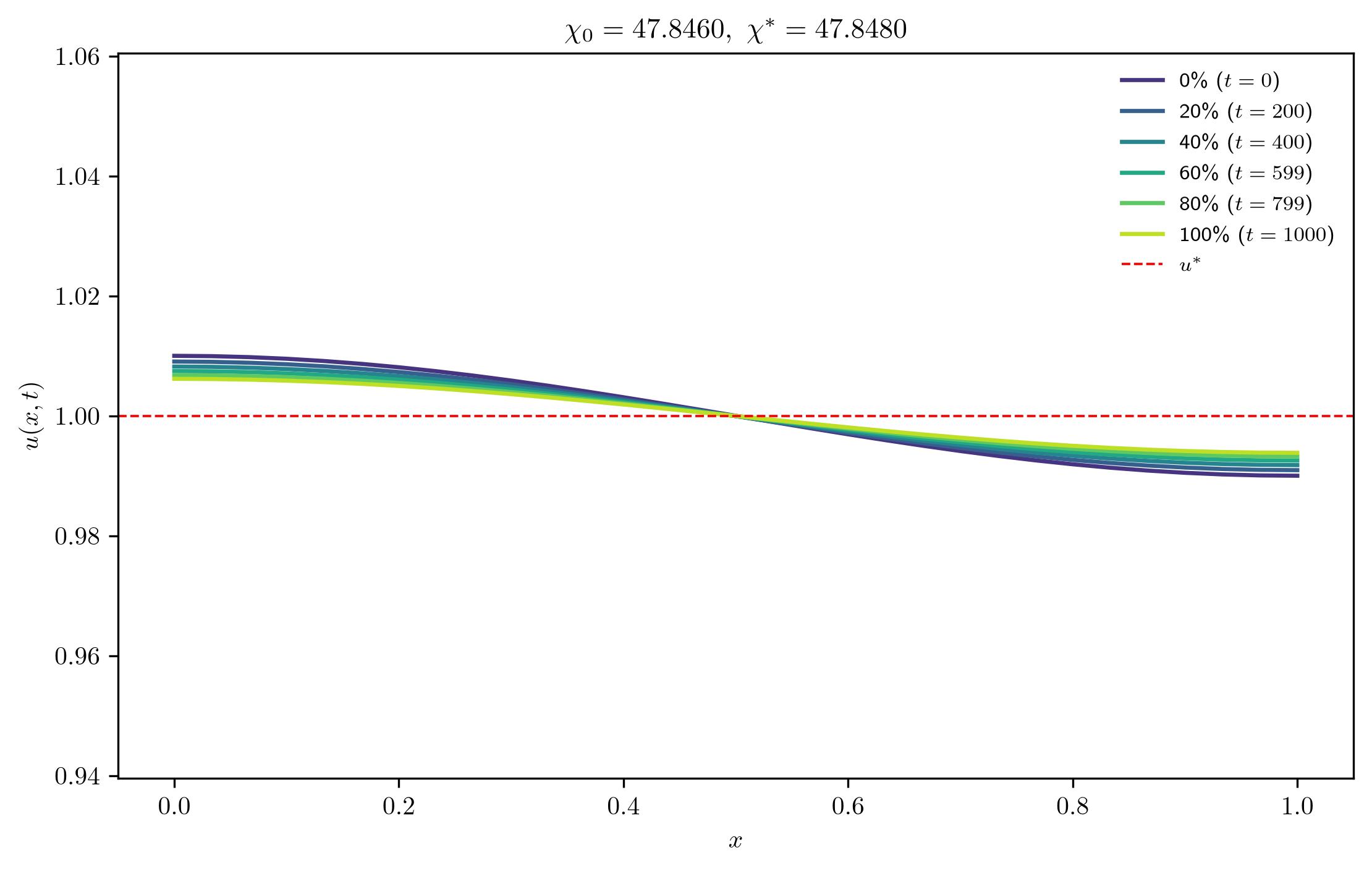}
  \end{subfigure}
  \hfill
  \begin{subfigure}[b]{0.45\textwidth}
    \centering
    \includegraphics[width=\textwidth]{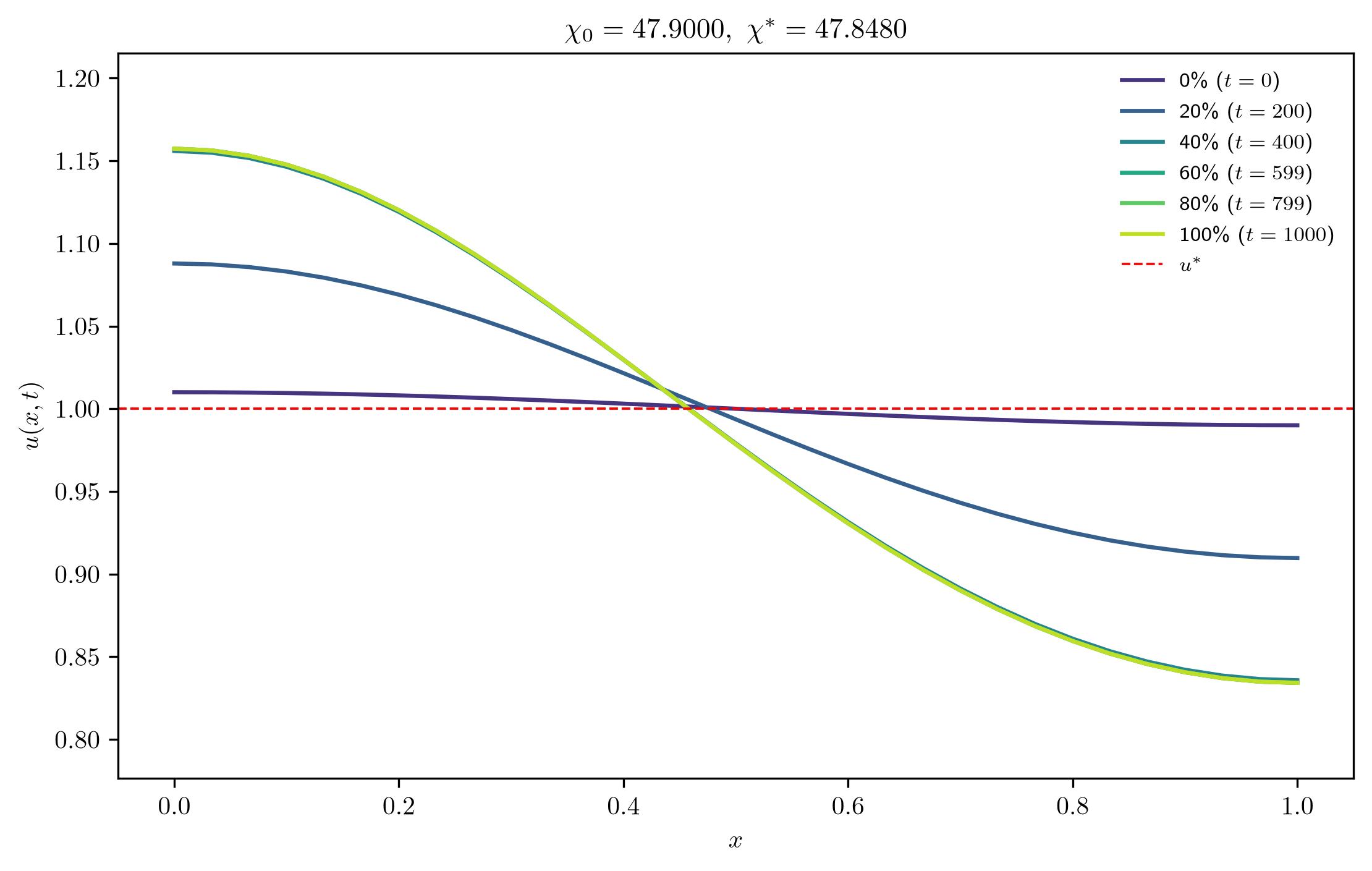}
  \end{subfigure}
  \hfill
  \begin{subfigure}[b]{0.45\textwidth}
    \centering
    \includegraphics[width=\textwidth]{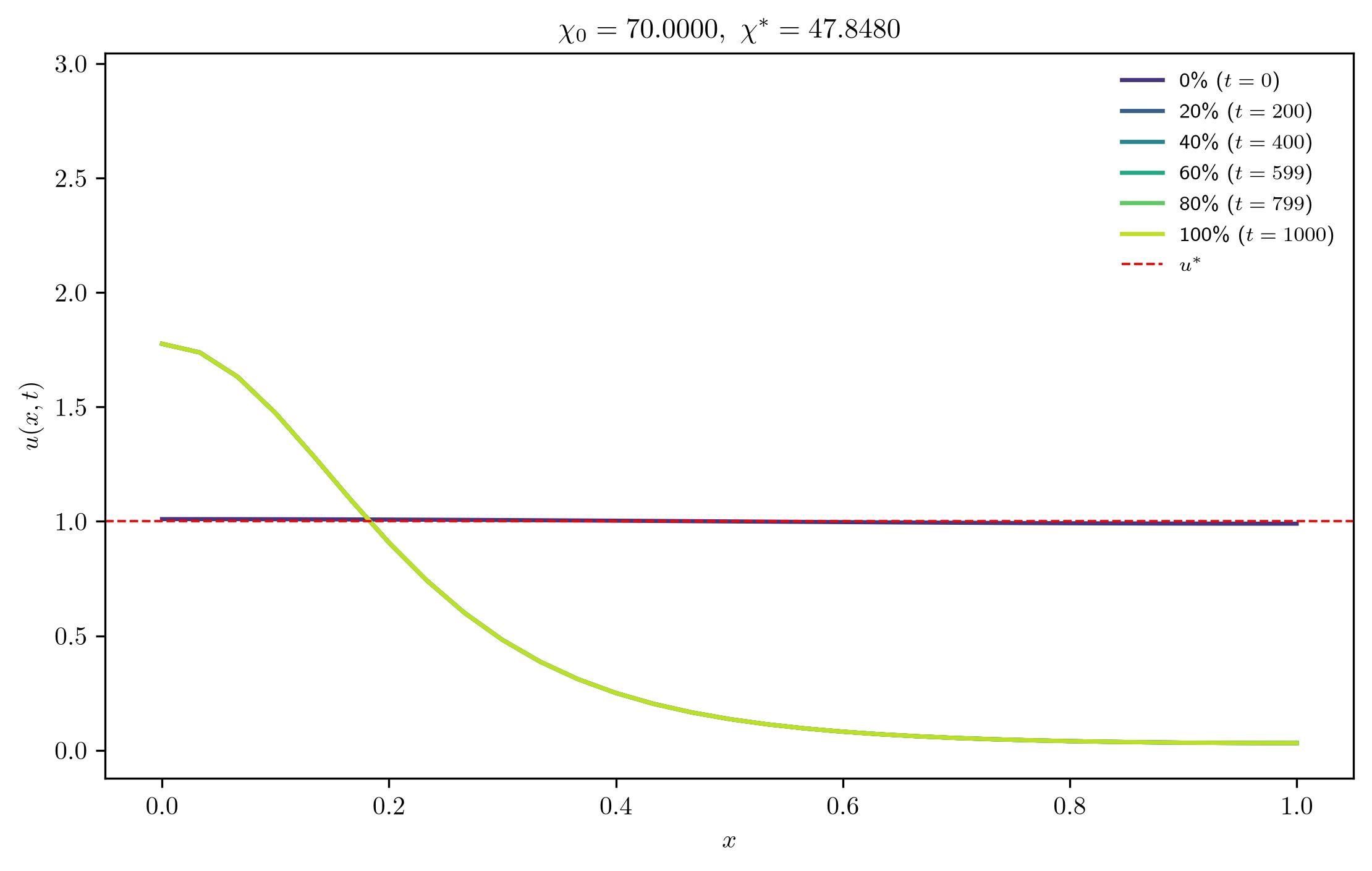}
  \end{subfigure}
  \hfill
  \begin{subfigure}[b]{0.45\textwidth}
    \centering
    \includegraphics[width=\textwidth]{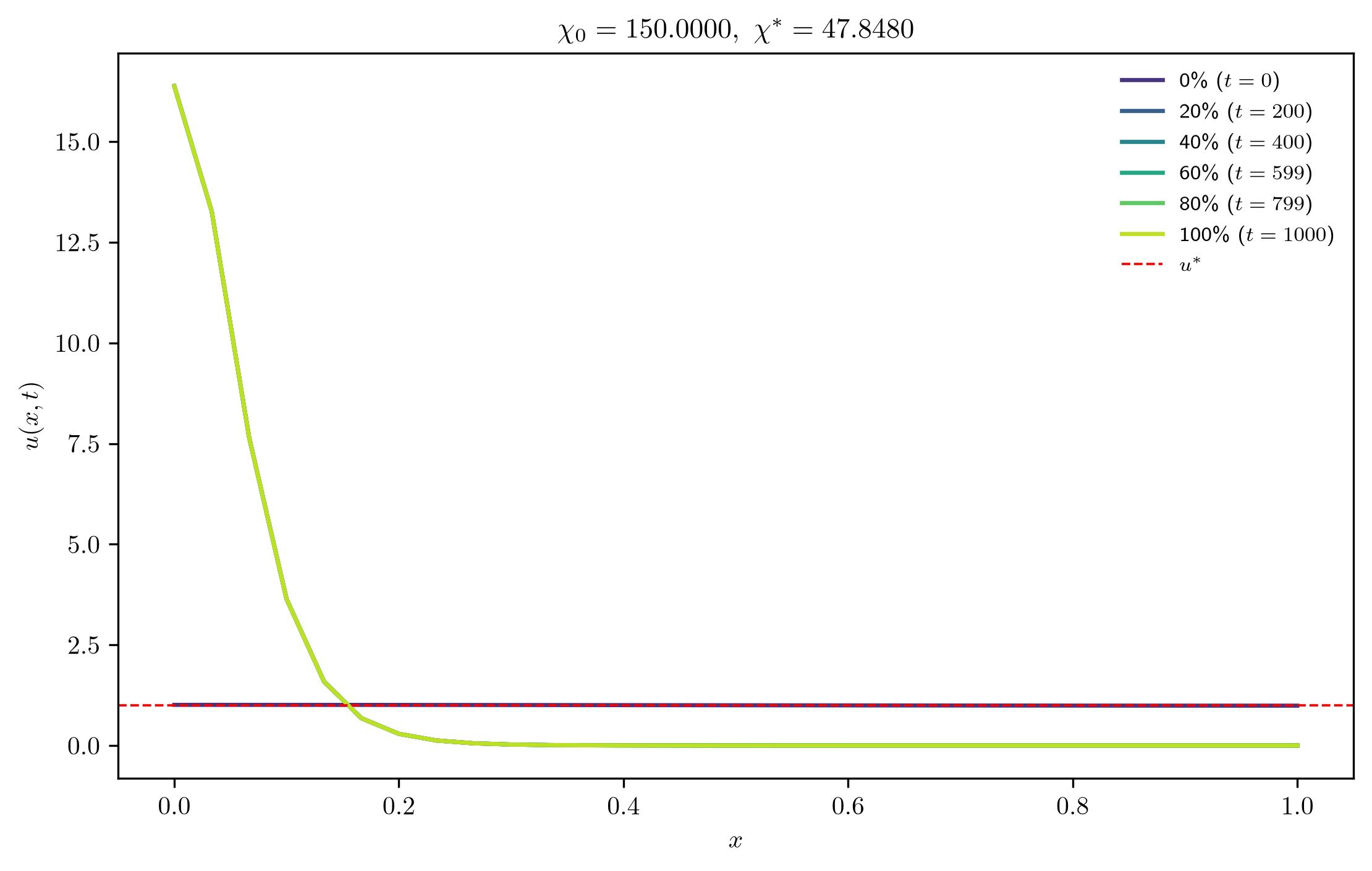}
  \end{subfigure}

  \caption{Numerical experiments for the parameters $a = b = \alpha = m = \mu =
  \nu = \gamma = L = 1$, $\beta = 2$, $\varepsilon = 0.01$, mesh size $N =
  30$, and final time $t = 1000$, $u_0 = 1 + 0.01\cos(\pi x)$, $\chi_0 \in
  \{47.846, \ 47.9, \ 70, \ 150\}$, and $\chi_{30}^{*,{\rm disc}} = 47.8480$.}

  \label{F:beta2-super-non-minimal}
\end{figure}

\begin{figure}[htbp]
  \centering
  \begin{subfigure}[b]{0.45\textwidth}
    \centering
    \includegraphics[width=\textwidth]{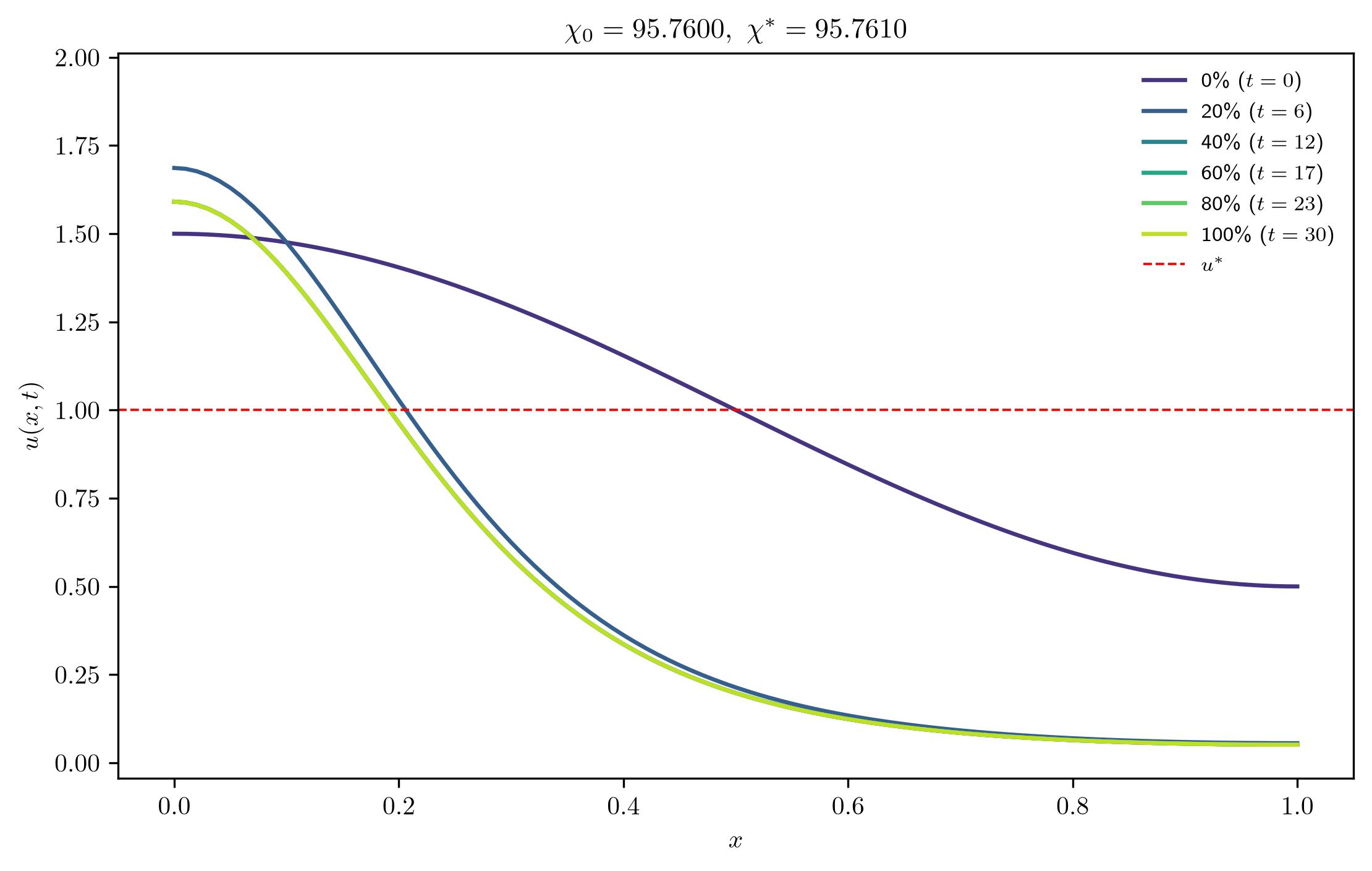}
  \end{subfigure}
  \hfill
  \begin{subfigure}[b]{0.45\textwidth}
    \centering
    \includegraphics[width=\textwidth]{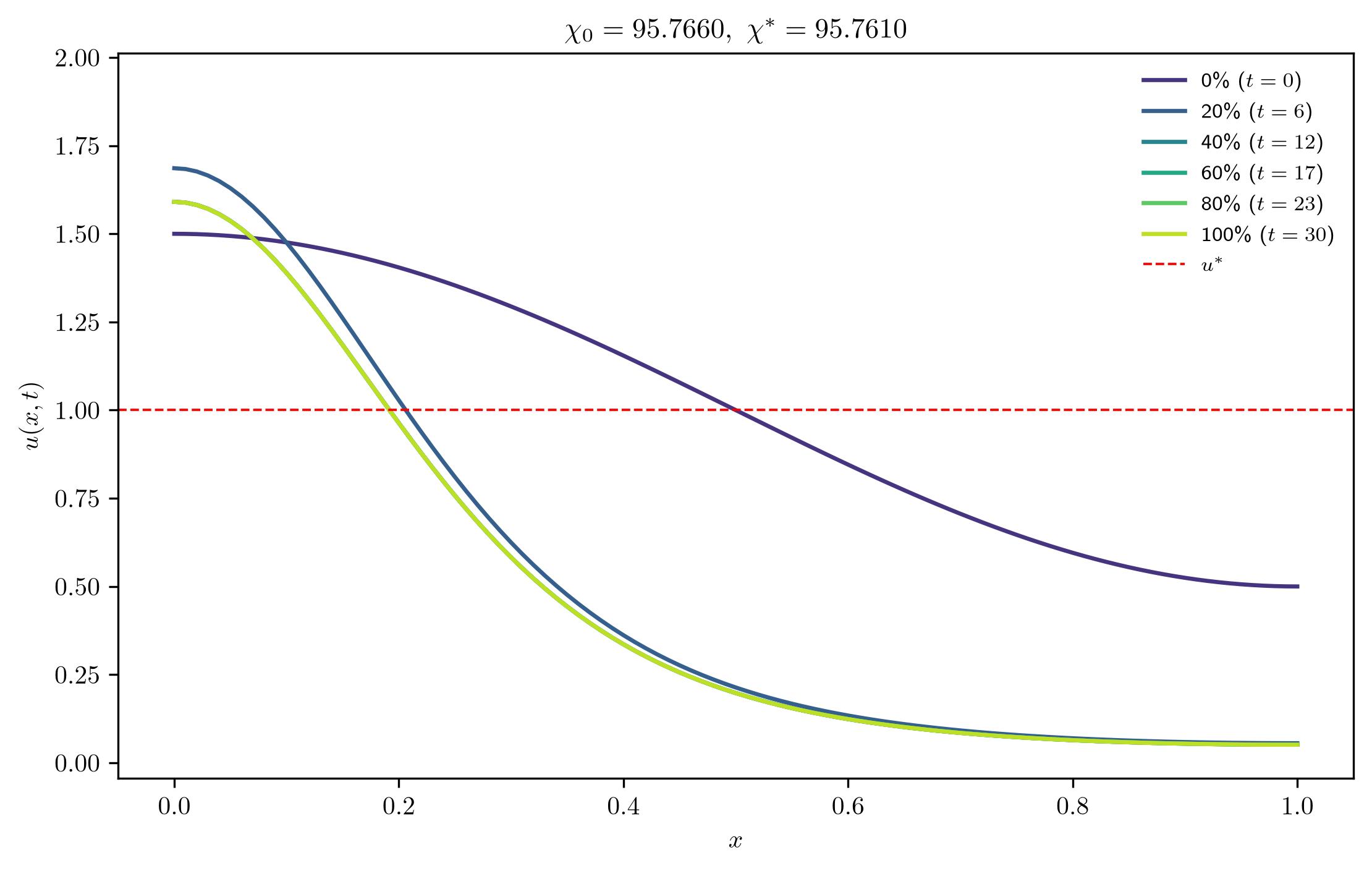}
  \end{subfigure}
  \hfill
  \begin{subfigure}[b]{0.45\textwidth}
    \centering
    \includegraphics[width=\textwidth]{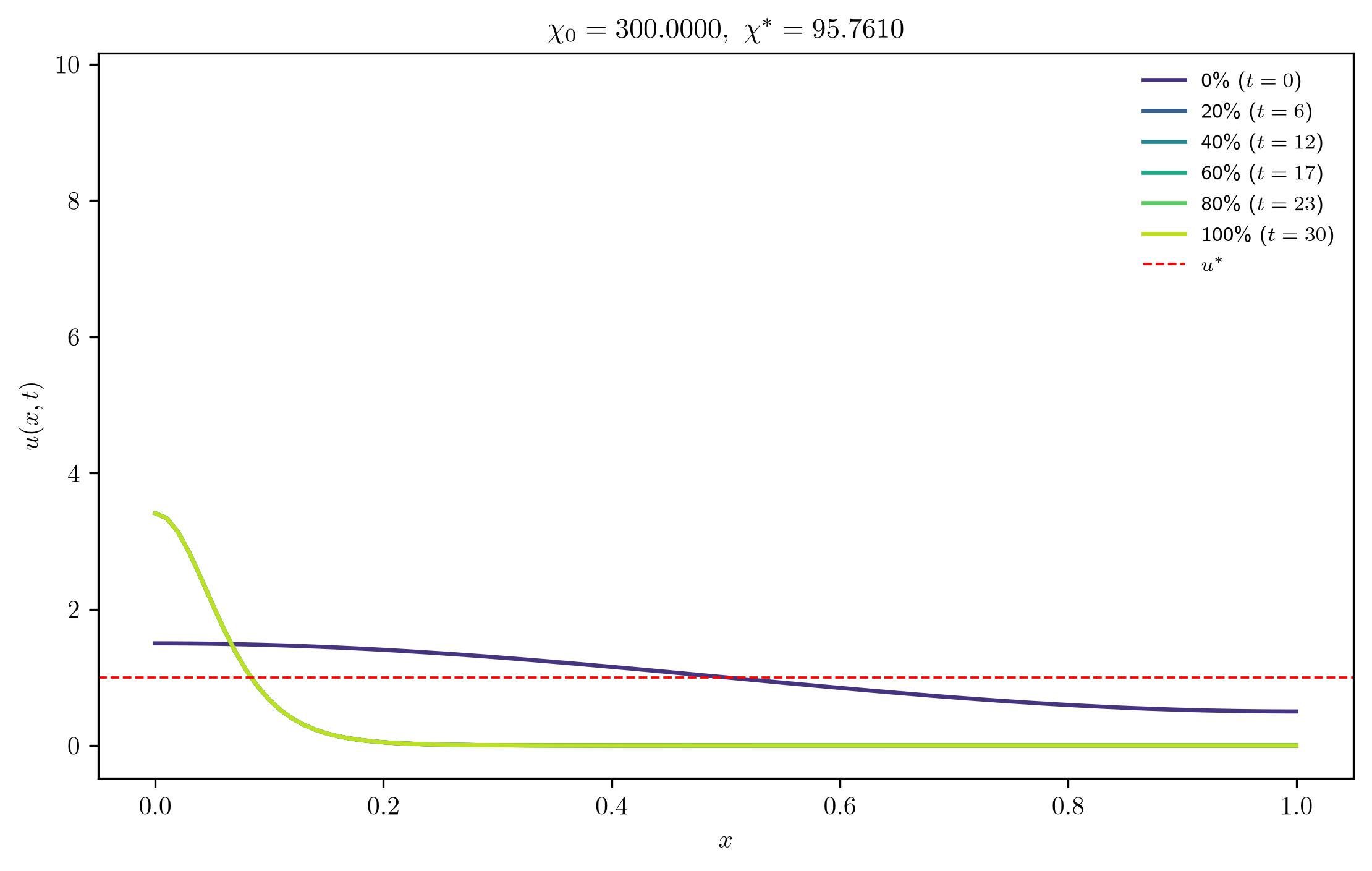}
  \end{subfigure}
  \hfill
  \begin{subfigure}[b]{0.45\textwidth}
    \centering
    \includegraphics[width=\textwidth]{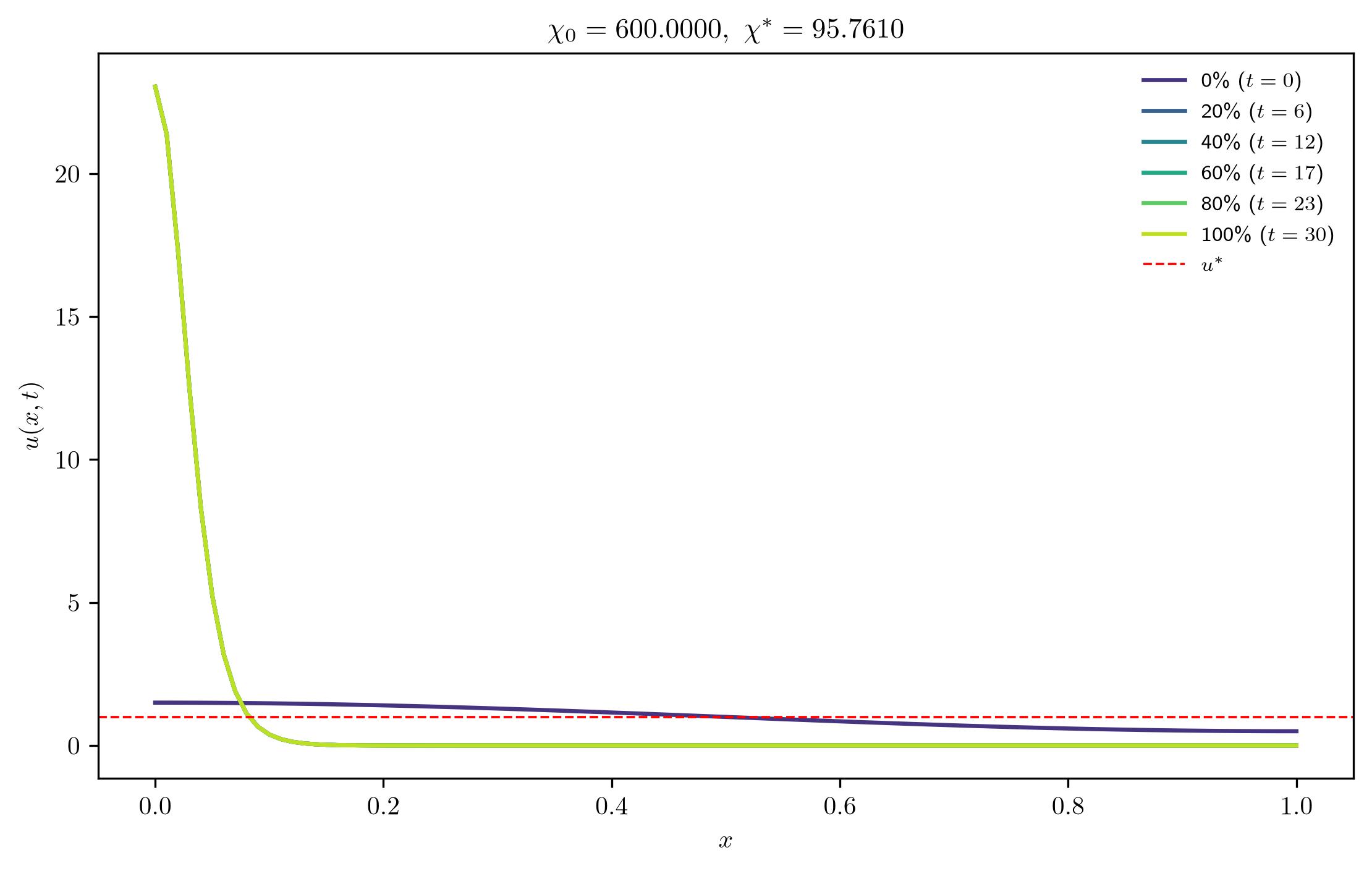}
  \end{subfigure}

  \caption{Numerical experiments for the parameters $a = b = \alpha = m = \mu =
  \nu = \gamma = L = 1$, $\beta = 3$, $\varepsilon = 0.5$, mesh size $N =
  100$, and final time $t = 30$, $u_0 = 1 + 0.5\cos(\pi x)$, and $\chi_0 \in
  \{95.76, \ 95.766, \ 300, \ 600\}$, and $\chi_{100}^{*,{\rm disc}} =
  95.7610$.}

  \label{F:beta3-sub-non-minimal}
\end{figure}

The first rows of Figures~\ref{F:beta4-super-minimal}
and~\ref{F:beta5-sub-minimal} show near-threshold experiments for the minimal
model with $\beta = 4$ and $\beta = 5$, respectively. These choices correspond
to the first block in the right panel (minimal model) of
Table~\ref{Tb:bifurcation-directions} and lie on opposite sides of the
transition $\beta_+\approx4.3692$. Thus, Figure~\ref{F:beta4-super-minimal} uses
parameters in the super-critical regime, whereas
Figure~\ref{F:beta5-sub-minimal} uses parameters in the sub-critical regime.

\begin{figure}[htbp]
  \centering
  \begin{subfigure}[b]{0.45\textwidth}
    \centering
    \includegraphics[width=\textwidth]{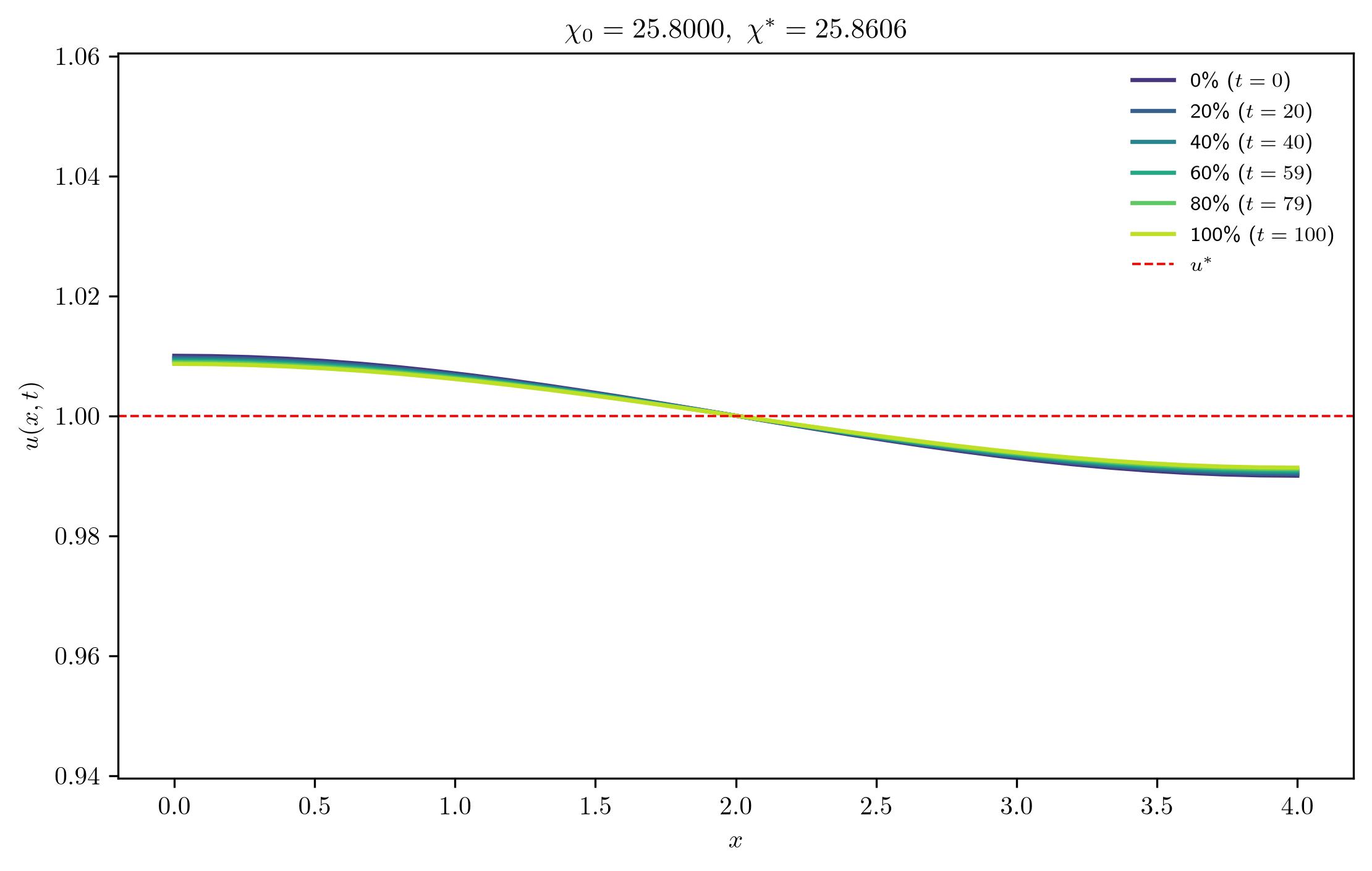}
  \end{subfigure}
  \hfill
  \begin{subfigure}[b]{0.45\textwidth}
    \centering
    \includegraphics[width=\textwidth]{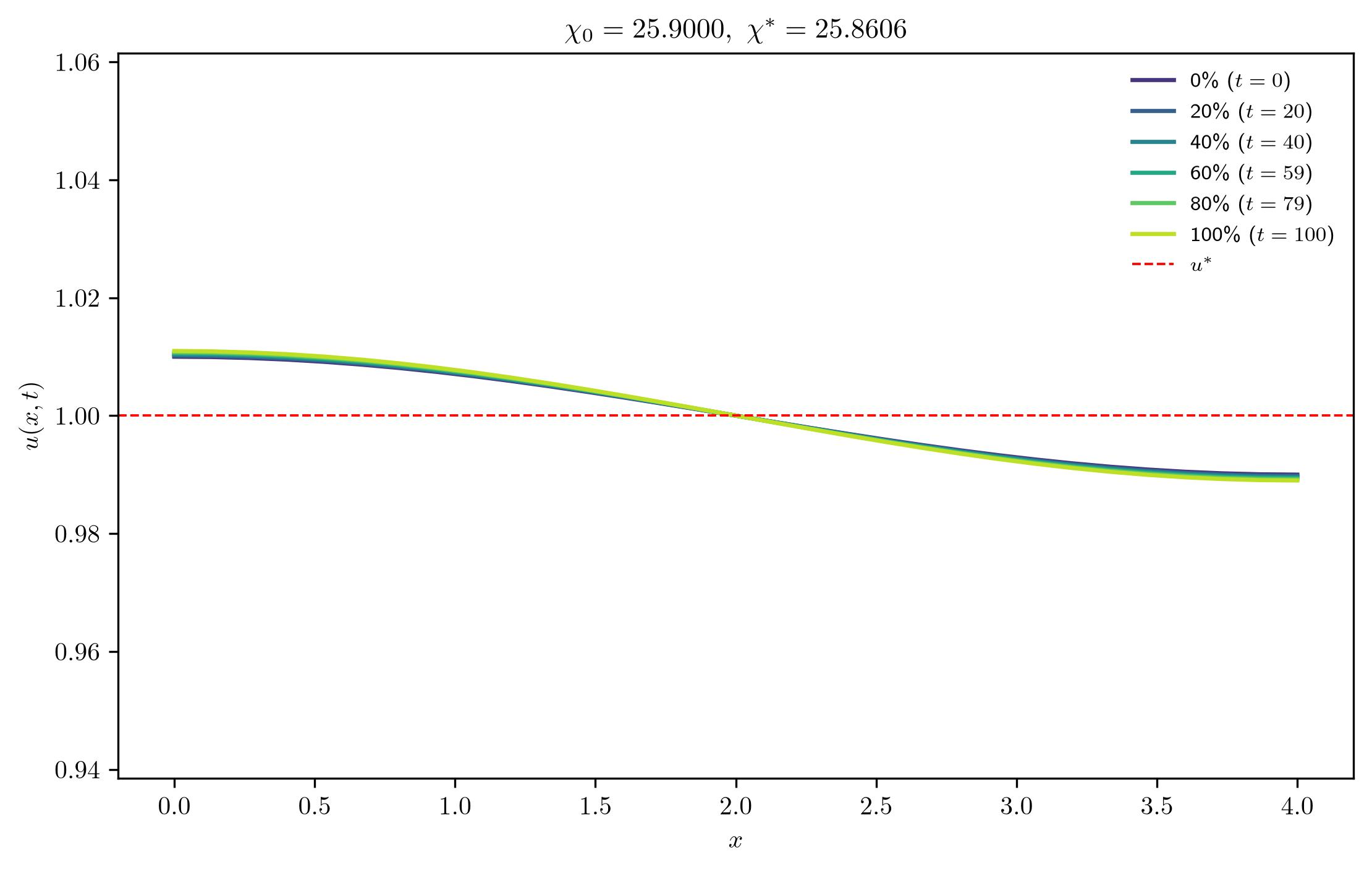}
  \end{subfigure}
  \hfill
  \begin{subfigure}[b]{0.45\textwidth}
    \centering
    \includegraphics[width=\textwidth]{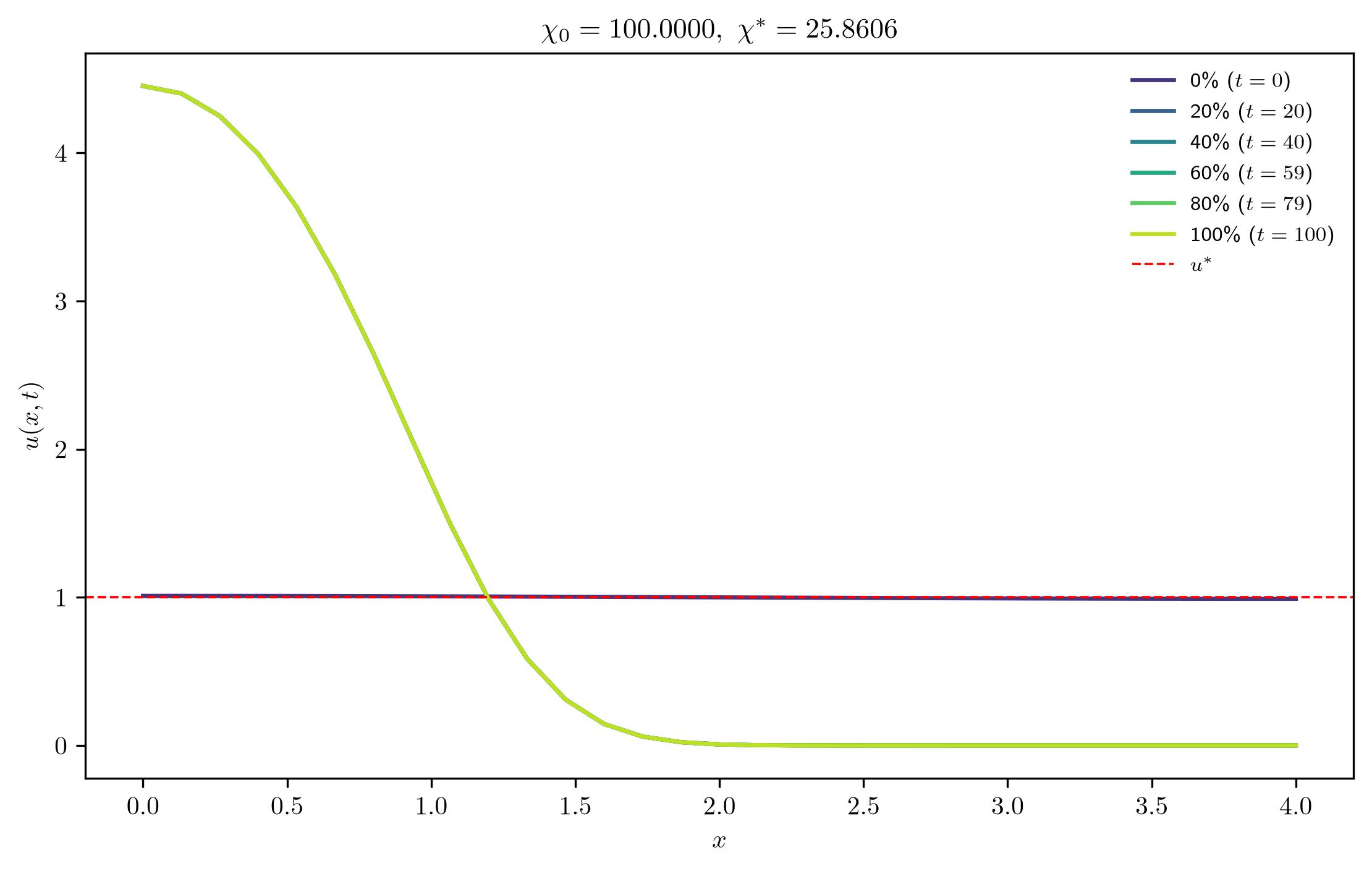}
  \end{subfigure}
  \hfill
  \begin{subfigure}[b]{0.45\textwidth}
    \centering
    \includegraphics[width=\textwidth]{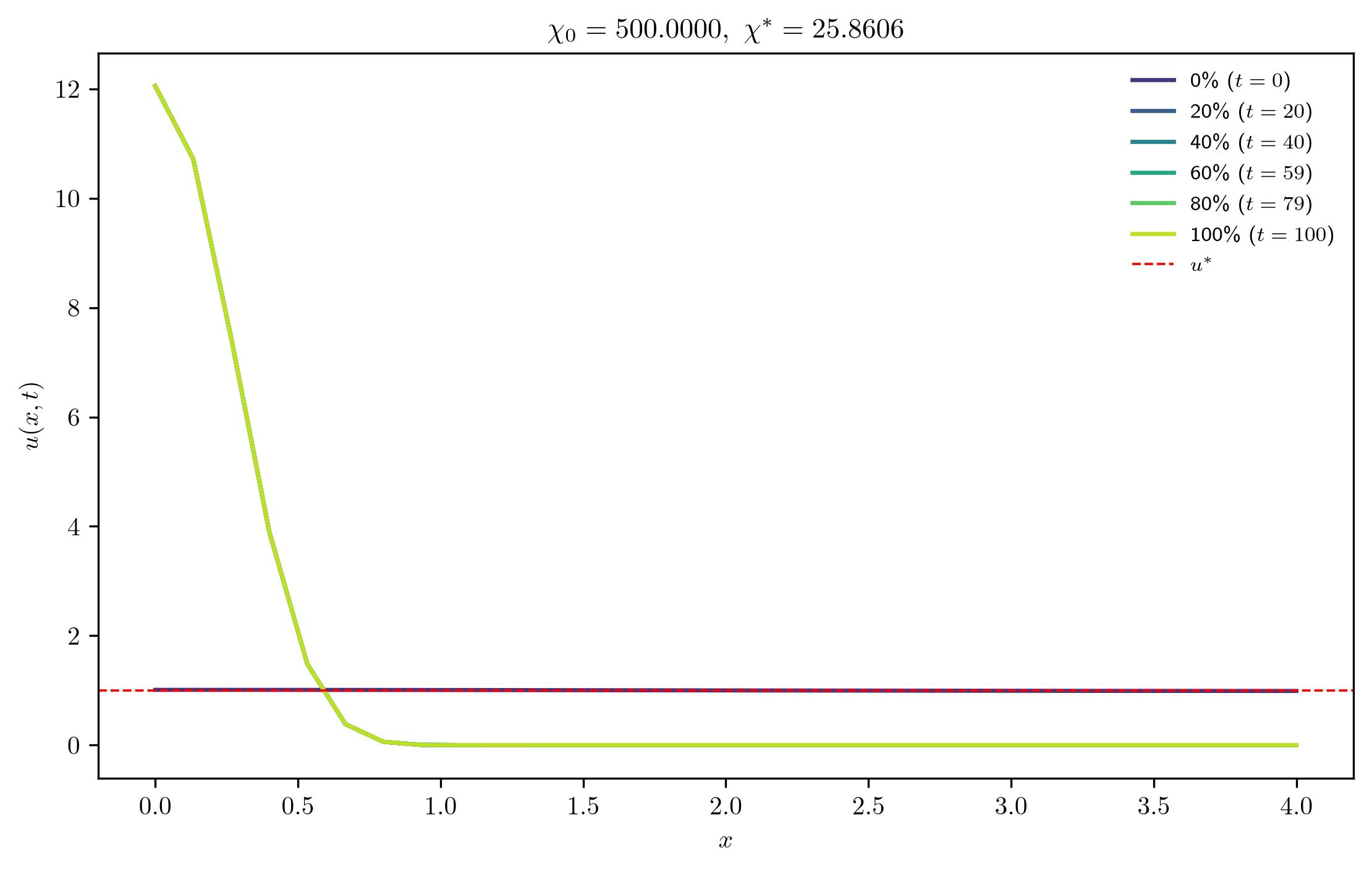}
  \end{subfigure}

  \caption{Numerical experiments for the parameters $a = b = 0$, $u^* = 1$,
  $\alpha = m = \mu = \nu = \gamma = 1$, $L = 4$, $\beta = 4$, $\varepsilon =
  0.01$, mesh size $N = 30$, and final time $t = 100$, $u_0 = 1 +
  0.01\cos(\frac{\pi x}{4})$, and $\chi_0 \in \{25.8, \ 25.9, \ 100, \ 500\}$,
  and $\chi_{30}^{*,{\rm disc}} = 25.8606$.}

  \label{F:beta4-super-minimal}
\end{figure}

\begin{figure}[htbp]
  \centering
  \begin{subfigure}[b]{0.45\textwidth}
    \centering
    \includegraphics[width=\textwidth]{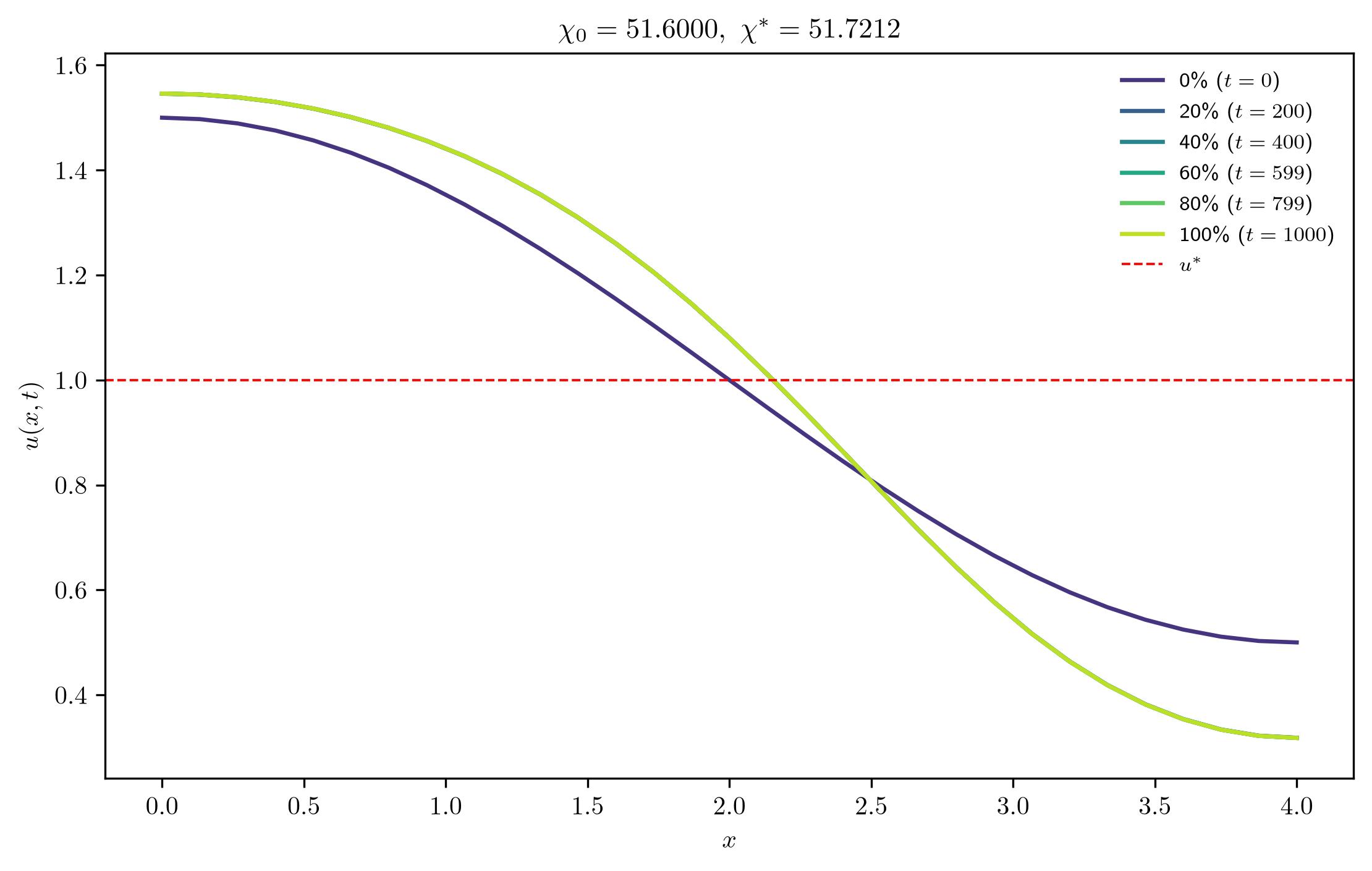}
  \end{subfigure}
  \hfill
  \begin{subfigure}[b]{0.45\textwidth}
    \centering
    \includegraphics[width=\textwidth]{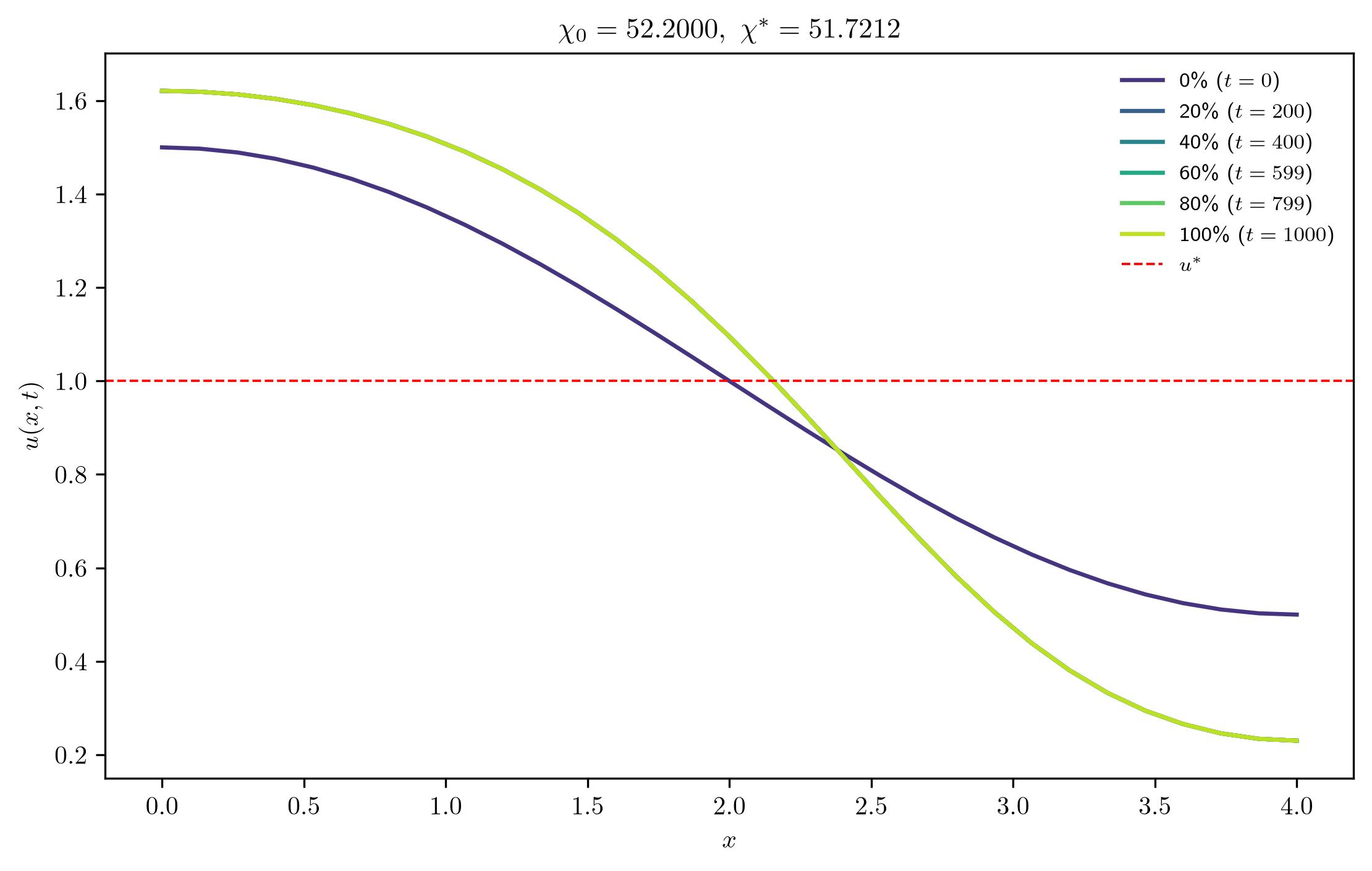}
  \end{subfigure}
  \hfill
  \begin{subfigure}[b]{0.45\textwidth}
    \centering
    \includegraphics[width=\textwidth]{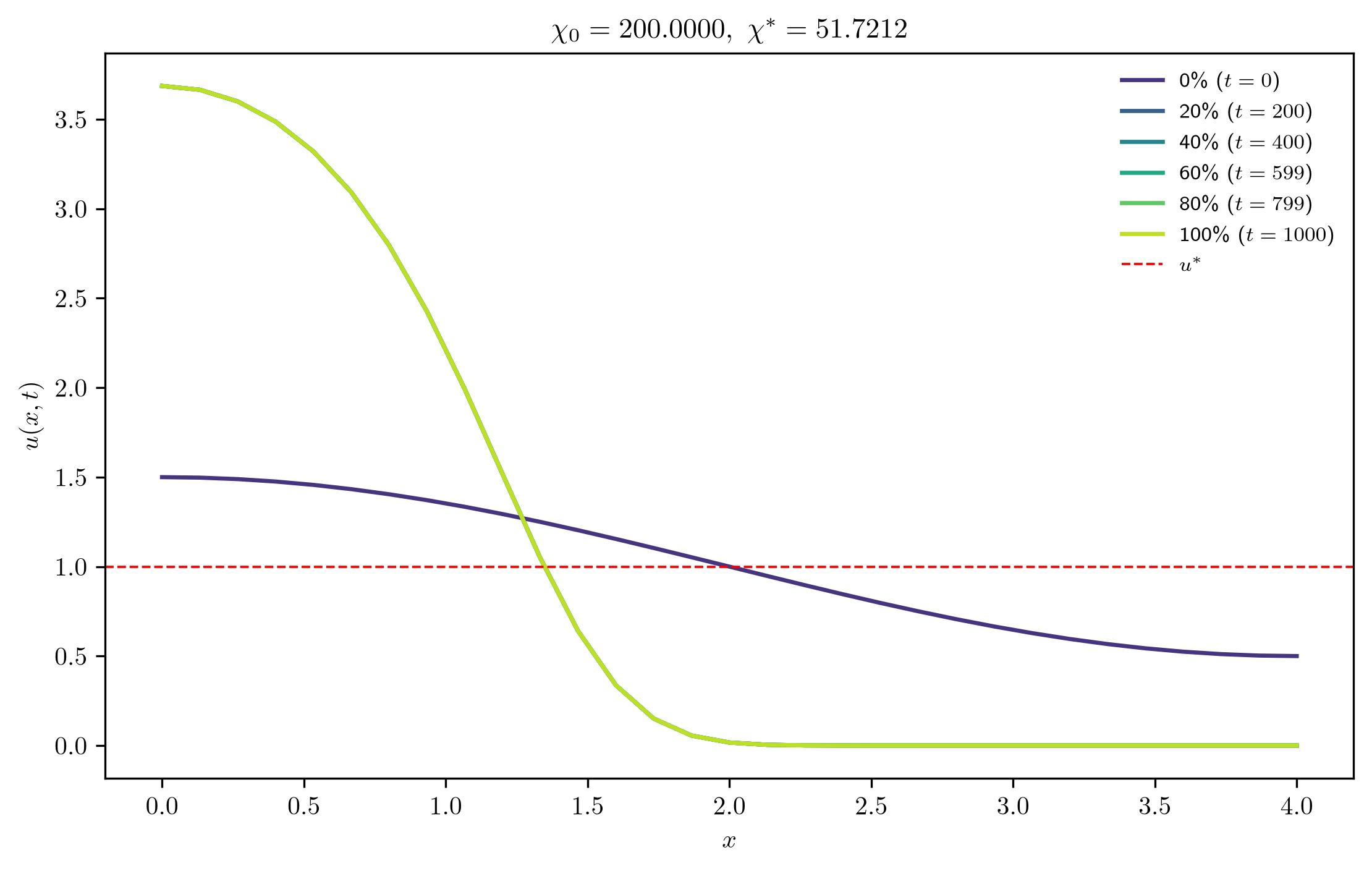}
  \end{subfigure}
  \hfill
  \begin{subfigure}[b]{0.45\textwidth}
    \centering
    \includegraphics[width=\textwidth]{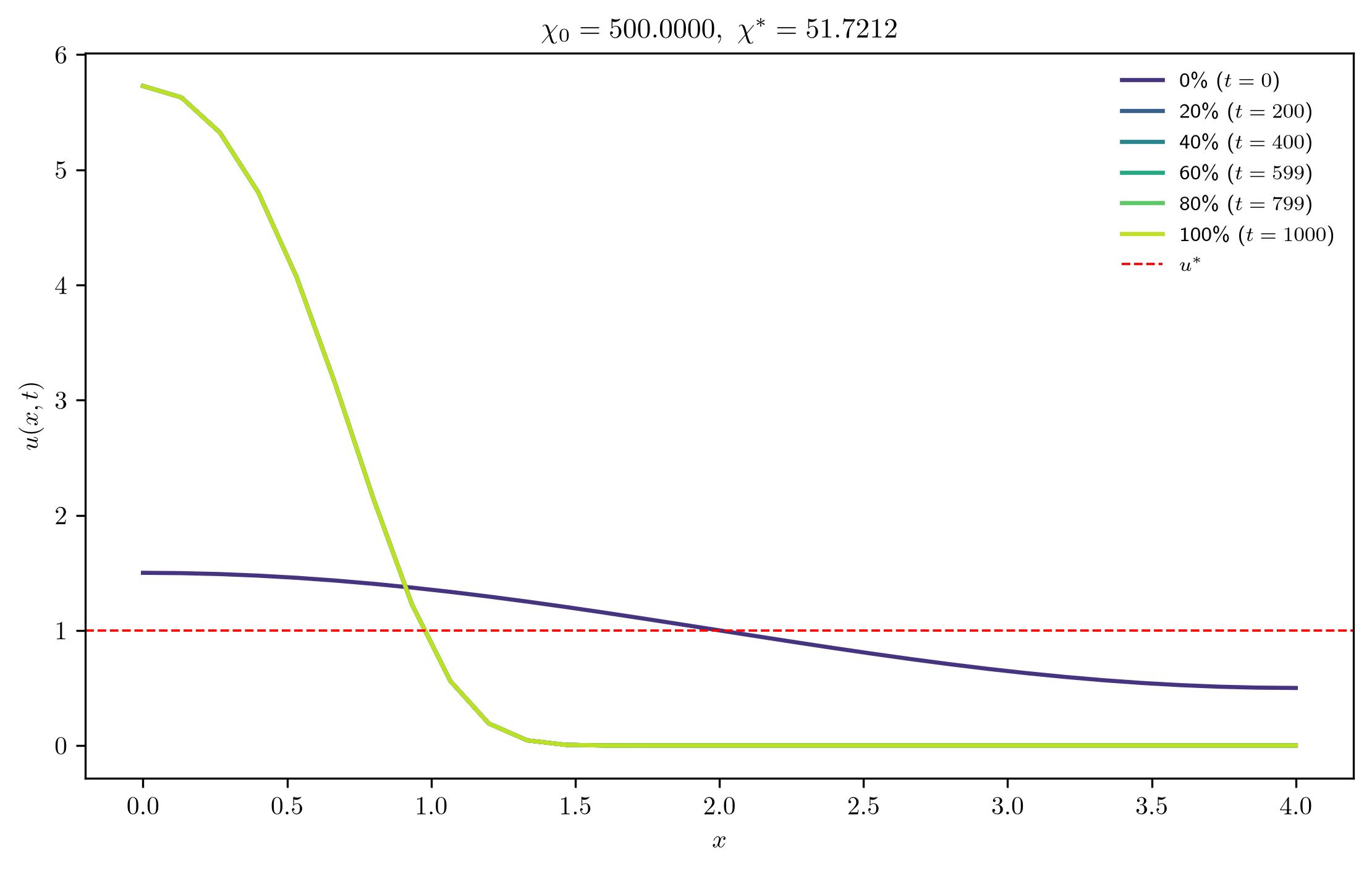}
  \end{subfigure}

  \caption{Numerical experiments for the parameters $a = b = 0$, $u^* = 1$,
  $\alpha = m = \mu = \nu = \gamma = 1$, $L = 4$, $\beta = 5$, $\varepsilon =
  0.5$, mesh size $N = 30$, and final time $t = 1000$, $u_0 = 1 +
  0.5\cos(\frac{\pi x}{4})$, and $\chi_0 \in \{51.6, \ 52.2, \ 200, \ 500\}$,
  and $\chi_{30}^{*,{\rm disc}} = 51.7212$.}

  \label{F:beta5-sub-minimal}
\end{figure}

From these numerical experiments, we observe the following:

\begin{itemize}
  \item Nonconstant profiles persist as $\chi_0$ increases, which suggests that
    the local bifurcation branches for $\chi_0$ near the threshold $\chi^*$
    extend to $\chi_0 \to \infty$, consistent with the unbounded alternative in
    Theorem~\ref{T:global-extension}(1).

  \item The maxima of the nonconstant profiles stay above $u^*$ in the tested
    regimes, where $u^* = (a/b)^{1/\alpha}$ for the non-minimal model and is the
    prescribed average for the minimal model, which suggests that there are
    non-constant stationary solutions satisfying the conditions in
    Theorem~\ref{T:spiky-solu-thm}, and then spike-like solutions exist as
    $\chi_0 \to \infty$.
\end{itemize}

\section*{Acknowledgments}
\addcontentsline{toc}{section}{Acknowledgments} L.~C. and I.~R. were partially
supported by the U.S. National Science Foundation under Grants DMS-2246850 and
DMS-2443823. L.~C. was also partially supported by Simons Foundation
Collaboration Grant for Mathematicians 959981. L. C. and I. R. thank OpenAI for
providing access to its AI tools through the ChatGPT for Academic Researchers
program.

\subsection*{Data and code availability}
The code and figures for the numerical experiments are available at
\url{https://github.com/ianruau/Chemotaxis_simulation}. All experiments were
performed using \texttt{simulation.py}. The script
\texttt{paper\_3\_numerical\_experiments.sh} reproduces the numerical
experiments and generates the figures stored in
\texttt{paper\_iii\_numerical\_experiments} that were used in this paper.

A Lean 4 companion formalizing selected results and conditional deductions is
available at \url{https://chenle02.github.io/chemotaxis-paper-iii-lean/}.
It documents the verified statements, their assumptions, and the remaining
exclusions; it does not constitute a formal verification of the entire paper.

\subsection*{AI declaration}
The authors manually derived all the expressions involved in  the proofs of
Theorems  \ref{T:local-bifurcation} and \ref{T:direction-change}  in the paper.
After the authors' manual computation of the expressions
in~\eqref{E:c-constant}, \eqref{E:alpha_n0}, \eqref{E:beta-n0},
and~\eqref{E:A-B-C-simple}, OpenAI's GPT models were used to assess that our
hand-derived expressions matched the models' expressions.

\bibliographystyle{amsrn}
\phantomsection
\addcontentsline{toc}{section}{References}
\bibliography{references}

@Article{	  baghaei:24:global,
  author	= {Baghaei, Khadijeh},
  title		= {Global boundedness of solutions to a chemotaxis
		  consumption model with signal dependent motility and
		  logistic source},
  journal	= {C. R. Math. Acad. Sci. Paris},
  fjournal	= {Comptes Rendus Math\'{e}matique. Acad\'{e}mie des
		  Sciences. Paris},
  volume	= {362},
  year		= {2024},
  number	= {G10},
  pages		= {1131--1145},
  issn		= {1631-073X},
  mrclass	= {35K51 (35K57 92C17)},
  mrnumber	= {4824914},
  doi		= {10.5802/crmath.605},
  url		= {https://doi.org/10.5802/crmath.605}
}

@Article{	  black.lankeit.ea:20:stabilization,
  author	= {Black, Tobias and Lankeit, Johannes and Mizukami,
		  Masaaki},
  title		= {Stabilization in the {K}eller-{S}egel system with
		  signal-dependent sensitivity},
  journal	= {Appl. Anal.},
  fjournal	= {Applicable Analysis. An International Journal},
  volume	= {99},
  year		= {2020},
  number	= {16},
  pages		= {2877--2891},
  issn		= {0003-6811},
  mrclass	= {35B40 (35K51 35Q92 92C17)},
  mrnumber	= {4174428},
  doi		= {10.1080/00036811.2019.1585534},
  url		= {https://doi.org/10.1080/00036811.2019.1585534}
}

@article {chen.hao.ea:14:stability,
    author = {Chen, Xinfu and Hao, Jianghao and Wang, Xuefeng and Wu, Yaping
              and Zhang, Yajing},
     title = {Stability of spiky solution of {K}eller-{S}egel's minimal
              chemotaxis model},
   journal = {J. Differential Equations},
  fjournal = {Journal of Differential Equations},
    volume = {257},
      year = {2014},
    number = {9},
     pages = {3102--3134},
      issn = {0022-0396,1090-2732},
   mrclass = {35K51 (35B35 35C20)},
  mrnumber = {3258133},
mrreviewer = {Jinhuan\ Wang},
       doi = {10.1016/j.jde.2014.06.008},
       url = {https://doi.org/10.1016/j.jde.2014.06.008},
}

@Article{	  chen.ruau.ea:26:chemotaxis*1,
  author	= {Chen, Le and Ruau, Ian and Shen, Wenxian},
  title		= {Chemotaxis models with signal-dependent sensitivity and a
		  logistic-type source, {I}: {B}oundedness and global
		  existence},
  journal	= {J. Differential Equations},
  fjournal	= {Journal of Differential Equations},
  volume	= {476},
  year		= {2026},
  pages		= {Paper No. 114477, 54},
  issn		= {0022-0396,1090-2732},
  mrclass	= {35K45 (35M31 35Q92 92C17 92D25)},
  mrnumber	= {5070860},
  doi		= {10.1016/j.jde.2026.114477},
  eprint	= {2512.14858},
  archiveprefix	= {arXiv},
  primaryclass	= {math.AP},
  url		= {https://doi.org/10.1016/j.jde.2026.114477}
}

@Article{	  chen.ruau.ea:26:chemotaxis,
  title		= {Chemotaxis models with signal-dependent sensitivity and a
		  logistic-type source, {II}: Persistence and stabilization},
  author	= {Le Chen and Ian Ruau and Wenxian Shen},
  year		= {2026},
  journal	= {Preprint arXiv:2604.02599v1},
  url		= {https://arxiv.org/abs/2604.02599v1}
}

@Article{	  crandall.rabinowitz:71:bifurcation,
  author	= {Crandall, Michael G. and Rabinowitz, Paul H.},
  title		= {Bifurcation from simple eigenvalues},
  journal	= {J. Functional Analysis},
  fjournal	= {Journal of Functional Analysis},
  volume	= {8},
  year		= {1971},
  pages		= {321--340},
  issn		= {0022-1236},
  mrclass	= {47.80 (35.00)},
  mrnumber	= {288640},
  mrreviewer	= {B.\ V.\ Loginov},
  doi		= {10.1016/0022-1236(71)90015-2},
  url		= {https://doi.org/10.1016/0022-1236(71)90015-2}
}

@Article{	  galakhov.salieva.ea:16:on,
  author	= {Galakhov, Evgeny and Salieva, Olga and Tello, J. Ignacio},
  title		= {On a parabolic-elliptic system with chemotaxis and
		  logistic type growth},
  journal	= {J. Differential Equations},
  fjournal	= {Journal of Differential Equations},
  volume	= {261},
  year		= {2016},
  number	= {8},
  pages		= {4631--4647},
  issn		= {0022-0396},
  mrclass	= {92C17 (35A01 35B35 35B40 35B44 35K59 35M33 35Q92)},
  mrnumber	= {3537839},
  doi		= {10.1016/j.jde.2016.07.008},
  url		= {https://doi.org/10.1016/j.jde.2016.07.008}
}

@Book{		  henry:81:geometric,
  author	= {Henry, Daniel},
  title		= {Geometric theory of semilinear parabolic equations},
  series	= {Lecture Notes in Mathematics},
  volume	= {840},
  publisher	= {Springer-Verlag, Berlin-New York},
  year		= {1981},
  pages		= {iv+348},
  isbn		= {3-540-10557-3},
  mrclass	= {35K55 (34G20 58D25)},
  mrnumber	= {610244},
  mrreviewer	= {J. A. Goldstein}
}

@Article{	  herrero.medina.ea:97:finite-time,
  author	= {Herrero, M. A. and Medina, E. and Vel\'{a}zquez, J. J. L.},
  title		= {Finite-time aggregation into a single point in a
		  reaction-diffusion system},
  journal	= {Nonlinearity},
  fjournal	= {Nonlinearity},
  volume	= {10},
  year		= {1997},
  number	= {6},
  pages		= {1739--1754},
  issn		= {0951-7715},
  mrclass	= {35K57 (35B40 93B05)},
  mrnumber	= {1483563},
  doi		= {10.1088/0951-7715/10/6/016},
  url		= {https://doi.org/10.1088/0951-7715/10/6/016}
}

@Article{	  herrero.velazquez:96:singularity,
  author	= {Herrero, Miguel A. and Vel\'{a}zquez, Juan J. L.},
  title		= {Singularity patterns in a chemotaxis model},
  journal	= {Math. Ann.},
  fjournal	= {Mathematische Annalen},
  volume	= {306},
  year		= {1996},
  number	= {3},
  pages		= {583--623},
  issn		= {0025-5831},
  mrclass	= {35Q99 (35B05 35K55 92C40)},
  mrnumber	= {1415081},
  mrreviewer	= {Joseph G.-G. Yan},
  doi		= {10.1007/BF01445268},
  url		= {https://doi.org/10.1007/BF01445268}
}

@Article{	  hu.tao:17:boundedness,
  author	= {Hu, Bingran and Tao, Youshan},
  title		= {Boundedness in a parabolic-elliptic chemotaxis-growth
		  system under a critical parameter condition},
  journal	= {Appl. Math. Lett.},
  fjournal	= {Applied Mathematics Letters. An International Journal of
		  Rapid Publication},
  volume	= {64},
  year		= {2017},
  pages		= {1--7},
  issn		= {0893-9659},
  mrclass	= {92C17 (35J91 35K59 35M33 35Q92)},
  mrnumber	= {3564732},
  doi		= {10.1016/j.aml.2016.08.003},
  url		= {https://doi.org/10.1016/j.aml.2016.08.003}
}

@Article{	  jin.wang.ea:16:pattern,
  title		= {Pattern formation in {K}eller--{S}egel chemotaxis models
		  with logistic growth},
  author	= {Ling Jin and Qi Wang and Zengyan Zhang},
  year		= {2016},
  journal	= {Internat. J. Bifur. Chaos Appl. Sci. Engrg.},
  volume	= {26},
  number	= {2},
  pages		= {1650033, 15},
  doi		= {10.1142/S0218127416500334},
  eprint	= {1407.5246},
  archiveprefix	= {arXiv},
  primaryclass	= {math.AP},
  url		= {https://arxiv.org/abs/1407.5246}
}

@Article{	  keller.segel:70:initiation,
  author	= {Keller, Evelyn F. and Segel, Lee A.},
  title		= {Initiation of slime mold aggregation viewed as an
		  instability},
  journal	= {J. Theoret. Biol.},
  fjournal	= {Journal of Theoretical Biology},
  volume	= {26},
  year		= {1970},
  number	= {3},
  pages		= {399--415},
  issn		= {0022-5193},
  mrclass	= {92C17 (35Q92 82C24 92C45)},
  mrnumber	= {3925816},
  doi		= {10.1016/0022-5193(70)90092-5},
  url		= {https://doi.org/10.1016/0022-5193(70)90092-5}
}

@Article{	  keller.segel:71:model,
  title		= {Model for chemotaxis},
  author	= {Keller, Evelyn F and Segel, Lee A},
  fjournal	= {Journal of theoretical biology},
  journal	= {J. Theoret. Biol.},
  volume	= {30},
  number	= {2},
  pages		= {225--234},
  year		= {1971},
  publisher	= {Elsevier},
  doi		= {10.1016/0022-5193(71)90050-6}
}

@Article{	  keller.segel:71:traveling,
  title		= {Traveling bands of chemotactic bacteria: a theoretical
		  analysis},
  author	= {Keller, Evelyn F and Segel, Lee A},
  fjournal	= {Journal of theoretical biology},
  journal	= {J. Theoret. Biol.},
  volume	= {30},
  number	= {2},
  pages		= {235--248},
  year		= {1971},
  publisher	= {Elsevier},
  doi		= {10.1016/0022-5193(71)90051-8}
}

@article {kong.wang:22:stability,
    author = {Kong, Fanze and Wang, Qi},
     title = {Stability, free energy and dynamics of multi-spikes in the
              minimal {K}eller-{S}egel model},
   journal = {Discrete Contin. Dyn. Syst.},
  fjournal = {Discrete and Continuous Dynamical Systems},
    volume = {42},
      year = {2022},
    number = {5},
     pages = {2499--2523},
      issn = {1078-0947,1553-5231},
   mrclass = {92C17 (35K57 35Q92)},
  mrnumber = {4405278},
       doi = {10.3934/dcds.2021200},
       url = {https://doi.org/10.3934/dcds.2021200},
}

@article {kong.ward.ea:24:existence,
    author = {Kong, Fanze and Ward, Michael J. and Wei, Juncheng},
     title = {Existence, stability and slow dynamics of spikes in a 1{D}
              minimal {K}eller-{S}egel model with logistic growth},
   journal = {J. Nonlinear Sci.},
  fjournal = {Journal of Nonlinear Science},
    volume = {34},
      year = {2024},
    number = {3},
     pages = {Paper No. 51, 76},
      issn = {0938-8974,1432-1467},
   mrclass = {35K51 (35B20 92C17)},
  mrnumber = {4728829},
       doi = {10.1007/s00332-024-10025-7},
       url = {https://doi.org/10.1007/s00332-024-10025-7},
}

@article {kong.wei.ea:23:existence,
    author = {Kong, Fanze and Wei, Juncheng and Xu, Liangshun},
     title = {The existence and stability of spikes in the one-dimensional
              {K}eller-{S}egel model with logistic growth},
   journal = {J. Math. Biol.},
  fjournal = {Journal of Mathematical Biology},
    volume = {86},
      year = {2023},
    number = {1},
     pages = {Paper No. 6, 54},
      issn = {0303-6812,1432-1416},
   mrclass = {92C17 (35B40 35K55 35Q92)},
  mrnumber = {4519183},
mrreviewer = {Qiao\ Xin},
       doi = {10.1007/s00285-022-01840-1},
       url = {https://doi.org/10.1007/s00285-022-01840-1},
}

@Article{	  kurt.shen.ea:24:stability,
  title		= {Stability, bifurcation and spikes of stationary solutions
		  in a chemotaxis system with singular sensitivity and
		  logistic source},
  author	= {Halil Ibrahim Kurt and Wenxian Shen and Shuwen Xue},
  year		= {2024},
  journal	= {Mathematical Models and Methods in Applied Sciences},
  volume	= {34},
  number	= {9},
  pages		= {1649--1700},
  mrnumber      = {4780931},
  doi		= {10.1142/S0218202524500325},
  eprint	= {2403.14907v1},
  archiveprefix	= {arXiv},
  primaryclass	= {math.AP},
  note		= {The theorem numbers cited in the text refer to arXiv:2403.14907v1},
  url		= {https://doi.org/10.1142/S0218202524500325}
}

@Article{	  kuto.osaki.ea:12:spatial,
  author	= {Kuto, Kousuke and Osaki, Koichi and Sakurai, Tatsunari and
		  Tsujikawa, Tohru},
  title		= {Spatial pattern formation in a chemotaxis-diffusion-growth
		  model},
  journal	= {Phys. D},
  fjournal	= {Physica D. Nonlinear Phenomena},
  volume	= {241},
  year		= {2012},
  number	= {19},
  pages		= {1629--1639},
  issn		= {0167-2789},
  mrclass	= {82C05},
  mrnumber	= {2965808},
  doi		= {10.1016/j.physd.2012.06.009},
  url		= {https://doi.org/10.1016/j.physd.2012.06.009}
}

@article {li:18:spiky,
    author = {Li, Huicong},
     title = {Spiky steady states of a chemotaxis system with singular
              sensitivity},
   journal = {J. Dynam. Differential Equations},
  fjournal = {Journal of Dynamics and Differential Equations},
    volume = {30},
      year = {2018},
    number = {4},
     pages = {1775--1795},
      issn = {1040-7294,1572-9222},
   mrclass = {92C17 (35B32 35B36 35B40 35J57 35K51)},
  mrnumber = {3871622},
mrreviewer = {Gabriela\ Marinoschi},
       doi = {10.1007/s10884-017-9621-3},
       url = {https://doi.org/10.1007/s10884-017-9621-3},
}

@Article{	  lin.ni.ea:88:large,
  author	= {Lin, C.-S. and Ni, W.-M. and Takagi, I.},
  title		= {Large amplitude stationary solutions to a chemotaxis
		  system},
  journal	= {J. Differential Equations},
  fjournal	= {Journal of Differential Equations},
  volume	= {72},
  year		= {1988},
  number	= {1},
  pages		= {1--27},
  issn		= {0022-0396,1090-2732},
  mrclass	= {35K55 (35J60 92A09)},
  mrnumber	= {929196},
  mrreviewer	= {Siegfried\ Carl},
  doi		= {10.1016/0022-0396(88)90147-7},
  url		= {https://doi.org/10.1016/0022-0396(88)90147-7}
}

@Article{	  nagai:01:blowup,
  author	= {Nagai, Toshitaka},
  title		= {Blowup of nonradial solutions to parabolic-elliptic
		  systems modeling chemotaxis in two-dimensional domains},
  journal	= {J. Inequal. Appl.},
  fjournal	= {Journal of Inequalities and Applications},
  volume	= {6},
  year		= {2001},
  number	= {1},
  pages		= {37--55},
  issn		= {1025-5834},
  mrclass	= {35B40 (35M20 92C17)},
  mrnumber	= {1887324},
  mrreviewer	= {Julio D. Rossi},
  doi		= {10.1155/S1025583401000042},
  url		= {https://doi.org/10.1155/S1025583401000042}
}

@Article{	  nagai.senba:98:global,
  author	= {Nagai, Toshitaka and Senba, Takasi},
  title		= {Global existence and blow-up of radial solutions to a
		  parabolic-elliptic system of chemotaxis},
  journal	= {Adv. Math. Sci. Appl.},
  fjournal	= {Advances in Mathematical Sciences and Applications},
  volume	= {8},
  year		= {1998},
  number	= {1},
  pages		= {145--156},
  issn		= {1343-4373},
  mrclass	= {35Q80 (35B05 35B60 35M10 92D50)},
  mrnumber	= {1623326},
  mrreviewer	= {Andreas M\"{u}ller-Rettkowski}
}

@Article{	  shi.wang:09:on,
  author	= {Shi, Junping and Wang, Xuefeng},
  title		= {On global bifurcation for quasilinear elliptic systems on
		  bounded domains},
  journal	= {J. Differential Equations},
  fjournal	= {Journal of Differential Equations},
  volume	= {246},
  year		= {2009},
  number	= {7},
  pages		= {2788--2812},
  issn		= {0022-0396,1090-2732},
  mrclass	= {35J60 (35B32 47H11 47J15 47N20 58E07)},
  mrnumber	= {2503022},
  mrreviewer	= {Jes\'{u}s\ Hern\'{a}ndez},
  doi		= {10.1016/j.jde.2008.09.009},
  url		= {https://doi.org/10.1016/j.jde.2008.09.009}
}

@Article{  shemtaga.shen.ea:23:stability,
  author  = {Shemtaga, Hewan and Shen, Wenxian and Sukhtaiev, Selim},
  title    = {Stability and bifurcation for logistic {K}eller--{S}egel models
      on compact graphs},
  journal  = {Pure Appl. Funct. Anal.},
  volume  = {9},
  number  = {5},
  pages    = {1359--1385},
  year    = {2024},
  url    = {https://arxiv.org/abs/2310.00756},
}

@Article{	  tello.winkler:07:chemotaxis,
  author	= {Tello, J. Ignacio and Winkler, Michael},
  title		= {A chemotaxis system with logistic source},
  journal	= {Comm. Partial Differential Equations},
  fjournal	= {Communications in Partial Differential Equations},
  volume	= {32},
  year		= {2007},
  number	= {4-6},
  pages		= {849--877},
  issn		= {0360-5302},
  mrclass	= {92C17 (35K57)},
  mrnumber	= {2334836},
  doi		= {10.1080/03605300701319003},
  url		= {https://doi.org/10.1080/03605300701319003}
}

@Article{	  viglialoro.woolley:18:boundedness,
  author	= {Viglialoro, Giuseppe and Woolley, Thomas E.},
  title		= {Boundedness in a parabolic-elliptic chemotaxis system with
		  nonlinear diffusion and sensitivity and logistic source},
  journal	= {Math. Methods Appl. Sci.},
  fjournal	= {Mathematical Methods in the Applied Sciences},
  volume	= {41},
  year		= {2018},
  number	= {5},
  pages		= {1809--1824},
  issn		= {0170-4214},
  mrclass	= {92C17 (35A01 35B40 35K59 35M33 35Q92)},
  mrnumber	= {3778089},
  mrreviewer	= {Si Ning Zheng},
  doi		= {10.1002/mma.4707},
  url		= {https://doi.org/10.1002/mma.4707}
}

@article {wang.xu:13:spiky,
    author = {Wang, Xuefeng and Xu, Qian},
     TITLE = {Spiky and transition layer steady states of chemotaxis systems
              via global bifurcation and {H}elly's compactness theorem},
   journal = {J. Math. Biol.},
  fjournal = {Journal of Mathematical Biology},
    volume = {66},
      year = {2013},
    number = {6},
     pages = {1241--1266},
      issn = {0303-6812,1432-1416},
   mrclass = {92C17 (35B32 35B41 35K51 35Q92)},
  mrnumber = {3040975},
       doi = {10.1007/s00285-012-0533-x},
       url = {https://doi.org/10.1007/s00285-012-0533-x},
}

@article {wang.xu:22:radial,
    author = {Wang, Zhi-An and Xu, Xin},
     title = {Radial spiky steady states of a flux-limited {K}eller-{S}egel
              model: existence, asymptotics, and stability},
   journal = {Stud. Appl. Math.},
  fjournal = {Studies in Applied Mathematics},
    volume = {148},
      year = {2022},
    number = {3},
     pages = {1251--1273},
      issn = {0022-2526,1467-9590},
   mrclass = {92C17 (35Q92)},
  mrnumber = {4429085},
       doi = {10.1111/sapm.12474},
       url = {https://doi.org/10.1111/sapm.12474},
}

@Article{	  wang.yan.ea:16:qualitative,
  title		= {Qualitative analysis of stationary {K}eller--{S}egel
		  chemotaxis models with logistic growth},
  author	= {Qi Wang and Jingda Yan and Chunyi Gai},
  year		= {2016},
  journal	= {Z. Angew. Math. Phys.},
  volume	= {67},
  number	= {3},
  pages		= {Paper No. 51, 25},
  doi		= {10.1007/s00033-016-0648-9},
  eprint	= {1312.0258},
  archiveprefix	= {arXiv},
  primaryclass	= {math.AP},
  url		= {https://arxiv.org/abs/1312.0258}
}

@Article{	  winkler:10:absence,
  author	= {Winkler, Michael},
  title		= {Absence of collapse in a parabolic chemotaxis system with
		  signal-dependent sensitivity},
  journal	= {Math. Nachr.},
  fjournal	= {Mathematische Nachrichten},
  volume	= {283},
  year		= {2010},
  number	= {11},
  pages		= {1664--1673},
  issn		= {0025-584X},
  mrclass	= {35K51 (35A01 35A02 35B45 35Q92 92C17)},
  mrnumber	= {2759803},
  mrreviewer	= {Narcisa C. Apreutesei},
  doi		= {10.1002/mana.200810838},
  url		= {https://doi.org/10.1002/mana.200810838}
}

@Article{	  xiang:19:dynamics,
  author	= {Xiang, Tian},
  title		= {Dynamics in a parabolic-elliptic chemotaxis system with
		  growth source and nonlinear secretion},
  journal	= {Commun. Pure Appl. Anal.},
  fjournal	= {Communications on Pure and Applied Analysis},
  volume	= {18},
  year		= {2019},
  number	= {1},
  pages		= {255--284},
  issn		= {1534-0392},
  mrclass	= {35M33 (35K57 35Q92 92C17)},
  mrnumber	= {3845565},
  doi		= {10.3934/cpaa.2019014},
  url		= {https://doi.org/10.3934/cpaa.2019014}
}

@article {xu:21:existence,
    author = {Xu, Xin},
     title = {Existence of spiky steady state of chemotaxis models with
              logarithm sensitivity},
   journal = {Math. Methods Appl. Sci.},
  fjournal = {Mathematical Methods in the Applied Sciences},
    volume = {44},
      year = {2021},
    number = {2},
     pages = {1484--1499},
      issn = {0170-4214,1099-1476},
   mrclass = {92C17 (35A01 35B40 35Q92)},
  mrnumber = {4185328},
       doi = {10.1002/mma.6846},
       url = {https://doi.org/10.1002/mma.6846},
}
\end{document}